\documentclass[12pt]{article}
\usepackage{amsmath}
\usepackage{amssymb}
\usepackage{amsthm}
\usepackage{mathrsfs}
\usepackage{comment}
\usepackage[capitalise]{cleveref}
\crefformat{equation}{(#2#1#3)}
\usepackage{enumerate}
\usepackage{graphicx}
\usepackage[margin=1.5in]{geometry}

\newtheorem{thm}{Theorem}
\numberwithin{thm}{section}
\newtheorem*{thm*}{Theorem}
\newtheorem{qst}[thm]{Question}
\newtheorem*{qst*}{Question}
\newtheorem*{oq*}{Open Question 3.10}
\newtheorem{prp}[thm]{Proposition}
\newtheorem{lma}[thm]{Lemma}
\newtheorem{cor}[thm]{Corollary}
\newtheorem{conj}[thm]{Conjecture}

\newcommand{\mo}{<_M}
\newcommand{\gmo}{>_M}

\newcommand{\wo}{\leq_S}
\newcommand{\swo}{<_S}

\newcommand{\slwo}{>_S}
\newcommand{\E}{\leq_E}
\newcommand{\sE}{<_E}
\newcommand{\sgE}{>_E}

\theoremstyle{definition}
\newtheorem{defn}[thm]{Definition}
\newtheorem{conv}[thm]{Convention}
\newtheorem{rmk}[thm]{Remark}
\newtheorem{example}[thm]{Example}
\begin{document}
\title{The Seed Order}
\author{Gabriel Goldberg}
\maketitle
\section{Introduction}
This paper is an exposition of the basic theory of the seed order, serving as a companion paper to \cite{MitchellOrder} and \cite{GCH}. The seed order is a definable relation on the class of uniform countably complete ultrafilters that consistently wellorders this class, even in the presence of very large cardinals. In fact, it seems likely that the linearity of the seed order is compatible with all large cardinals. If this is the case, then the statement that the seed order is linear, which we call the {\it Ultrapower Axiom}, is a candidate for a new axiom for set theory, with some striking consequences for the structure of the universe of sets, especially under strong large cardinal hypotheses.

Under the Ultrapower Axiom, the theory of countably complete ultrafilters is quite tractable, and the seed order is the main tool in its analysis. Under large cardinal hypotheses, particularly supercompactness, there are enough countably ultrafilters that one can leverage this tractability to derive consequences outside the realm of pure ultrafilter theory, for example in cardinal arithmetic and inner model theory (assuming the Ultrapower Axiom). We list some sample applications and then outline in detail the organization of this paper.

As a first example, we state what was chronologically the first consequence of the Ultrapower Axiom, the one that motivated its formulation.
\begin{thm}[Ultrapower Axiom]\label{NormalMitchell}
The Mitchell order linearly orders the class of normal ultrafilters.
\end{thm} 
This solves the open question, see \cite{Kanamori}, of whether the linearity of the Mitchell order is compatible with the existence of a cardinal \(\kappa\) that is \(2^\kappa\)-supercompact, since by \cite{FiniteLevels} and \cite{NeemanSteel}, the Ultrapower Axiom holds in an inner model satisfying this large cardinal hypothesis. (The construction of this model requires an iteration hypothesis for now. For a more complete discussion of this problem, see \cite{MitchellOrder}.) The theorem is proved by showing that the restriction of the seed order to normal ultrafilters is equal to the Mitchell order. A basic open problem is to characterize the uniform ultrafilters on which the Mitchell order and the seed order coincide. In fact, this question turns out to be the fundamental problem in the subject, and the attempt to understand it leads naturally to every theorem proved here, in \cite{MitchellOrder}, and in \cite{GCH}. In \cite{MitchellOrder} in particular, we prove much stronger results on the Mitchell order. The underlying fact seems to be the following:
\begin{thm}[Ultrapower Axiom]\label{DoddSolidMitchell}
The Mitchell order is linear on all Dodd solid ultrafilters.
\end{thm}
A Dodd solid ultrafilter is one whose ultrapower contains as much of its own extender as is possible. (See \cite{MitchellOrder} for the precise definition, which is not required for understanding this introduction.) Once again, this is shown by proving that the seed order and the Mitchell order coincide on Dodd solid ultrafilters. By applying Solovay's lemma, we obtain the following corollary:
\begin{thm}[Ultrapower Axiom]\label{SupercompactMitchell}
The Mitchell order is linear on all normal fine ultrafilters on \(P_\lambda(\lambda)\) for any regular cardinal \(\lambda = 2^{<\lambda}\).
\end{thm}
Indeed a natural variant of the Mitchell order is linear on all normal fine ultrafilters on \(P_\lambda(\lambda)\) for {\it any} cardinal \(\lambda = 2^{<\lambda}\); the singular case turns out to be quite interesting, involving a generalization of Solovay's lemma to singular cardinals. In this paper, we content ourselves with the linearity theorem for normal ultrafilters, \cref{NormalMitchell}. We also show here that if \(2^{<\lambda} = \lambda\), the seed order extends the Mitchell order on uniform ultrafilters on \(\lambda\). This theorem is part of the motivation for \cite{GCH}.

Second, assuming the Ultrapower Axiom, the class of countably complete ultrafilters satisfies a version of the Fundamental Theorem of Arithmetic: 
\begin{thm}[Ultrapower Axiom]\label{FactorizationTheorem} Every countably complete ultrafilter factors as a finite linear iteration of irreducible ultrafilters.\end{thm} 
An irreducible ultrafilter is an ultrafilter that admits no nontrivial factorization as an iterated ultrapower, see \cref{primitivity}. \cref{FactorizationTheorem} should be viewed as an abstract generalization of Kunen's theorem that every ultrafilter in \(L[U]\) is Rudin-Keisler equivalent to a finite iteration of the normal ultrafilter \(U\). Indeed, under the Ultrapower Axiom irreducibility is equivalent to a property called {\it predictiveness} that manifestly generalizes normality. Second, by a theorem of Schlutzenberg, in the Mitchell-Steel models below a superstrong cardinal, the irreducible ultrafilters are precisely the total ultrafilters indexed on the extender sequence, the exact analogs of the ultrafilter \(U\) in the context of \(L[U]\). In particular, in these models, every irreducible ultrafilter is Dodd solid, by a theorem of Steel, see \cite{Schimmerling}.

To what extent the Dodd solidity of irreducible ultrafilters persists to inner models for larger cardinals, and to what extent it is a consequence of the Ultrapower Axiom, is an open problem, intimately related to the relationship between the seed order and the Mitchell order by \cref{DoddSolidMitchell}. In any case, \cref{FactorizationTheorem} reduces the theory of ultrafilters under the Ultrapower Axiom to the analysis of the irreducible ultrafilters and their iterations. At least part of this analysis can be carried out assuming the Ultrapower Axiom by an induction through the large cardinal hierarchy. The first step of this induction is a direct analog of Kunen's theorem: under the Ultrapower Axiom, either there is a \(\mu\)-measurable cardinal or else every countably complete ultrafilter is Rudin-Keisler equivalent to a finite iteration of normal ultrafilters.

Our last two examples show that under large cardinal assumptions, the consequences of Ultrapower Axiom extend beyond ultrafilter theory to answer fundamental structural questions about the universe of sets. Moreover, the proofs make significant use of the theory of the seed order. For example, we will show the following by a surprisingly simple argument:

\begin{thm}[Ultrapower Axiom] Suppose \(\kappa\) is supercompact. Then \(V\) is a set-generic extension of \(\textnormal{HOD}\) by a partial order of size \(2^{2^\kappa}\). Moreover \(\kappa\) is supercompact in \(\textnormal{HOD}\).\end{thm}
In fact, under these assumptions, \(\textnormal{HOD}\) is a weak extender model for the supercompactness of \(\kappa\), in the sense of \cite{Woodin}.

We mention one last theorem, whose proof appears in \cite{GCH} and requires the Bounding Lemma of \cref{MinimalSection}, the most basic result in the deeper part of the general theory of the seed order.
\begin{thm}[Ultrapower Axiom]\label{SupercompactGCH} Suppose \(\kappa\) is supercompact. Then for every \(\lambda \geq \kappa\), \(2^\lambda = \lambda^+\).
\end{thm}
\cref{SupercompactGCH} also arises from the interplay between the seed order and the Mitchell order. In particular, when one attempts to remove the cardinal arithmetic assumption from \cref{SupercompactMitchell}, one instead is led to {\it prove} this assumption.  

Having provided a sampling of what the Ultrapower Axiom can do, let us say more precisely what the Ultrapower Axiom is. It is first of all an {\it axiom of comparison}, closely related to Woodin's axiom Weak Comparison. For the purposes of this introduction, if \(M_0\) and \(M_1\) are inner models and \(N\) is an inner model of both \(M_0\) and \(M_1\), let us call a pair \(\langle k_0,k_1\rangle\) a {\it comparison} of \(\langle M_0,M_1\rangle\) to \(N\) if \(k_0:M_0\to N\) and \(k_1: M_1\to N\) are elementary embeddings definable over \(M_0\) and \(M_1\) respectively. (Comparisons are closely related to the comparison process of inner model theory.) The Ultrapower Axiom states that if \(U_0\) and \(U_1\) are countably complete ultrafilters and \(j_{U_0}: V\to M_{U_0}\) and \(j_{U_1}:V\to M_{U_1}\) are their ultrapowers, then \(M_{U_0}\) and \(M_{U_1}\) admit a comparison \(\langle k_0,k_1\rangle\) to a common model \(N\) that makes the following diagram commute: 

\begin{center}
\includegraphics[scale=.7]{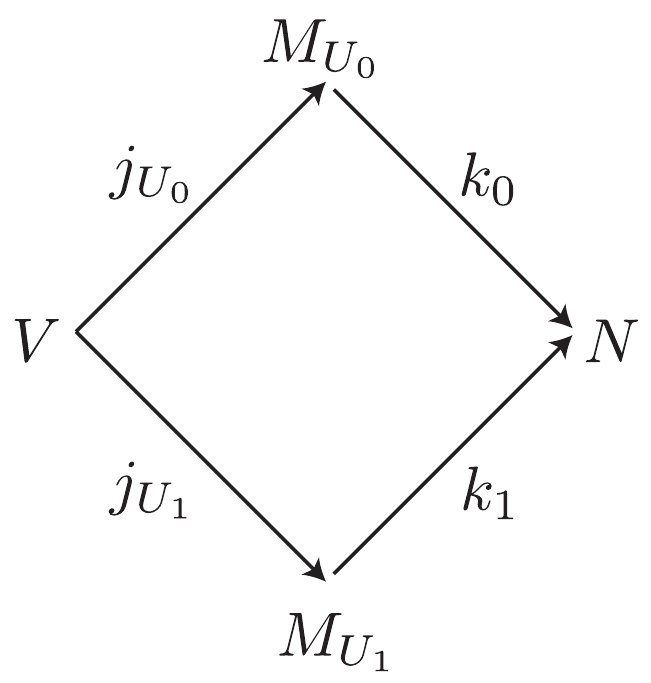}
\end{center}

Assuming the Ultrapower Axiom, we define the seed order on countably complete uniform ultrafilters \(U_0\) and \(U_1\) concentrating on ordinals by setting \(U_0 \wo U_1\) if for some comparison \(\langle k_0,k_1\rangle\) of the ultrapowers \(M_{U_0}\) and \(M_{U_1}\) to a common inner model \(N\), \(k_0(\alpha_0) \leq k_1(\alpha_1)\) where \(\alpha_0 = [\text{id}]_{U_0}\) and \(\alpha_1 = [\text{id}]_{U_1}\) are the {\it seeds} of \(U_0\) and \(U_1\) respectively. The first theorems of this paper state that if the Ultrapower Axiom holds then the relation \(\wo\) does not depend on the choice of comparison and in fact is a wellorder. 


We now describe the organization of the paper. In \cref{UltrapowerAxiomSection}, which is very short, we introduce the Ultrapower Axiom, the general comparison principle that will be the hypothesis of most of our theorems. In \cref{SeedOrderSection}, we introduce the seed order and establish its basic properties under the assumption of the Ultrapower Axiom. In \cref{CanonicalComparisonSection}, we show from the Ultrapower Axiom that any ultrapowers of \(V\) admit a unique minimal comparison, which we call their canonical comparison. Then in \cref{mitchellorder}, we explore the consequences of the Ultrapower Axiom for the Mitchell order, proving for example that the Mitchell order is linear on all Dodd solid ultrafilters. We also explore here the relationship between the seed order and the Mitchell order on countably complete ultrafilters that are not assumed to be normal or Dodd solid, proving under the Ultrapower Axiom and the additional assumption of GCH that the seed order in fact extends the Mitchell order.

Next, in \cref{MinimalSection}, we establish some basic facts about ultrapowers that are useful in our abstract comparison arguments. The high point is a proof of the pointwise minimality of definable embeddings, \cref{minult}. The minimality of definable embeddings suggests a version of the seed order, which we call the \(E\)-order, whose basic properties we put down in \cref{ZFCSection}. This order provably extends the seed order, and is equal to the seed order assuming the Ultrapower Axiom. Thus the assumption that the extended seed order is linear is a trivial consequence of the Ultrapower Axiom, and we prove here that in fact it is equivalent to the Ultrapower Axiom. We conclude this section by relating the \(E\)-order with the Rudin-Keisler order, the Lipschitz order, and the Mitchell order.

In \cref{IrreducibleSection}, we introduce the notion of an irreducible ultrafilter, and show that under the Ultrapower Axiom, every ultrafilter factors as an internal iteration of irreducible ultrafilters. Next, in \cref{RealmSection}, we provide an inductive analysis of the seed order that shows for example the following theorem, generalizing Kunen's theorem for \(L[U]\):
\begin{thm*}[Ultrapower Axiom] Any countably complete ultrafilter on a cardinal below the least \(\mu\)-measurable cardinal is equivalent to a finite iteration of normal ultrafilters.\end{thm*}
In fact our analysis goes far past \(\mu\)-measures and well into the Radin hierarchy. This leads to \cref{UltrapowerLatticeSection}, in which we attempt to study the decomposition theory of ultrafilters. Assuming the Ultrapower Axiom, we show that the category of ultrapowers (of \(V\)) with internal ultrapower embeddings forms a lattice, with joins given by the canonical comparison. Moreover, we show this lattice is finitely generated: 
\begin{thm*}[Ultrapower Axiom]
Suppose \(M\) is an ultrapower. There are at most finitely many ultrapowers \(N\) of which \(M\) is an internal ultrapower.
\end{thm*}
This answers the question (raised in \cite{KanamoriUltrafilters}) of whether the existence of a cardinal \(\kappa\) that is \(2^\kappa\)-supercompact implies that there is an ultrafilter with infinitely many Rudin-Frolik predecessors: it does not, again assuming an iteration hypothesis. In \cref{comparisonsection}, we introduce and motivate Woodin's comparison principle, Weak Comparison, and prove that it implies the Ultrapower Axiom assuming \(V = \text{HOD}\) and large cardinals. (A proper class of strong cardinals is more than enough.) This forms the basis of the claim that the Ultrapower Axiom holds in fine structure models. In \cref{IndependenceSection}, we show by a forcing argument that the Ultrapower Axiom does not follow from the linearity of the Mitchell order on normal ultrafilters. Finally, \cref{QuestionSection} contains some open questions.

\section{The Ultrapower Axiom}\label{UltrapowerAxiomSection}
\begin{defn}\label{comparison}
Suppose \(U_0\) and \(U_1\) are countably complete ultrafilters. A {\it comparison} of \(\langle U_0,U_1\rangle\) {\it by internal ultrafilters} is a pair \(\langle W_0,W_1\rangle\) such that 
\begin{enumerate}[(1)]
\item \(W_0\in \textnormal{Ult}(V,U_0)\)
\item \(W_1\in \textnormal{Ult}(V,U_1)\)
\item \(\textnormal{Ult}(V,U_0)\vDash W_0\) is a countably complete ultrafilter
\item \(\textnormal{Ult}(V,U_1)\vDash W_1\) is a countably complete ultrafilter
\item \(\textnormal{Ult}(\textnormal{Ult}(V,U_0),W_0) = \textnormal{Ult}(\textnormal{Ult}(V,U_1),W_1)\)
\item \(j_{W_0}\circ j_{U_0} = j_{W_1}\circ j_{U_1}\)
\end{enumerate}
\end{defn}
\begin{figure}
\begin{center}\includegraphics[scale=.8]{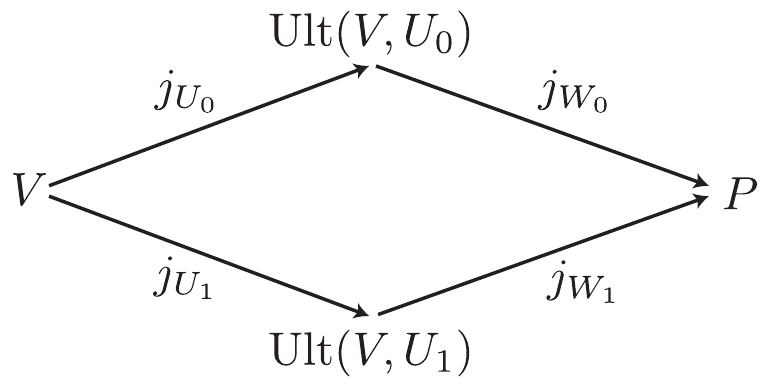}\end{center}
\caption{A comparison \(\langle W_0,W_1\rangle\) of \(\langle U_0,U_1\rangle\) to a common model \(P\).}
\end{figure}

\begin{example}
Suppose \(\kappa\) is a measurable cardinal and \(U_0\) and \(U_1\) are \(\kappa\)-complete ultrafilters on \(\kappa\) with \(U_0 \gmo U_1\). Then \(\langle U_1, j_{U_1}(U_0)\rangle\) is a comparison of \(\langle U_0, U_1\rangle\) by internal ultrafilters.
\end{example}

We note the following fact which implies that the commutativity requirement (6) in \cref{comparison} can be dropped if one assumes \(V = \text{HOD}\). We only sketch the proof. See also \cref{minult}.

\begin{lma}\label{defemb}
Any two elementary embeddings of \(V\) into the same model \(M\) that are definable from parameters agree on the ordinals.
\begin{proof}
Suppose the lemma is false. Fix a number \(n\) such that there are \(\Sigma_n\)-definable counterexamples. Suppose \(\alpha\) is the least ordinal such that there exist \(\Sigma_n\)-definable elementary embeddings \(j_0\) and \(j_1\) from \(V\) into the same model \(M\) that disagree on \(\alpha\). Then \(\alpha\) is definable {\it without parameters}; this is easy but tedious to check. It follows that any elementary embeddings from \(V\) to the same model \(N\) move \(\alpha\) to the same ordinal, namely the ordinal defined over \(N\) using the definition by which \(\alpha\) is defined over \(V\). But this contradicts that there exist exist elementary embeddings \(j_0\) and \(j_1\) from \(V\) into the same model \(M\) such that \(j_0(\alpha) \neq j_1(\alpha)\).
\end{proof}
\end{lma}

\begin{defn}\label{UltrapowerAxiom}
The following statement is called {\it the Ultrapower Axiom}: every pair of countably complete ultrafilters admits a comparison by internal ultrafilters.
\end{defn}

We introduce some terminological variations of \cref{comparison} that will be useful. Suppose \(\langle W_0,W_1\rangle\) is a comparison of \(\langle U_0,U_1\rangle\) by internal ultrafilters. We will also say that the {\it embeddings} \(\langle j_{W_0},j_{W_1}\rangle\) are a comparison of \(\langle U_0, U_1\rangle\), or that they are a comparison of the {\it models} \(\langle \textnormal{Ult}(V,U_0),\textnormal{Ult}(V,U_1)\rangle\) {\it to a common model} \(P\), which means that \(P = \textnormal{Ult}(\textnormal{Ult}(V,U_0),W_0)\).

Finally, as a justification of the abuses of notation listed above, we mention a consequence of the Ultrapower Axiom related to \cref{defemb}: assuming the Ultrapower Axiom, if \(M\) is an ultrapower of \(V\), there is a unique ultrapower embedding of \(V\) into \(M\). Thus under the Ultrapower Axiom, we would not lose information if we simply called the common model \(P\) a comparison of \(M_0\) and \(M_1\)! This is the content of \cref{everythingcommutes}, which is not a deep fact, but whose proof we delay until we have some more terminology.  

We state without proof an obvious combinatorial equivalent of the notion of a comparison.

\begin{defn}\label{iterating}
Suppose \(U\) is a countably complete ultrafilter on \(X\) and \(W\) is a countably ultrafilter of \(\textnormal{Ult}(V,U)\). Let \(f:X \to V\) be such that \([f]_{U} = W\), and assume without loss of generality that for each \(x\in X\), \(f(x)\) is an ultrafilter on a set \(Y_x\). Let \(Z = \{(x,y) : y\in Y_x\}\). Then \((U,W)\) denotes the ultrafilter on \(Z\) containing all \(A\subseteq Z\) such that \(\{x \in X : \{y\in Y_x : (x,y)\in A\}\in f(x)\}\in U\).
\end{defn}

We note that \((U,W)\) depends in a trivial way on the choice of \(f\). Distinct choices result in ultrafilters that are equal when restricted to a common measure one set. In fact, \((U,W)\) is the ultrafilter derived from \(j_W\circ j_U\) using \((j_W([\text{id}]_U), [\text{id}]_W)\), see \cref{uniform} and the comments following it.

\begin{prp}\label{sigma2comp}
Suppose \(U_0\) and \(U_1\) are countably complete ultrafilters. Then \(\langle W_0,W_1\rangle\) is a comparison of \(\langle U_0,U_1\rangle\) if and only if \((U_0, W_0)\equiv_{\textnormal{RK}} (U_1, W_1)\). In particular, the existence of a comparison of \(\langle U_0,U_1\rangle\) is a \(\Sigma_2\) statement about \(\langle U_0,U_1\rangle\).
\end{prp}

We close this section on a much less trivial note, stating the consistency result that makes the theory developed in this paper interesting.

\begin{thm}[Woodin]\label{finlev}
Assume the \(\mathcal E\)-Iteration Hypothesis. There is an inner model \(M\) with the following properties.
\begin{enumerate}[(1)]
\item  \(M\) satisfies \textnormal{GCH} and the Ultrapower Axiom.
\item Fix \(n < \omega\). Suppose there is a cardinal \(\kappa\) that is \(\beth_n(\kappa)\)-supercompact in \(V\). Then in \(M\) there is a cardinal \(\kappa\) that is \((\kappa^{+n})^M\)-supercompact.
\end{enumerate}
\end{thm}

The models built in \cite{FiniteLevels} and \cite{NeemanSteel} satisfy {\it much} stronger forms of comparison than the Ultrapower Axiom, but the Ultrapower Axiom alone becomes very interesting even just at a cardinal \(\kappa\) that is \(2^\kappa\)-supercompact (or just \(\kappa\)-compact), since there are so many \(\kappa\)-complete ultrafilters on \(\kappa\): every \(\kappa\)-complete filter on \(\kappa\) extends to a \(\kappa\)-complete ultrafilter.

\section{The Seed Order}\label{SeedOrderSection}
In this section we define the key notion of this article: the seed order, a relation that, assuming the Ultrapower Axiom, {\it wellorders} the class of uniform ultrafilters and generalizes the Mitchell order in a sense that will become clearer in \cref{mitchellorder}. We first define the class of uniform ultrafilters. This definition is slightly nonstandard. First of all, we allow ultrafilters concentrating on sets of the form \([\kappa]^n\) where \(\kappa\) is an ordinal and \(n\) is a number. We also demand that uniform ultrafilters be countably complete in order to avoid repeating this hypothesis in every theorem. Finally, and most importantly, we do not require of a uniform ultrafilter \(U\) that for all \(A\in U\), \(|A| = \textsc{sp}(U)\), but instead that \(U\) contain no bounded sets. In particular, principal ultrafilters can be uniform.

\begin{defn}
An ultrafilter \(U\) is {\it uniform} if \(U\) is a countably complete ultrafilter on \([\kappa]^n\) for some ordinal \(\kappa\) and number \(n\), and \(\kappa\) is the least ordinal such that \([\kappa]^n\in U\). The ordinal \(\kappa\) is the {\it space} of \(U\), denoted \(\textsc{sp}(U)\).
\end{defn}

Our focus on uniform ultrafilters throughout this paper is largely a matter of convenience. It {\it is} essential to the definition of the seed order that we consider ultrafilters on sets that are canonically wellordered, but uniformity is only used to ensure that a uniform ultrafilter \(U\) is uniquely determined by its ultrapower embedding \(j_U\) and its seed \([\text{id}]_U\). Without this property, for example, \cref{antisymmetric} would fail for trivial reasons.  

\begin{defn}
Suppose \(a\) and \(b\) are elements of \([\text{Ord}]^{<\omega}\). We say \(a < b\) if the largest element of \(a\cup b\) that is not an element of \(a\cap b\) lies in \(b\).
\end{defn} 
 
In other words, we view \([\text{Ord}]^{<\omega}\) as the collection of {\it descending} sequences of ordinals, and order it lexicographically. It is a standard fact that this is a wellorder. 
 
\begin{defn}[Seed Order]\label{seed} The {\it seed order} is a binary relation \(\wo\) on the class of uniform ultrafilters defined as follows: if \(U_0\) and \(U_1\) are uniform ultrafilters, then \(U_0\wo U_1\) if there exists a comparison \(\langle W_0, W_1\rangle\) of \(\langle U_0, U_1\rangle\) such that \(j_{W_0}([\text{id}]_{U_0}) \leq j_{W_1}([\text{id}]_{U_1})\).
\end{defn}

We in this case that the comparison \(\langle W_0,W_1\rangle\) {\it witnesses} that \(U_0\wo U_1\).

The most obvious connection between the Ultrapower Axiom and the seed order is worth pointing out:

\begin{prp}
The Ultrapower Axiom is equivalent to the statement that the seed order is total.
\end{prp}

(Note in the reverse direction that every countably complete ultrafilter is equivalent to a uniform ultrafilter by a crude application of the Axiom of Choice.) It is somewhat less obvious that, assuming the Ultrapower Axiom, the seed order is a nonstrict wellorder. We note that these theorems are improved in \cref{ZFCSection}. There we show, for example, that \cref{welldefined} and \cref{antisymmetric} are provable in ZFC. These improvements involve introducing a second order, the \(E\)-order, whose introduction is far from well-motivated at this point in the discussion. The key is that in fact \(\E\) extends \(\wo\), and the analogs of all the theorems of this section are provable for \(\E\) in ZFC alone. For now, we prefer to work with the Ultrapower Axiom, so that the proofs are somewhat simpler. 

\begin{thm}[Ultrapower Axiom]\label{transitive}
The seed order is transitive.
\begin{proof}
Suppose \(U_0\wo U_1\) and \(U_1\wo U_2\). The idea is to compare \(\langle U_0,U_1\rangle\) to a common model \(M\), compare \(\langle U_1,U_2\rangle\) to a common model \(N\), compare \(M\) and \(N\) to a common model \(Q\), and then to show that the composite comparison is a comparison of \(\langle U_0,U_2\rangle\) witnessing \(U_0\wo U_2\). It is not too hard to chase through \cref{transitivityfig} to see that this works, but we now give the details. 

Fix a comparison \(\langle W_0,W_1\rangle\) of \(\langle U_0,U_1\rangle\) to a common model \(M\) witnessing \(U_0 \wo U_1\), and a comparison \(\langle Z_1,Z_2\rangle\) of \(\langle U_1,U_2\rangle\) to a common model \(N\) witnessing that \(U_1\wo U_2\). Finally, applying the Ultrapower Axiom in \(\textnormal{Ult}(V,U_1)\), fix an \(\textnormal{Ult}(V,U_1)\)-comparison \(\langle F_M,F_N\rangle\) of \(\langle W_1,Z_1\rangle\) to a common model \(Q\).

\begin{figure}
\begin{center}\includegraphics[scale=.8]{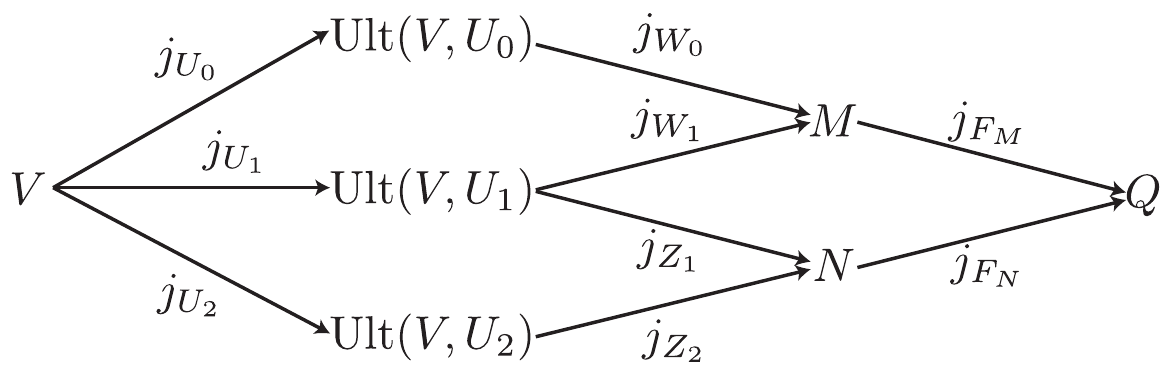}\end{center}
\caption{The transitivity of the seed order.}
\label{transitivityfig}
\end{figure}

Letting \(a_0 = [\text{id}]_{U_0}\), \(a_1 = [\text{id}]_{U_1}\), and \(a_2 = [\text{id}]_{U_2}\), \begin{equation}\label{transidentity} j_{F_M}(j_{W_0}(a_0)) \leq j_{F_M}(j_{W_1}(a_1)) = j_{F_N}(j_{Z_1}(a_1))\leq j_{F_N}(j_{Z_2}(a_2))\end{equation} The first inequality follows from \(\langle W_0,W_1\rangle\) witnessing that \(U_0\wo U_1\), and the last from \(\langle Z_1, Z_2\rangle\) witnessing \(U_1\wo U_2\). The inner equality follows from \cref{defemb} applied in \(\textnormal{Ult}(V,U_1)\), which implies \[j_{F_M}\circ j_{W_1} \restriction \text{Ord} = j_{F_N}\circ j_{Z_1}\restriction\text{Ord}\]

Finally, we note that the iterated ultrapower embedding \(j_{F_M}\circ j_{W_0}\) of \(\textnormal{Ult}(V,U_0)\) is the ultrapower of  \(\textnormal{Ult}(V,U_0)\) by the ultrafilter \((W_0,F_M)\), see \cref{iterating}. Similarly for \(j_{F_N}\circ j_{Z_2}\), \(\textnormal{Ult}(V,U_2)\), and \((W_1,F_N)\). Indeed \(\langle (W_0,F_M), (W_1,F_N)\rangle\) is a comparison of \(\langle U_0,U_1\rangle\): for this we must check the commutativity requirement of \cref{comparison}. This is routine:
\begin{align*}
j_{F_M}\circ j_{W_0}\circ j_{U_0} &= j_{F_M}\circ j_{W_1}\circ j_{U_1}
\\ &=  j_{F_N}\circ j_{Z_1}\circ j_{U_1}
\\ &=  j_{F_N}\circ j_{Z_2}\circ j_{U_2}
\end{align*}
By \cref{transidentity}, this comparison witnesses \(U_0\wo U_2\).
\end{proof}
\end{thm}

As a corollary of the proof of \cref{transitive} in the case \(U_2 = U_0\), we have the following useful fact, which says in a sense that the seed order is ``well-defined."
\begin{thm}[Ultrapower Axiom]\label{welldefined}
Suppose \(U_0\wo U_1\) and \(\langle Z_0,Z_1\rangle\) is a comparison of \(\langle U_0,U_1\rangle\). Then \(\langle Z_0,Z_1\rangle\) witnesses \(U_0\wo U_1\).
\begin{proof}
Suppose not, so there is a comparison \(\langle W_0,W_1\rangle\) witnessing \(U_0\wo U_1\), yet \(\langle Z_0, Z_1\rangle \) witnesses \(U_0\slwo U_1\). The proof of \cref{transitive} yields a single comparison that {\it simultaneously} witnesses that \(U_0\wo U_1\) and that \(U_0\slwo U_1\), which is a contradiction.
\end{proof}
\end{thm}

The next lemma is useful since it reduces the study of the seed order to the study of its restriction to ultrafilters on a fixed space. (We note that the analog of this theorem is not provable in ZFC for the Mitchell order on arbitrary ultrafilters.) 
\begin{lma}\label{spaces}
Suppose \(U_0\) and \(U_1\) are uniform ultrafilters with \(U_0\wo U_1\). Then \(\textsc{sp}(U_0)\leq \textsc{sp}(U_1)\).
\begin{proof}
Let \(k_0:\textnormal{Ult}(V,U_0)\to P\) and \(k_1:\textnormal{Ult}(V,U_1)\to P\) be embeddings witnessing \(U_0\wo U_1\). Let \(a_0 = [\text{id}]_{U_0}\) and \(a_1=[\text{id}]_{U_1}\). Then \[k_0(a_0) \leq k_1(a_1) \subseteq k_1(j_{U_1}(\textsc{sp}(U_1))) = k_0(j_{U_0}(\textsc{sp}(U_1)))\] It follows that \(k_0(a_0)\subseteq k_0(j_{U_0}(\textsc{sp}(U_1)))\), so \(a_0 \subseteq j_{U_0}(\textsc{sp}(U_1))\). This implies \(\textsc{sp}(U_0)\leq \textsc{sp}(U_1)\) since \(\textsc{sp}(U_0)\) is the least \(\kappa\) such that \(a_0\subseteq j_{U_0}(\kappa)\).
\end{proof}
\end{lma}

\begin{cor}\label{setlike}
The seed order is setlike.
\begin{proof}



This is an immediate corollary of \cref{spaces}.
\end{proof}
\end{cor}

\begin{thm}[Ultrapower Axiom]\label{antisymmetric}
The seed order is antisymmetric.
\begin{proof}
Suppose \(U_0\wo U_1\) and \(U_1\wo U_0\). By \cref{welldefined}, any comparison \(\langle W_0, W_1\rangle\) of \(\langle U_0,U_1\rangle\) witnesses simultaneously that \(U_0\wo U_1\) and that \(U_1\wo U_0\). In other words, \(j_{W_0}(a_0) = j_{W_1}(a_1)\) where we let \(a_0 = [\text{id}]_{U_0}\) and \(a_1 = [\text{id}]_{U_1}\).  

Note that \(\textsc{sp}(U_0) = \textsc{sp}(U_1)\) by \cref{spaces}; we denote this ordinal by \(\kappa\). We also have \(|a_0| =|a_1|\); we denote this number by \(n\). For any \(X\subseteq [\kappa]^n\),
\begin{align*}X\in U_0 &\iff a_0\in j_{U_0}(X)\\
&\iff j_{W_0}(a_0)\in j_{W_0}\circ j_{U_0}(X)\\
&\iff j_{W_1}(a_1)\in j_{W_1}\circ j_{U_1}(X)\\
&\iff a_1\in j_{U_1}(X)\\
&\iff X\in U_1
\end{align*}
For the third equality, we use that \(j_{W_0}\circ j_{U_0} = j_{W_1}\circ j_{U_1}\) (note that \cref{defemb} does not seem to suffice here) and the fact that \(j_{W_0}(a_0) = j_{W_1}(a_1)\). Hence \(U_0 = U_1\).
\end{proof}
\end{thm}

We remark that the proof of the next theorem bears a resemblance to the proof of the wellfoundedness of the mouse order assuming the Dodd-Jensen Lemma.
\begin{figure}
\begin{center}\includegraphics[scale=.8]{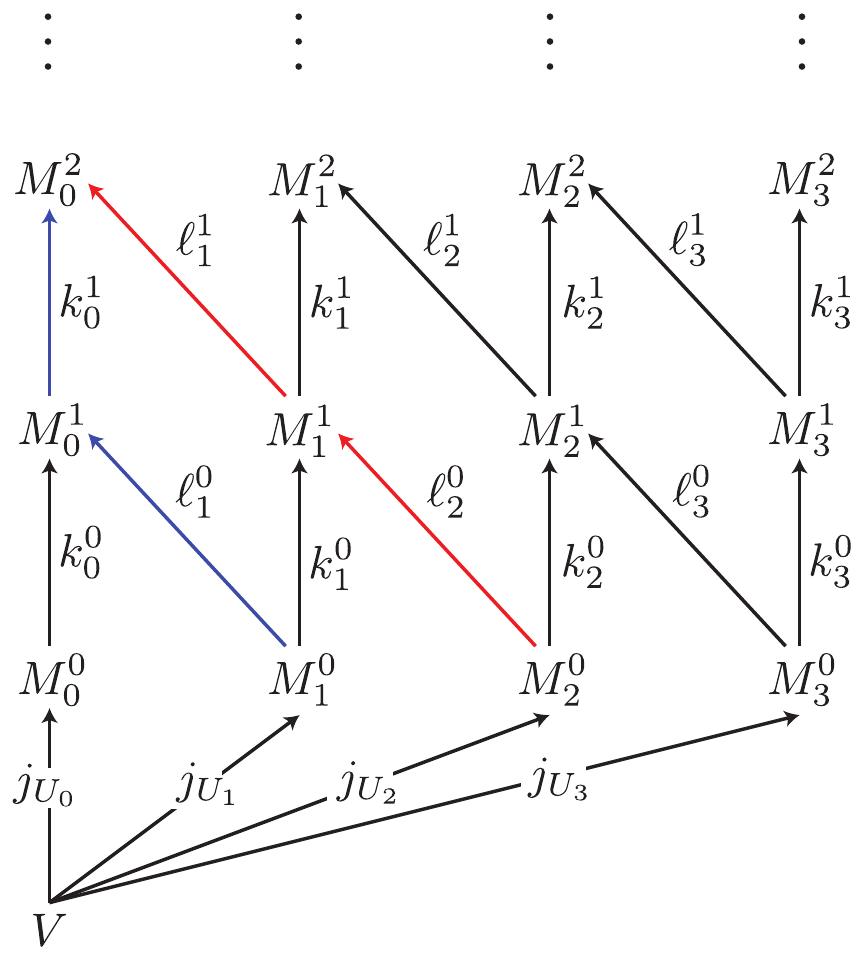}\end{center}
\caption{The wellfoundedness of the seed order, with \(j_1\) in blue and \(h_2\) in red.}
\end{figure}
\begin{thm}[Ultrapower Axiom]\label{wellfounded}
The seed order is wellfounded.
\begin{proof}
Suppose towards a contradiction that \[U_0\slwo U_1 \slwo U_2 \slwo\cdots\]
Let \(M^0_i\) denote \(\textnormal{Ult}(V,U_i)\). Fix ultrapower embeddings \[k^0_{0},k^0_1,k^0_{2},\dots,\ell^0_{1},\ell^0_2,\ell^0_{3},\dots\] and models \[M^1_0,M^1_1,M^1_2,\dots\] such that \(\langle k^0_i,\ell^0_{i+1}\rangle\) is a comparison of \(\langle M^0_i,M^0_{i+1}\rangle\) to the model \(M^1_i\). That is, \(k^0_{i}:M^0_i\to M^1_i\) and \(\ell^0_{i+1}:M^0_{i+1}\to M^1_i\) are internal ultrapower embeddings of \(M^0_i\) and \(M^0_{i+1}\) respectively. (See the remarks following \cref{UltrapowerAxiom} for the abuse of terminology we are  employing.) Next fix \[k^1_{0},k^1_1,k^1_{2},\dots,\ell^1_{1},\ell^1_2,\ell^1_{3},\dots\] and models \[M^2_0,M^2_1,M^2_2,\dots\] such that \(\langle k^1_{i},\ell^1_{i+1}\rangle\) are a comparison of \(\langle M^1_i,M^1_{i+1}\rangle\) to \(M^2_i\). Proceeding in this way, we obtain for each \(n < \omega\), models \(M^n_0,M^n_1,M^n_2,\dots\) and for each \(i< \omega\), internal ultrapower embeddings \begin{align*}k^n_i&: M^n_i\to M^{n+1}_i\\ 
\ell^n_{i+1}&: M^n_{i+1}\to M^{n+1}_i\end{align*}

For \(i < \omega\), define \(h_i: M^0_i\to M^i_0\) by \[h_i = \ell^{i-1}_1 \circ \cdots \circ \ell^2_{i-2}\circ \ell^1_{i-1} \circ \ell^0_i\] Setting \(j_i = k^{i}_{0} \circ h_{i}\), the embeddings 
\begin{align*}j_{i}&: M^0_i\to M^{i+1}_0\\ h_{i+1}&: M^0_{i+1}\to M^{i+1}_0\end{align*} constitute a comparison of \(M^0_i\) and \(M^0_{i+1}\), and so by \cref{welldefined} witness \(U_i\slwo U_{i+1}\).

Setting \(a_i = [\text{id}]_{U_i}\) for each \(i < \omega\), it follows that \[k^i_{0}(h_i(a_i)) = j_i(a_i) > h_{i+1}(a_{i+1})\] Thus the linear iteration \[V\to M^0_0 \to M^1_0 \to M^2_0\to \cdots\] of \(V\) by the internal ultrapower embeddings \(\langle k^i_0 : i < \omega\rangle\) has an illfounded direct limit. It is a standard fact that this is impossible, so we have obtained a contradiction.
\end{proof}
\end{thm}

\subsection{Application: Normal Ultrafilters}
We give a quick proof that the Ultrapower Axiom implies that normal ultrafilters are linearly ordered by the Mitchell order, so that the idea does not get lost in the generality of \cref{doddlin}.

\begin{thm}[Ultrapower Axiom]\label{mitchellnormalmeasures}
Suppose \(U_0\) and \(U_1\) are normal ultrafilters on a measurable cardinal \(\kappa\). Then \(U_0\) and \(U_1\) are Mitchell comparable.
\begin{proof}
Let \(\langle W_0,W_1\rangle\) be a comparison of \(\langle U_0,U_1\rangle\) witnessing without loss of generality that \(U_0\swo U_1\). Since \(U_0\) is normal, \(X\in U_0\) if and only if \(\kappa\in j_{U_0}(X)\). Applying \(j_{W_0}\), \(X\in U_0\) if and only if \(j_{W_0}(\kappa)\in j_{W_0}(j_{U_0}(X))\). By the commutativity requirement of \cref{comparison}, \(j_{W_0}(j_{U_0}(X)) = j_{W_1}(j_{U_1}(X))\). So \[X\in U_0\iff j_{W_0}(\kappa)\in j_{W_1}(j_{U_1}(X))\]

Since \(\langle W_0,W_1\rangle\) witnesses \(U_0\swo U_1\), \(j_{W_0}(\kappa) < j_{W_1}(\kappa)\). It follows that \(j_{W_0}(\kappa)\in j_{W_1}(j_{U_1}(X))\) if and only if \(j_{W_0}(\kappa)\in j_{W_1}(j_{U_1}(X))\cap j_{W_1}(\kappa)\), but \[j_{W_1}(j_{U_1}(X))\cap j_{W_1}(\kappa) = j_{W_1}(j_{U_1}(X)\cap \kappa) = j_{W_1}(X)\]
We conclude that \(X\in U_0\) if and only if \(j_{W_0}(\kappa) \in j_{W_1}(X)\), and so since \(j_{W_1}\) is definable over \(\textnormal{Ult}(V,U_1)\), \(U_0\in \textnormal{Ult}(V,U_1)\).
\end{proof}
\end{thm}

From this and \cref{finlev} we can dispense with Open Question 3.10 assuming an iteration hypothesis.

\begin{cor}
Assume the \(\mathcal E\)-Iteration Hypothesis and that there is a cardinal \(\kappa\) that is \(2^\kappa\)-supercompact. Then there is an inner model \(M\) in which there is a cardinal \(\kappa\) that is \(2^\kappa\) supercompact but carries a unique normal ultrafilter of Mitchell order \(0\).
\end{cor}

\subsection{Application: HOD}
As an application of the basic facts about the seed order, we show that the Ultrapower Axiom implies that if there is a supercompact cardinal \(\kappa\), then \(V = \text{HOD}_A\) for a set \(A\subseteq \kappa\), and \(\kappa\) is supercompact in \(\text{HOD}\). The key is the following proposition, immediate from the fact that the seed order is a definable wellorder.

\begin{prp}[Ultrapower Axiom]\label{odults}
Every uniform ultrafilter is ordinal definable.
\begin{proof}
Suppose \(U\) is a uniform ultrafilter of seed rank \(\alpha\). Then \(U\) is the {\it unique} uniform ultrafilter of seed rank \(\alpha\), so \(U\) is ordinal definable.
\end{proof}
\end{prp}

\begin{thm}[Ultrapower Axiom]
Suppose \(\kappa\) is supercompact and \(A\subseteq \kappa\) is such that \(V_\kappa\subseteq \textnormal{HOD}_A\). Then \(V = \textnormal{HOD}_A\).
\begin{proof}
Recall that one can define a Laver function from any wellorder of \(V_\kappa\). In particular, if \(A\subseteq \kappa\) is such that \(V_\kappa\subseteq \text{HOD}_A\), then there is a function \(f:\kappa\to V_\kappa\) in \(\text{HOD}_A\) such that for every set \(x\), there is a uniform ultrafilter \(U\) such that \(x = j_U(f)(\kappa)\), and since \(U\) is \(\text{OD}\) by \cref{odults}, it follows that \(x\in \text{HOD}_A\). In other words \(V = \text{HOD}_A\).
\end{proof}
\end{thm}

\begin{cor}[Ultrapower Axiom]\label{hodgen}
Suppose \(\kappa\) is supercompact. Then \(V\) is a generic extension of \(\textnormal{HOD}\) for a forcing of size at most \(2^{2^\kappa}\).
\begin{proof}
This follows from Vopenka's theorem that every set of ordinals is generic over \(\textnormal{HOD}\).
\end{proof}
\end{cor}

\begin{cor}[Ultrapower Axiom]
Suppose \(\kappa\) is supercompact. Then \(\kappa\) is supercompact in \(\textnormal{HOD}\). In fact \(\textnormal{HOD}\) is a weak extender model for \(\kappa\) is supercompact.
\begin{proof}
We show that for arbitrarily large \(\lambda\), there is a normal fine \(\kappa\)-complete ultrafilter \(\mathcal U\) on \(P_\kappa(\lambda)\) such that \(\mathcal U\cap \text{HOD}\in \text{HOD}\) and \(\text{HOD}\cap P_\kappa(\lambda)\in \mathcal U\). For any normal fine \(\kappa\)-complete ultrafilter on \(P_\kappa(\lambda)\), \(\mathcal U\cap \text{HOD}\in \text{HOD}\), since \(\mathcal U\) is actually ordinal definable: \(\mathcal U\) is ordinal definable from any uniform ultrafilter \(U\equiv_\text{RK}\mathcal U\) (noting that \(\mathcal U\) is definable as the unique normal fine ultrafilter on \(P_\kappa(\lambda)\) that is Rudin-Keisler equivalent to \(U\)), and such an ultrafilter \(U\) is itself ordinal definable by \cref{odults}.

We now show that if \(\lambda\) is a sufficiently large regular cardinal, then any normal fine \(\kappa\)-complete ultrafilter \(\mathcal U\) on \(P_\kappa(\lambda)\) has \(\text{HOD}\cap P_\kappa(\lambda) \in \mathcal U\). For this it suffices to show that \(j_\mathcal U[\lambda]\in j_\mathcal U(\textnormal{HOD}) = \textnormal{HOD}^{M_\mathcal U}\). If \(\lambda\) is a sufficiently large regular cardinal then by \cref{hodgen}, \(\textnormal{HOD}\) is stationary correct at \(\lambda\). So fix a partition \(\langle S_\alpha : \alpha < \lambda\rangle\in \text{HOD}\) of \(\text{cof}(\omega)\cap \lambda\) into stationary sets. By Solovay's lemma, \(j_\mathcal U[\lambda]\) is definable in \(M_\mathcal U\) from \(j_\mathcal U(\langle S_\alpha : \alpha < \lambda\rangle)\) and \(\sup j_\mathcal U[\lambda]\). Since \(j_\mathcal U(\langle S_\alpha : \alpha < \lambda\rangle)\in  \textnormal{HOD}^{M_\mathcal U}\) it follows that \(j_\mathcal U[\lambda]\in \textnormal{HOD}^{M_\mathcal U}\) as desired.
\end{proof}
\end{cor}

\section{The Canonical Comparison}\label{CanonicalComparisonSection}
In this section we develop further the basic comparison theory for uniform ultrafilters. Assuming the Ultrapower Axiom, we show that any pair of ultrafilters has a least comparison in a precise sense, which we call the {\it canonical comparison}. (The meaning of ``least" is made more precise in \cref{UltrapowerLatticeSection}, where we prove the universal property of the canonical comparison, see \cref{universalproperty}.) We begin with two definitions. The first is fairly standard, and we will use it throughout the paper.

\begin{defn}\label{uniform}
Suppose \(j:V\to M\) is an elementary embedding and \(a\) is a finite set of ordinals. The (uniform) {\it ultrafilter derived from \(j\) using \(a\)} is the unique uniform ultrafilter \(U\) such that for all \(X\subseteq [\textsc{sp}(U)]^{|a|}\), \(X\in U\) if and only if \(a\in j(X)\). 
\end{defn}

Note that \(\textsc{sp}(U)\) is determined by \(a\) and the requirement that \(U\) be uniform. There is a more general notion of derived ultrafilter: there is really no need to assume \(a\) is a finite set of ordinals, though in the general case one must also fix a set \(A\) such that \(a\in j(A)\) to serve as the space of the derived ultrafilter. We will occasionally use this more general notion as well. The point of \cref{uniform} is to demand that a derived ultrafilter be uniform whenever this demand makes sense. We will also relativize this definition to inner models of \(V\) without comment as in the next definition.

\begin{defn}
Suppose \(\langle W_0,W_1\rangle\) is a comparison of \(\langle U_0,U_1\rangle\) by internal ultrafilters. We call \(\langle W_0,W_1\rangle\) {\it canonical} if the following hold:
\begin{enumerate}[(1)]
\item \(W_0\) is derived from \(j_{W_0}:\textnormal{Ult}(V,U_0)\to \textnormal{Ult}(\textnormal{Ult}(V,U_0),W_0)\) using \(j_{U_1}([\textnormal{id}]_{U_1})\)
\item \(W_1\) is derived from \(j_{W_1}:\textnormal{Ult}(V,U_1)\to\textnormal{Ult}(\textnormal{Ult}(V,U_1),W_1)\) using \(j_{U_0}([\textnormal{id}]_{U_0})\)
\end{enumerate}
\end{defn}

The first proposition of this section to some extent explains why the Ultrapower Axiom is a reasonable axiom according to inner model theoretic considerations. (See \cref{comparisonsection} for a better explanation.) We require the following definition.

\begin{defn}
Suppose \(M\) and \(N\) are transitive models of \(\textnormal{ZFC}\). A cofinal elementary embedding \(j:M\to N\) is a {\it close embedding}, or equivalently is {\it close to} \(M\), if for every \(a\in [\text{Ord}]^{<\omega}\cap N\), the \(M\)-ultrafilter derived from \(j\) using \(a\) is an element of \(M\).
\end{defn}

Close embeddings arise as branch embeddings of maximal nondropping iteration trees on fine structural models of ZFC, such as the iteration trees that appear in the process of comparison by least disagreement, see \cite{MitchellSteel}. The next theorem says roughly that a ``close comparison" of two ultrapowers can be converted into a comparison by internal ultrapowers. 

\begin{prp}\label{closetoultra}
Suppose \(M\) is a transitive model of \textnormal{ZFC} and \(U_0\) and \(U_1\) are countably complete ultrafilters of \(M\). Let \(M_0 = \textnormal{Ult}(M,U_0)\) and \(M_1 = \textnormal{Ult}(M,U_1)\), and suppose that for some model \(N\), there are close embeddings \begin{align*}k_0&:M_0\to N\\ k_1&: M_1\to N\end{align*} such that \(k_0\circ j_{U_0} = k_1\circ j_{U_1}\). Then in \(M\), \(\langle U_0,U_1\rangle\) admits a canonical comparison \(\langle W_0,W_1\rangle\) to a common model \(P\). Moreover, the model \(P\) itself embeds in \(N\) by an elementary embedding \(h: P\to N\) such that \(h \circ j_{W_0} = k_0\) and \(h\circ j_{W_1} = k_1\).
\begin{rmk}
We do not assume that \(k_0\) and \(k_1\) are amenable to \(M\).
\end{rmk}
\begin{proof}
Let \(a_0= [\textnormal{id}]_{U_0}\) and \(a_1 = [\textnormal{id}]_{U_1}\). Let \[\ell:M\to N\] denote the common embedding \(k_0\circ j_{U_0} = k_1\circ j_{U_1}\). Let 
\[X = \{\ell(f)(k_0(a_0)\cup k_1(a_1)) :f\in M\}\]
By the proof of Los's Theorem, \(X\prec N\). Note that \(k_0[M_0]\subseteq X\) because every element of \(M_0\) is of the form \(j_{U_0}(f)(a_0)\), and so every element of \(k_0[M_0]\) is of the form \[k_0(j_{U_0}(f)(a_0)) = k_0\circ j_{U_0}(f)(k_0(a_0)) = \ell(f)(k_0(a_0))\in X\] Similarly, \(k_1[M_1]\subseteq X\). 
Let \(W_0\) be the uniform \(M_0\)-ultrafilter derived from \(k_0\) using \(k_1(a_1)\) and let \(W_1\) be the uniform \(M_1\)-ultrafilter derived from \(k_1\) using \(k_0(a_0)\). Since \(k_0\) is close to \(M_0\), \(W_0\in M_0\), and since \(k_1\) is close to \(M_1\), \(W_1\in M_1\). We claim \(\langle W_0,W_1\rangle\) is a canonical comparison of \(\langle U_0,U_1\rangle\).

We define an elementary embedding \(h_0:\textnormal{Ult}(M_0,W_0)\to N\) whose range is \(X\). This is just the factor map of \(\textnormal{Ult}(M_0,W_0)\) into \(N\), defined as usual for \(g\in M_0\) by \[h_0(j_{W_0}(g)([\text{id}]_{W_0})) = k_0(g)(k_1(a_1))\]
Since \(M_0 = \{j_{U_0}(f)(a_0) : f\in M\}\), we have 
\begin{align*}\textnormal{ran}(h_0) &= \{k_0(j_{U_0}(f)(a_0))(k_1(a_1)) : f\in M\} \\
&= \{\ell(f)(k_0(a_0))(k_1(a_1)) : f\in M\} \\
&= \{\ell(f^*)(k_0(a_0)\cup k_1(a_1)) : f^*\in M\}\\
&=X\end{align*}

Similarly we define \(h_1:\textnormal{Ult}(M_1,W_1)\to N\) with range \(X\). It follows that \[\textnormal{Ult}(M_0,W_0) = \textnormal{Ult}(M_1,W_1)\] since each is isomorphic to the transitive collapse of \(X\). It also follows that \(h_0 = h_1\), since each is equal to the inverse of the transitive collapse of \(X\). It is clear from the definitions that \(h_0\circ j_{W_0} = k_0\) and \(h_1\circ j_{W_1} = k_1\). Denoting the common model  \(\textnormal{Ult}(M_0,W_0) = \textnormal{Ult}(M_1,W_1)\) by \(P\) and the common embedding \(h_0 = h_1\) by \(h\), it remains only to see that \(\langle W_0,W_1\rangle\) is a canonical comparison, which is not hard.

Note first that \(\langle W_0,W_1\rangle\) is a comparison: this amounts to checking the commutativity requirement (6) of \cref{comparison}, which holds since \[h_0\circ j_{W_0} \circ j_{U_0} = k_0\circ j_{U_0} = k_1\circ j_{U_1} = h_1\circ j_{W_0} \circ j_{U_0}\] and so since \(h_0 = h_1\), \(j_{W_0} \circ j_{U_0} =  j_{W_1} \circ j_{U_1}\). To see \(W_0\) is the \(M_0\)-ultrafilter derived from \(j_{W_0}\) using \(j_{W_1}(a_1)\), take \(X\in P([\textsc{sp}(W_0)]^{|a_1|})\cap M_0\):
\begin{align*}
X\in W_0&\iff k_1(a_1)\in k_0(X)\\
&\iff h_1(j_{W_1}(a_1))\in h_0(j_{W_0}(X))\\
&\iff h(j_{W_1}(a_1))\in h(j_{W_0}(X))\\
&\iff j_{W_1}(a_1)\in j_{W_0}(X)
\end{align*}
That \(W_1\) is the \(M_1\)-ultrafilter derived from \(j_{W_1}\) using \(j_{W_0}(a_0)\) is shown in exactly the same way. This completes the proof.
\end{proof}
\end{prp}

\begin{thm}[Ultrapower Axiom]\label{canonical}
Suppose \(U_0\) and \(U_1\) are uniform ultrafilters. There is a {\it unique} canonical comparison of \(\langle U_0,U_1\rangle\).
\begin{proof}
The existence of a canonical comparison follows from the Ultrapower Axiom and the special case of \cref{closetoultra} in which \(M = V\) and the embeddings \(k_0\) and \(k_1\) are ultrapower embeddings. It remains to prove uniqueness.

Set \(a_0 = [\textnormal{id}]_{U_0}\) and \(a_1=[\textnormal{id}]_{U_1}\). Suppose \(\langle W_0,W_1\rangle\) is a canonical comparison of \(\langle U_0,U_1\rangle\) to a common model \(P\) and \(\langle W_0',W_1'\rangle\) is a canonical comparison of \(\langle U_0,U_1\rangle\) to a common model \(P'\). We will show \(\langle W_0,W_1\rangle = \langle W_0',W_1'\rangle\).

Fix an \(\textnormal{Ult}(V,U_0)\)-comparison \(\langle Z,Z'\rangle\) of \(\langle W_0,W_0'\rangle\) to a common model \(N\). Then \(j_Z\circ j_{W_1}\) and \(j_{Z'}\circ j_{W_1'}\) agree on the ordinals by \cref{defemb} applied in \(\textnormal{Ult}(V,U_1)\). In particular, \(j_Z(j_{W_1}(a_1)) = j_{Z'}(j_{W_1'}(a_1))\). By canonicity, \(j_Z([\text{id}]_{W_0}) = j_{Z'}([\text{id}]_{W'_0})\). Applying \cref{antisymmetric} in \(\textnormal{Ult}(V,U_0)\), we obtain that \(W_0 = W_0'\).

By the same argument, \(W_1 = W_1'\). This completes the proof.
\end{proof}
\end{thm}

We note that assuming the Ultrapower Axiom, by the proof of \cref{canonical}, if \(U_0\) and \(U_1\) are countably complete ultrafilters and \(\langle W_0,W_1\rangle\) is a comparison of \(\langle U_0,U_1\rangle\), then \(W_0\) is certified to be the ultrafilter of the canonical comparison by the mere fact that \([\text{id}]_{W_0} = j_{W_1}([\text{id}]_{U_1})\); one does not need any assumptions about \(W_1\).

The following proposition, generalizing \cref{defemb}, essentially expresses an assumption built in to the Ultrapower Axiom, ultimately tracing back to the commutativity requirement (6) in the definition of a comparison, \cref{comparison}.

\begin{prp}[Ultrapower Axiom]\label{everythingcommutes}
If \(M\) is an ultrapower of \(V\), then there is a unique ultrapower embedding \(j:V\to M\).
\begin{proof}
Suppose \(U\) and \(U'\) are countably complete ultrafilters such that \(M = \textnormal{Ult}(V,U) = \textnormal{Ult}(V,U')\). We must show \(j_U = j_{U'}\). Using the Axiom of Choice, we may assume \(U\) and \(U'\) are uniform ultrafilters. Let \(\langle W,W'\rangle\) be the canonical comparison of \(\langle U,U'\rangle\). Note that \(j_W\restriction \text{Ord} = j_{W'}\restriction \text{Ord}\) by \cref{defemb} applied in \(M\). It follows that \(j_{W}([\text{id}]_U) \in \text{ran}(j_{W'})\), since \(j_{W}([\text{id}]_U) = j_{W'}([\text{id}]_U)\). Since \(W'\) is derived from \(j_{W'}\) using \(j_{W}([\text{id}]_U)\), by the definition of the canonical comparison, \(W'\) is principal. Thus \(j_{W'}\) is the identity. Similarly \(j_W\) is the identity. But \(j_W\circ j_U = j_{W'}\circ j_{U'}\) by the definition of a comparison, and so \(j_U = j_{U'}\).
\end{proof}
\end{prp}

\section{The Mitchell Order}\label{mitchellorder}
In this section, we will prove that assuming the Ultrapower Axiom, the Mitchell order is linear on certain kinds of ultrafilters. Since the definition of the Mitchell order is often given only for normal ultrafilters, we state the definition we are using (following \cite{Kanamori}):

\begin{defn}
Suppose \(U_0\) and \(U_1\) are countably complete ultrafilters. Then \(U_0 <_M U_1\) if \(U_0\in \textnormal{Ult}(V,U_1)\).
\end{defn}

An important observation is that if \(U_0\mo U_1\) then \(P(\textsc{sp}(U_0))\subseteq \textnormal{Ult}(V,U_1)\). This is a sense in which the Mitchell order is less general than the seed order. In fact, one might take some of the results here to suggest that perhaps the ``right" generalization of the Mitchell order to ultrafilters \(U\) such that \(\textsc{sp}(U) \neq \textsc{crt}(U)\) is the seed order, or more accurately the \(E\)-order. (At the very least, the basic theory of the \(E\)-order seems to admit a more natural development than that of the generalized Mitchell order.)

We do not assume \(U_0\) concentrates on ordinals or finite sets of ordinals here, since we will be interested in the Mitchell order for ultrafilters that do not concentrate on ordinals (in particular supercompactness measures). Nevertheless, this fully general definition has the feature that it is not invariant under Rudin-Keisler equivalence of ultrafilters. For a trivial example, suppose \(U\) is a normal ultrafilter on \(\kappa\) and \(U' = \{A\subseteq \theta : A\cap \kappa\in U\}\) for some \(\theta > \kappa\). Obviously \(U\) and \(U'\) are Rudin-Keisler equivalent, yet there may be ultrafilters \(W\) such that \(U\in \textnormal{Ult}(V,W)\) while \(U'\notin \textnormal{Ult}(V,W)\) simply because \(P(\theta)\nsubseteq \textnormal{Ult}(V,W)\). On the other hand, the Mitchell order does respect the variant of Rudin-Keisler equivalence that also demands the spaces of equivalent ultrafilters have the same cardinality.

We remark that the Mitchell order is not linear on arbitrary ultrafilters. This is obvious from the example above, since there can be no Mitchell relation between the equivalent ultrafilters \(U\) and \(U'\), nor between the ultrafilters \(U'\) and \(W\). But there are less trivial examples here. Suppose \(U_0\) is a countably complete ultrafilter on \(X\) and \(U_1\) is the ultrafilter obtained by iterating \(U_0\) twice. (Thus \(\textnormal{Ult}(V,U_1) = \textnormal{Ult}(\textnormal{Ult}(V,U_0),j_{U_0}(U_0))\); that is, \(U_1\) is the product of \(U_0\) with itself.) Then there can be no Mitchell relation between \(U_0\) and \(U_1\) since \(U_0\) and \(U_1\) can be computed from one another (given the power of \(X\)). It follows from this example that if there is a measurable cardinal then there are uniform ultrafilters that bear no Mitchell relation to one another. We will see later that the counterexamples in the last two paragraphs are essentially the only provable counterexamples to the linearity of \(<_M\) below a superstrong cardinal. See \cref{shortland}.

We begin by pointing out that many of the properties of the Mitchell order on normal ultrafilters carry over to all nonprincipal ultrafilters. This is part of the folklore, but does not seem to have appeared in print. We begin by stating a very well-known fact, whose proof appears in \cite{Kanamori}.

\begin{prp}\label{mostrict}
The Mitchell order on nonprincipal ultrafilters is strict.
\end{prp}

In order to prove the transitivity and wellfoundedness of the Mitchell order, we will prove a very coarse bound on the relative size of one ultrafilter compared to another ultrafilter lying above it in the Mitchell order. 
\begin{prp}\label{annoying}
Suppose \(U\) on \(X\) and \(W\) on \(Y\) are nonprincipal ultrafilters and \(U <_M W\). Then \[|U|^{\textnormal{Ult}(V,W)}< j_{W}(|Y|)\] In particular, \(U\in \textnormal{Ult}(H_{|Y|^+},W)\).
\begin{proof}

Let \(\langle \kappa_n :  n < \omega\rangle\) denote the critical sequence of \(j_{W}\). We have \(P(X)\in \textnormal{Ult}(V,W)\), and so for some \(n < \omega\), \[\kappa_n \leq |X| < \kappa_{n+1}\] since otherwise \(P(\sup\kappa_n)\in  \textnormal{Ult}(V,W)\), contradicting Kunen's inconsistency theorem. Fixing such an \(n < \omega\), we have \(P(\kappa_n)\in \textnormal{Ult}(V,W)\), since \(P(|X|)\in \textnormal{Ult}(V,W)\). It follows that \(\kappa_n\leq |Y|\): otherwise, assuming \(|Y| < \kappa_n\), since \(\kappa_n\) is inaccessible in \(V\) and \(P(\kappa_n)\in \textnormal{Ult}(V,W)\), this implies \(W\in \textnormal{Ult}(V,W)\) contradicting \cref{mostrict}. It follows that \(\kappa_{n+1}\leq j_{W}(|Y|)\). Since \(\kappa_{n+1}\) is inaccessible in \(\textnormal{Ult}(V,W)\), we have \[|U|^{\textnormal{Ult}(V,W)} = (2^{|X|})^{\textnormal{Ult}(V,W)}< \kappa_{n+1} \leq j_{W}(|Y|)\] as desired.
\end{proof}
\end{prp}

\begin{cor}\label{strongcor}
Suppose \(U\) and \(W\) are nonprincipal ultrafilters and \(M\) is a transitive model of \textnormal{ZFC}. If \(U<_M W\) and \(W\in M\), then \(U\in M\) and \(M\vDash U <_M W\). In particular, the Mitchell order on uniform ultrafilters is transitive.
\begin{proof}
This is essentially immediate from \cref{annoying}. Assume \(W\) is an ultrafilter on \(Y\). The point is just that since \(W\in M\), \(P(Y)\in M\) and hence \(H_{|Y|^+}\in M\) (and is equal to \((H_{|Y|^+})^M\)). It follows that \(\textnormal{Ult}(H_{|Y|^+}, W)\in M\) and so \(U\in M\) since \(U\in \textnormal{Ult}(H_{|Y|^+}, W)\) by \cref{annoying}. Moreover \[\textnormal{Ult}(H_{|Y|^+}, W)\subseteq \textnormal{Ult}(M,W)\] and so in fact \(U\in \textnormal{Ult}(M,W)\), which means \(M\vDash U <_M W\).
\end{proof}
\end{cor}

\begin{prp}\label{mowo}
The Mitchell order on nonprincipal ultrafilters is wellfounded.
\begin{proof}
Suppose not. Assume \(X_0\) is a set of least possible cardinality carrying an ultrafilter \(U_0\) below which the Mitchell order is illfounded. That is, there is a sequence of uniform ultrafilters \[U_0 >_M U_1 >_M U_2 >_M \cdots\] Suppose that \(U_1\) is an ultrafilter on \(X_1\).
By elementarity, in \(\textnormal{Ult}(V,U_0)\), \(j_{U_0}(X_0)\) is a set of least possible cardinality carrying a uniform ultrafilter below which the Mitchell order is illfounded. But by \cref{strongcor}, \(\textnormal{Ult}(V,U_0)\vDash U_1 >_M U_2 >_M \cdots\), and by \cref{annoying}, \(|X_1| < j_{U_0}(|X_0|)\). This is a contradiction.
\end{proof}
\end{prp}

Assuming the Ultrapower Axiom, the most general class of ultrafilters on which we can prove that the Mitchell order is linear is the class of Dodd solid ultrafilters. 

\begin{defn}\label{dodddef}
Suppose \(U\) is a uniform ultrafilter and \(a\in [\textnormal{Ord}]^{<\omega}\). The {\it extender of \(U\) below \(a\)} is the function \(E: P(\textsc{sp}(U)) \to V\) defined by 
\[E(X) = j_U(X) \cap \{b \in [\text{Ord}]^{<\omega} : b < a\}\] We will denote this extender by \(U|a\).
A nonprincipal uniform ultrafilter \(U\) is called {\it Dodd solid} if \(U| a\) is an element of \(\textnormal{Ult}(V,U)\) where \(a = [\text{id}]_U\).
\end{defn}

The following proposition shows one sense in which Dodd solid ultrafilters are related to supercompactness.

\begin{prp}\label{doddsoundsupercompact}
Suppose \(U\) is Dodd solid and \(\kappa = \textsc{sp}(U)\). Then \(\textnormal{Ult}(V,U)\) is closed under \(2^{<\kappa}\)-sequences.
\begin{proof}
Let \(S = \bigcup_{\alpha < \kappa} P(\alpha)\). Since the extender \(E = U| [\text{id}]_U\) is in \(\textnormal{Ult}(V,U)\) and \(\text{dom}(E) = P(\kappa)\), \(P(\kappa)\in \textnormal{Ult}(V,U)\), and hence \(S\in \textnormal{Ult}(V,U)\). In fact, \(j_U\restriction S\in \textnormal{Ult}(V,U)\), since for \(X\in S\), \(j_U(X) = E(X)\): since \(X\subseteq \alpha\) for some \(\alpha < \textsc{sp}(U)\) and \(\textsc{sp}(U)\) is the least ordinal that \(j_U\) maps above \([\text{id}]_U\), \[j(X)\subseteq \{a \in [\text{Ord}]^{<\omega} : a <[\text{id}]_U\}\] Thus \(j_U\restriction S\in \textnormal{Ult}(V,U)\). Now \(\textnormal{Ult}(V,U)\) is closed under \(S\)-sequences by the usual argument, and hence under \(2^{<\kappa}\)-sequences since \(|S| = 2^{<\kappa}\).
\end{proof}
\end{prp}

There is therefore a bit of GCH implicit in the definition of Dodd solidity.

\begin{cor}\label{doddgch}
Suppose \(U\) is Dodd solid and \(\kappa = \textsc{sp}(U)\). Then \(2^{<\kappa} = \kappa\).
\begin{proof}
By \cref{doddsoundsupercompact}, \(\textnormal{Ult}(V,U)\) is closed under \(2^{<\kappa}\)-sequences. If \(2^{<\kappa} > \kappa\), then \(\textnormal{Ult}(V,U)\) is closed under \(\kappa^+\)-sequences and hence computes that \(j_U(\kappa^+)\) has cofinality \(\kappa^+\), noting that \(j_U\) is continuous at \(\kappa^+\) since \(\textsc{sp}(U) = \kappa\). Since \(j_U(\kappa^+)\) is regular in \(\textnormal{Ult}(V,U)\), it follows that \(j_U(\kappa^+) = \kappa^+\). On the other hand since \(\text{Ult}(V,U)\) is closed under \(\kappa^+\)-sequences, \(P(\kappa^+)\subseteq \text{Ult}(V,U)\), which contradicts Kunen's inconsistency theorem since \(\kappa^+\) is a fixed point of \(j_U\) above \(\textsc{crt}(j_U)\). 
\end{proof}
\end{cor}

One does not really need to cite Kunen's inconsistency theorem in the context of ultrapower embeddings, see \cite{Kanamori}. 

The following is the key consequence of the Ultrapower Axiom.

\begin{thm}[Ultrapower Axiom]\label{doddlin}
The Mitchell order wellorders the class of Dodd solid ultrafilters.
\end{thm}
In fact we prove the following strengthening of \cref{doddlin}.

\begin{thm}\label{doddlin*}
Suppose \(U_0\) is a uniform ultrafilter and \(U_1\) is a Dodd solid ultrafilter. If \(U_0 \swo U_1\), then \(U_0<_M U_1\).
\begin{proof}
Let \(\langle W_0, W_1\rangle\) be a comparison of \(\langle U_0, U_1\rangle\) to a common model \(M\) witnessing that \(U_0\swo U_1\). Let \(k_{0}: \textnormal{Ult}(V,U_0)\to M\) and \(k_1: \textnormal{Ult}(V,U_1)\to M\) be the ultrapower embeddings by \(W_0\) and \(W_1\) respectively, and let \(a_0 = [\text{id}]_{U_0}\) and \(a_1 = [\text{id}]_{U_1}\). By the definition of the seed order, \begin{equation}\label{pointy} k_0(a_0) < k_1(a_1)\end{equation}

For any \(X\subseteq [\textsc{sp}(U_0)]^{|a_0|}\), 
\begin{align}\label{nododd}
X\in U_0&\iff a_0\in j_{U_0}(X)\nonumber\\
&\iff k_0(a_0)\in k_0(j_{U_0}(X))\\
&\iff k_0(a_0)\in k_1(j_{U_1}(X))\nonumber
\end{align}

It follows from \cref{nododd} that \(U_0\) can be computed from the parameter \(k_0(a_0)\) and the classes \(k_1\) and \(j_{U_1}\). The parameter \(k_0(a_0)\) is in \(\text{Ult}(V,U_1)\), and the embedding \(k_1\) is definable over \(\text{Ult}(V,U_1)\) from the parameter \(W_1\). The key point is that because \(U_0\swo U_1\), one only needs a fragment of \(j_{U_1}\) to define \(U_0\) as in \cref{nododd}, and Dodd solidity implies that this fragment is in \(\textnormal{Ult}(V,U_1)\). 

Indeed, let \(E = U_1 | a_1\) be the extender of \(U_1\) below \(a_1\). We calculate:
\begin{align*}k_0(a_0)\in k_1(j_{U_1}(X))&\iff k_0(a_0)\in k_1(j_{U_1}(X))\cap \{c : c < k_1(a_1)\} \\
&\iff k_0(a_0)\in k_1(j_{U_1}(X) \cap \{c : c < a_1\})\\
&\iff k_0(a_0)\in k_1(E(X))
\end{align*}  
The first equivalence follows from \cref{pointy}. Again, since \(k_0(a_0)\) is in \(\textnormal{Ult}(V,U_1)\), \(k_1\) is definable over \(\textnormal{Ult}(V,U_1)\) from \(W_1\), and \(E\) is in \(\text{Ult}(V,U_1)\) by the Dodd solidity of \(U_1\), \(U_0\) is in \(\textnormal{Ult}(V,U_1)\), as desired.
\end{proof}
\end{thm}

Note that \cref{doddlin*} implies that the Mitchell order and the seed order agree on Dodd solid ultrafilters. It is natural to ask whether, assuming the Ultrapower Axiom, the seed order actually {\it extends} the Mitchell order everywhere. In fact this is not true in complete generality for a trivial reason. Suppose \(U\) is a nonprincipal \(\kappa\)-complete ultrafilter on \(\kappa\) and \(W\) is the principal ultrafilter on \(\kappa+1\) concentrating on \(\{\kappa\}\). Then \(U \swo W\) by \cref{spaces}. Hence \(U\not\wo W\), but clearly \(W \mo U\). (Of course, \(U\mo W\) as well, which is part of the reason we restricted the lemmas regarding the Mitchell order to nonprincipal ultrafilters.) It is not clear whether there can be nontrivial counterexamples assuming the Ultrapower Axiom. Our next theorem says that there cannot be if one assumes in addition a bit of GCH, so this question is tied into the question of forcing the GCH to fail at a measurable cardinal while preserving the Ultrapower Axiom, see \cref{QuestionSection}.

We first note the following useful bound for canonical comparisons.
\begin{cor}\label{smallsp}
If \(\langle W_0,W_1\rangle\) is a canonical comparison of \(\langle U_0,U_1\rangle\) witnessing \(U_0\swo U_1\), then \(\{\textsc{sp}(W_1)\} \leq [\textnormal{id}]_{U_1}\) and in particular \(\textsc{sp}(W_1) < j_{U_1}(\textsc{sp}(U_1))\).
\begin{proof}
Note that \[\textsc{sp}(W_1) = \min \{\kappa : j_{W_0}([\textnormal{id}]_{U_0}) \subseteq j_{W_1}(\kappa)\}\] Since \(j_{W_0}([\textnormal{id}]_{U_0}) < j_{W_1}([\textnormal{id}]_{U_1})\), \(\{\textsc{sp}(W_1)\} \leq [\textnormal{id}]_{U_1}\), by simple properties of the canonical wellorder of \([\textnormal{Ord}]^{<\omega}\).
\end{proof}
\end{cor}

\begin{thm}[Ultrapower Axiom]\label{mitchextend0}
Suppose \(U_0\) is a uniform ultrafilter. Suppose \(U_1\) is a nonprincipal uniform ultrafilter such that \(U_0 <_M U_1\) and \(P\left(2^{<\textsc{sp}(U_0)}\right)\subseteq \textnormal{Ult}(V,U_1)\). Then \(U_0 \swo U_1\).
\begin{proof}
Let \(\delta = 2^{<\textsc{sp}(U_0)}\). We note that \(j_{U_0}\restriction H_{\delta^+}\) is in \(\textnormal{Ult}(V,U_1)\). Since \(P(\delta)\subseteq \textnormal{Ult}(V,U_1)\), we have \(H_{\delta^+}\subseteq \textnormal{Ult}(V,U_1)\). Working in \(\textnormal{Ult}(V,U_1)\), let \(j:H_{\delta^+}\to N\) be the ultrapower by \(U_0\). Then \(j = j_{U_0}|H_{\delta^+}\) and \(N = j_{U_0}(H_{\delta^+})\) (because \(H_{\delta^+}\) is closed under \(\textsc{sp}(U_0)\) sequences).

Suppose towards a contradiction that \(U_0>_S U_1\). Let \(\langle W_0, W_1\rangle\) be the canonical comparison of \(\langle U_0,U_1\rangle\). Thus \(W_0\) is derived from \(j_{W_0}\) using \(j_{W_1}([\text{id}]_{U_1})\) and \begin{align}\label{wbound}\textsc{sp}(W_0) < j_{U_0}(\textsc{sp}(U_0))\end{align} We claim \(W_0\in \textnormal{Ult}(V,U_1)\).  Note that by \cref{wbound} and elementarity, \[\textnormal{Ult}(V,U_0)\vDash |W_0| = 2^{\textsc{sp}(W_0)} \leq j_{U_0}(\delta)\] Hence \(W_0\in j_{U_0}(H_{\delta^+})  = N\). But \(N\subseteq \textnormal{Ult}(V,U_1)\), so \(W_0\in \textnormal{Ult}(V,U_1)\) as desired.

But as in \cref{doddlin}, \(U_1\) can be computed from \(W_0\) and \(j_{U_0}\restriction P(\textsc{sp}(U_0))\), and all these sets are in \(\textnormal{Ult}(V,U_1)\). We now give the details. Let \(a_0 = [\text{id}]_{U_0}\) and \(a_1 = [\text{id}]_{U_1}\). For any \(X\subseteq [\textsc{sp}(U_1)]^{|a_1|}\),
\begin{align*}
X\in U_1 &\iff a_1\in j_{U_1}(X)\\
&\iff k_1(a_1)\in j_{W_1}(j_{U_1}(X))\\
&\iff k_1(a_1)\in j_{W_0}(j_{U_0}(X))\\
&\iff j_{U_0}(X)\cap [\textsc{sp}(W_0)]^{|a_1|}\in W_0
\end{align*}
(The last equivalence follows from the fact that \(\langle W_0,W_1\rangle\) is a canonical comparison.) We have seen that \(j_{U_0}|H_{\delta^+}\in \textnormal{Ult}(V,U_1)\), and so since \[\textsc{sp}(U_1) \leq \textsc{sp}(U_0)\leq \delta\] \(j_{U_0}\restriction P(\textsc{sp}(U_1))\in \textnormal{Ult}(V,U_1)\). It follows that \(U_1\in \textnormal{Ult}(V,U_1)\), and this contradicts \cref{mostrict}.
\end{proof} 
\end{thm}

As a corollary we have the following:
\begin{thm}[Ultrapower Axiom + GCH]\label{mitchextend}
The seed order extends the Mitchell order on the class of uniform ultrafilters \(U\) with \(\textsc{sp}(U)\) a cardinal.
\begin{proof}
Suppose \(U_0 \mo U_1\) and \(\textsc{sp}(U_0)\) is a cardinal. Then \(P(\textsc{sp}(U_0))\in \textnormal{Ult}(V,U_1)\), and since \(\textsc{sp}(U_0)\) is a cardinal, \(\textsc{sp}(U_0) = 2^{<\textsc{sp}(U_0)}\) by GCH. Thus the hypotheses of \cref{mitchextend0} are satisfied, and \(U_0 \swo U_1\).
\end{proof}
\end{thm}

We note that the restriction to uniform ultrafilters whose spaces are cardinals is necessary to avoid the trivial counterexamples involving principal ultrafilters mentioned after the proof of \cref{doddlin}. The proof of \cref{mitchextend0} actually yields the following stronger result, which does not require the Ultrapower Axiom:
\begin{thm}\label{optmitchextend}
Suppose \(U_0\E U_1<_M U_2\) and \(P\left(2^{<\textsc{sp}(U_1)}\right)\subseteq \textnormal{Ult}(V,U_2)\). Then \(U_0 <_M U_2\).
\end{thm}

Here \(\E\) is the \(E\)-order, which we will introduce in \cref{ZFCSection}. We note that the theorem is {\it stronger} because it is stated using the \(E\)-order, and in particular implies the version of the theorem where the seed order is used instead. This statement takes into account the fact that we do not actually need both sides of the comparison involved in \cref{mitchextend} to be induced by internal ultrapowers.

\section{The Minimality of Definable Embeddings}\label{MinimalSection}
In this section, we prove some lemmas about ultrapower embeddings that will be of use in the inductive analysis of the seed order in \cref{RealmSection}. During this analysis, we will sometimes consider the restriction of the seed order to those ultrafilters \(U\) that constitute a minimal representative of their Rudin-Keisler equivalence class. We define the class of minimal ultrafilters now.

\begin{defn}\label{notationdef}
Suppose \(j: M \to N\) is an elementary embedding. A finite set of ordinals \(a\) generates \(N\) over \(j[M]\) if \(N = \{j(f)(a) : f\in M\}\).

A uniform ultrafilter \(U\) is {\it minimal} if \(U\) is the ultrafilter derived from \(j_U\) using the least sequence of ordinals that generates \(\textnormal{Ult}(V,U)\) over \(j_U[V]\). 
\end{defn}

We remark that (quite trivially) a minimal ultrafilter is the same thing as a uniform ultrafilter that is minimal in the seed order among all uniform ultrafilters giving rise to its ultrapower embedding. We are certainly not first to note the following two facts, which are attributed to Solovay in a slightly different form in \cite{Kanamori2}.

\begin{lma}[Rigid Ultrapowers Lemma]\label{rul}
Suppose \(j:V\to M\) is an ultrapower of \(V\). Then there is no nontrivial elementary embedding \[k:M\to M\] such that \(k\circ j = j\).
\begin{proof}
We show any such embedding \(k\) must be trivial. Let \(U\) be the minimal ultrafilter derived from \(j\). We claim \(k([\text{id}]_U) = [\text{id}]_U\). Suppose not. Then \(k([\text{id}]_U) > [\text{id}]_U\). Since \([\text{id}]_U\) generates \(M\) over \(j[V]\), there is a function \(f\in V\) such that \(k([\text{id}]_U) = j(f)([\text{id}]_U)\) Since \(k\circ j = j\), \(k(j(f)) = j(f)\). Thus \(M\) thinks that there is an \(a < k([\text{id}]_U)\) such that \(k(j(f))(a) = k([\text{id}]_U)\). By the elementarity of \(k\), there is an \(a < [\text{id}]_U\) such that \(j(f)(a) = [\text{id}]_U\). This contradicts the minimality of \(U\).

Thus \(k([\text{id}]_U) = [\text{id}]_U\). Along with our assumption that \(k\circ j = j\), this implies that \(k\) is surjective: \([\text{id}]_U\) generates \(M\) over \(j[V]\), and both \(j[V]\) and \([\text{id}]_U\) are contained in the range of \(k\). Since \(k\) is surjective, \(k\) is the identity.
\end{proof}
\end{lma}

As a corollary, we obtain the following obvious-sounding fact: for any countably complete ultrafilter \(U\), there is a unique \(a\in \textnormal{Ult}(V,U)\) such that \(U\) is the ultrafilter derived from from \(j_U\) using \(a\).

\begin{cor}
Assume \(U\) is a countably complete ultrafilter and \(a\in \textnormal{Ult}(V,U)\) is such that \(U\) is derived from its own ultrapower embedding \(j_U\) using \(a\). Then \(a = [\textnormal{id}]_U\).
\begin{proof}
Let \(M = \textnormal{Ult}(V,U)\). We have an elementary embedding \(k: M \to M\) with \(k([\textnormal{id}]_U) = a\), defined by \(k([f]_U) = j_U(f)(a)\). Note that \(k\circ j_U = j_U\). It follows from \cref{rul} that \(k\) is the identity, and in particular \([\textnormal{id}]_U = k([\text{id}]_U) = j_U(\text{id})(a) = a\).
\end{proof}
\end{cor}

We now prove a generalization of \cref{defemb} that shows in many situations that if \(M\) and \(N\) are transitive models of set theory and \(j:M\to N\) is an elementary embedding definable over \(M\) from parameters, then \(j\) is pointwise minimal on the ordinals among all elementary embeddings from \(M\) into \(N\). In particular, this is the case in the situation that \(M\) is an inner model of \(V\). For context, we point out the following open question, which is stated in the language of second order set theory, and is meant to be considered in the context of NBG with the Axiom of Choice plus large cardinals.

\begin{qst}
Suppose \(j_0\) and \(j_1\) are elementary embeddings of \(V\) into the same inner model. Must \(j_0\restriction \textnormal{Ord} = j_1\restriction \textnormal{Ord}\)?
\end{qst}

The answer is yes if \(j_0\) and \(j_1\) are definable by \cref{defemb}. More interestingly, a positive answer also follows from a strong form of Woodin's \(\text{HOD}\) Hypothesis. We sketch Woodin's proof of this here, since it has not been published.

\begin{defn}
The {\it Strong \textnormal{HOD} Hypothesis} is the statement that if \(\lambda\) is a strong limit singular cardinal of uncountable cofinality then \(\lambda^+\) is not \(\omega\)-strongly measurable in \(\text{HOD}\).
\end{defn}

\begin{thm}[Woodin]
Assume the Strong \(\textnormal{HOD}\) Hypothesis. Suppose \(j_0\) and \(j_1\) are elementary embeddings of \(V\) into the same inner model. Then \(j_0\restriction \textnormal{Ord} = j_1\restriction \textnormal{Ord}\).
\begin{proof}
Fix \(\delta\in \text{Ord}\) a common fixed point of \(j_0\) and \(j_1\). We will show \(j_0[\delta] = j_1[\delta]\). Fix a strong limit singular cardinal \(\lambda > \delta\) of uncountable cofinality that is also a common fixed point of \(j_0\) and \(j_1\). By the Strong \(\text{HOD}\) Hypothesis, \(\lambda^+\) is not \(\omega\)-strongly measurable in \(\text{HOD}\), and so since \(2^\delta < \lambda\), there is a \(<_{\text{HOD}}\)-least partition \(\mathcal S = \langle S_\alpha : \alpha < \delta\rangle\in \text{HOD}\) of \(\text{cof}(\omega)\cap \lambda^+\) into truly stationary sets. Let \(\mathcal T_0 = \langle T^0_\alpha : \alpha < j_0(\delta)\rangle = j_0(\mathcal S)\) and \(\mathcal T_1 = \langle T^1_\alpha : \alpha < j_1(\delta)\rangle = j_1(\mathcal S)\). 

We claim \(\alpha\in j_0[\delta]\) if and only if \(T^0_\alpha\) is truly stationary. This is just like the proof of Solovay's Lemma, so we omit the proof. Similarly \(\alpha\in j_1[\delta]\) if and only if \(T^1_\alpha\) is truly stationary. Note however that \(\mathcal T = j_0(\mathcal S) = j_1(\mathcal S) = \mathcal T'\) since \(\mathcal S\) is definable from common fixed points of the embeddings \(j_0\) and \(j_1\), which have the same target model. Thus \(j_0[\delta] = j_1[\delta]\).
\end{proof}
\end{thm}

For the proof of the minimality of definable embeddings, we need a simple iterability lemma for internally definable iterations of length \(\omega\), \cref{defwf}, which follows Kunen's proof of the wellfoundedness of iterated ultrapowers. First we introduce some notation for the definable embeddings in which we will be interested.

\begin{defn}
Suppose \(M\) is a transitive model of ZF, \(\alpha\in \text{Ord}\cap M\), and \(\ell \leq \omega\). Then a linear directed system \[\mathcal I = \langle M_m, j_{nm}:  n < m < \ell\rangle\] is a {\it \(\Sigma_k\) iteration of \(M\)} if \(M_0 = M\) and for all \(n\) with \(n + 1 < \ell\), \(M_n\) is a transitive model of ZF + DC and \(j_{n,n+1} : M_n\to M_{n+1}\) is an elementary embedding that is \(\Sigma_k\) definable over \(M_n\) from parameters. The {\it length} of \(\mathcal I\) is \(\ell\), and the objects associated with \(\mathcal I\) are denoted by \(M_n^\mathcal I\) and \(j_{nm}^\mathcal I\).

A sequence \(\vec x = \langle x_n : n + 1 < \ell\rangle\) is a \(k\)-code for \(\mathcal I\) if \(j_{n,n+1}\) is defined over \(M_n\) by the universal \(\Sigma_k\) formula using the parameter \(x_n\in M_n\). We denote the iteration coded by \(\vec x\) by \(\mathcal I_{\vec x}\), and we let \(M^{\vec x}_n = M^{\mathcal I_{\vec x}}_n\) and \(j^{\vec x}_{nm}= j^{\mathcal I_{\vec x}}_{nm}\). 

We say \(\vec x\) is \(M\)-bounded if \(\text{rank}(\vec x) \in \text{Ord}\cap M\), and we say a \(\Sigma_k\) iteration is \(M\)-bounded if it has an \(M\)-bounded \(k\)-code.
\end{defn}

Note that every \(\Sigma_k\) iteration has a \(k\)-code, but an infinite \(\Sigma_k\) iteration of \(M\) may not have a \(k\)-code in \(M\). The motivating example is an \(\omega\)-length iterated ultrapower, which is a \(\Sigma_2\) iteration of length \(\omega\). Such an iteration is \(M\)-bounded if and only if it is based on a rank initial segment of \(M\), a fact that we generalize in a lemma.

\begin{defn}
A \(k\)-code \(\vec x = \langle x_n : n + 1 < \ell\rangle\) is {\it based below} \(\alpha\in \text{Ord}\cap M^{\vec x}_0\) if for all \(n\) with \(n + 1 < \ell\), \(\textnormal{rank}(x_n) < j^{\vec x}_{0n}(\alpha)\). 
\end{defn}

The following lemma is fairly routine.

\begin{lma}
For any ordinal \(\alpha\) and number \(\ell \leq \omega\), the class \(C_\ell\) of \(k\)-codes for iterations of \(V\) of length \(\ell\) based below \(\alpha\) forms a set.
\begin{proof}
This is proved by induction on \(\ell < \omega\), noting that for any \(\vec y\in C_{\ell+1}\), \[\text{rank}(\vec y) \leq \sup \{ j^{\vec x}_{0\ell}(\alpha)+ 1:\vec x\in C_\ell\}\]
(Given that each \(C_\ell\) forms a set for \(\ell < \omega\), it is elementary to see that \(C_\omega\) forms a set.)
\end{proof}
\end{lma}

\begin{lma}\label{basedbounded}
Suppose \(\mathcal I\) is an iteration of \(M\) of length \(\omega\) that has a \(k\)-code \(\vec x\) based below an ordinal of \(M\). Then \(\mathcal I\) is \(M\)-bounded.
\begin{proof}
Fix \(\alpha \in \text{Ord}\cap M\) such that \(\vec x\) is based below \(\alpha\). By the previous lemma applied in \(M\), the class of \(k\)-codes in \(M\) for iterations of length \(\omega\) based on \(M\cap V_\alpha\) forms a set \(C\) in \(M\). Let \(\xi\) be the rank of \(C\). Since \(M\) contains every finite initial segment of \(\vec x\), \[\text{rank}(\vec x) = \sup_{\ell < \omega} \text{rank}(\vec x\restriction \ell) \leq \text{rank}(C)\]
Thus \(\vec x\) is \(M\)-bounded.
\end{proof}
\end{lma}

\begin{lma}\label{defwf}
Suppose \(M\) is a transitive model of \textnormal{ZF + DC}. Then the direct limit of any \(M\)-bounded \(\Sigma_k\) iteration of \(M\) of length \(\omega\) is wellfounded.
\begin{proof}
Note that if a transitive model \(N\) of ZF + DC thinks that there is no \(k\)-code for an \(\omega\)-length iteration whose direct limit is illfounded below the image of \(\eta\), then in \(V\) there is no such iteration that is \(N\)-bounded, since for each \(\xi < \text{Ord}\cap M\), \(N\) ranks the tree of attempts to build a \(k\)-code of rank \(\xi\) for such an iteration.

We now use Kunen's proof that iterated ultrapowers are wellfounded. Suppose the theorem fails. By the preceding paragraph there is some least ordinal \(\eta\in \text{Ord}\cap M\) such that \(M\) {\it contains} a \(k\)-code for a \(\Sigma_k\) iteration \(\langle M_n, j_{nm}: n\leq m < \omega\rangle\) whose direct limit \(M_\omega\) is illfounded below \(j_{0\omega}(\eta)\). This is first order over \(M\), so in each \(M_n\), for any \(\xi < j_{0n}(\eta)\), \(M_n\) thinks that there is no \(k\)-code for an iteration whose direct limit is illfounded below the image of \(\xi\). 

Fix a descending sequence \(e_0 > e_1 > e_2 > \cdots\) in \(M_\omega\) with \(e_0 < j_{0\omega}(\eta)\). Then fix \(n_0 < \omega\) such that for some \(\xi \in \text{Ord}\cap M_n\), \(j_{n_0\omega}(\xi) = e_0\). Clearly \(\xi < j_{0n_0}(\eta)\). But \(\mathcal I = \langle M_n, j_{nm}: n_0 \leq n\leq m < \omega\rangle\) is then an \(\omega\)-length iteration whose direct limit is illfounded below the image of \(\xi\). (This does not immediately contradict the minimality of \(j_{0n_0}(\eta)\), since \(M_{n_0}\) may not contain a \(k\)-code for \(\mathcal I\).) Since \(\text{Ord}\cap M_{n_0} = \text{Ord}\cap M\), \(\mathcal I\) is \(M_{n_0}\)-bounded, and so by the first paragraph we have a contradiction.
\end{proof}
\end{lma}

The following is one way that an elementary embedding \(i:M\to N\) can fail badly to be a class of \(M\), by which we mean that the Axiom of Replacement fails in the structure \((M,i)\).

\begin{defn}
Suppose \(M\) and \(N\) are transitive models of ZF with the same ordinals and \(i:M\to N\) is an elementary embedding. We say that \(i\) {\it iterates out of \(M\)} if for some \(\eta \in \text{Ord}\cap M\), \(\sup_{n<\omega} i^n(\eta) = \textnormal{Ord}\cap M\).
\end{defn}

An example of an amenable embedding that iterates out of a transitive model of ZFC is an \(I_3\)-embedding.

We next lemma should be interpreted to include the case \(\Omega = \text{Ord}\), in which case the following is formally a proposition in the second order language of set theory, proved (very easily) in NBG.

\begin{lma}
Suppose \(M\) and \(N\) are transitive models of \textnormal{ZF} of the same ordinal height \(\Omega\), and suppose \(\Omega\) has uncountable cofinality. Then no elementary embedding \(i:M\to N\) iterates out of \(M\).
\end{lma}

\begin{prp}\label{adhoc}
Suppose \(M\) and \(N\) are transitive models of \textnormal{ZF + DC}. Suppose \(j:M\to N\) is an elementary embedding definable over \(M\) from parameters.  Suppose \(i:M\to N\) is an elementary embedding that does not iterate out of \(M\). Then for any \(\xi\in \textnormal{Ord}\cap M\), \(j(\xi) \leq i(\xi)\).
\begin{proof}
Towards a contradiction fix \(\xi\in \text{Ord}\cap M\) such that \(i(\xi) < j(\xi)\). For \(n < \omega\), let \(M_n = i^n(M)\) and \(j_{n,n+1} = i^n(j)\). Thus \(j_{n,n+1}\) is an elementary embedding from \(i^n(M) = M_n\) to \(i^n(N) = i^n(i(M)) = M_{n+1}\).\begin{center}
\begin{figure}
\includegraphics[scale=1.05]{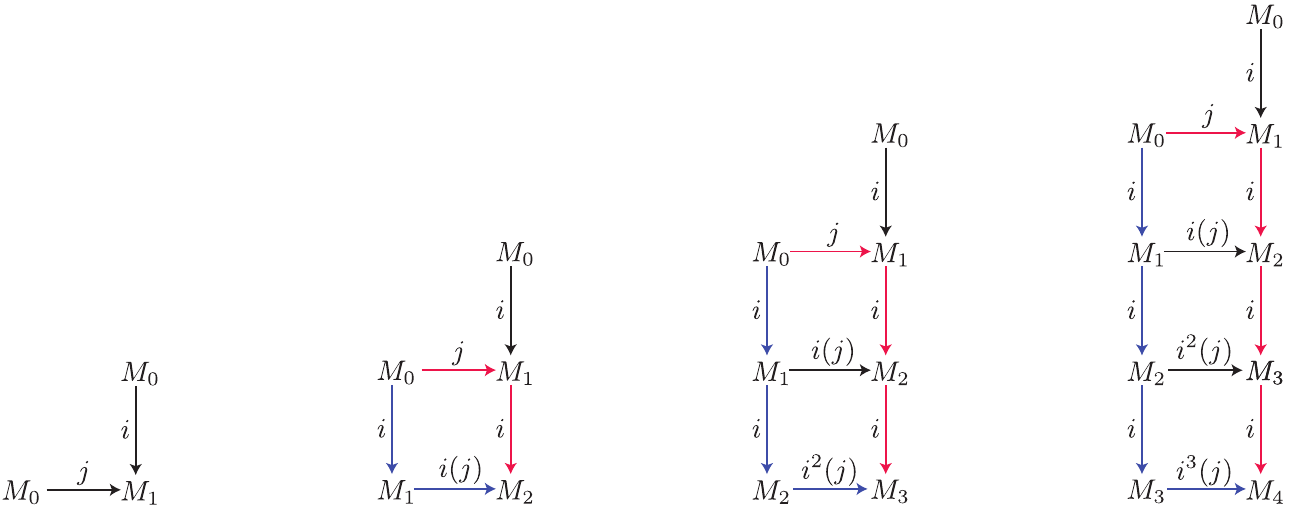}
\caption{Since \(j_{n,n+1}\circ i^n=i^n\circ j\) (shown in blue and red) and \(j(\xi) > i(\xi)\), we have \(j_{n,n+1}(i^n(\xi)) =i^n(j(\xi)) > i^{n+1}(\xi)\).}

\end{figure}
\end{center}
We claim that the direct limit \(M_\omega\) of the directed system \[\langle M_n,j_{nm}  : n \leq m < \omega\rangle\] generated by the embeddings \(j_{n,n+1}\) is illfounded. To see this, we claim the sequence \(\langle j_{n\omega}(i^n(\xi)) : n < \omega\rangle\) forms a descending sequence in \(M_\omega\). For this it suffices to see that \(i^{n+1}(\xi) < j_{n,n+1}(i^n(\xi))\). We have \[j_{n,n+1}(i^n(\xi)) = i^n(j)(i^n(\xi)) = i^n(j(\xi)) > i^{n}(i(\xi)) = i^{n+1}(\xi)\] Thus \(M_\omega\) is illfounded. To obtain a contradiction, it suffices by \cref{defwf} to verify that \(\langle M_n,j_{nm}  : n \leq m < \omega\rangle\) is \(M\)-bounded.

Suppose that \(j\) is defined over \(M\) by the \(\Sigma_k\) universal formula using the parameter \(x\in M\). Note that \(\vec x = \langle i^n(x) : n < \omega\rangle\) is a \(k\)-code for this iteration by construction, and \(\text{rank}(\vec x)\in\text{Ord}\cap M\) since \(\sup_{n<\omega} i^n(\eta)\in \textnormal{Ord}\cap M\) where \(\eta = \text{rank}(x)\). Thus \(\vec x\) is \(M\)-bounded.
\end{proof}
\end{prp}

As a corollary, we have the following theorem.

\begin{thm}\label{minult}
Suppose \(M\) and \(N\) are transitive models of \textnormal{ZF + DC} of ordinal height \(\Omega\) where \(\textnormal{cf}(\Omega) \neq \omega\). Suppose \(j:M\to N\) is a definable elementary embedding, and \(i:M\to N\) is an arbitrary elementary embedding. Then for any ordinal \(\alpha < \Omega\), \(j(\alpha) \leq i(\alpha)\).
\end{thm}

In particular, this theorem holds in the case \(\Omega = \text{Ord}\); that is, the theorem holds for \(M\) an arbitrary inner model and \(i\) an arbitrary elementary embedding. 

It is a bit strange that we need to assume here that \(i\) does not iterate out of \(M\), which appears to be a smallness assumption, in order to show that \(i\) is {\it above} a definable embedding. We do not know if this is necessary, but the following shows that one can still get an asymptotic result without the smallness assumption.

\begin{cor}
Suppose \(M\) and \(N\) are transitive models of \textnormal{ZF + DC}. Suppose \(j:M\to N\) is an elementary embedding definable over \(M\) from a parameter in \(M\cap V_{\xi_0}\). Suppose \(i:M\to N\) is an arbitrary elementary embedding. Then for all \(\xi\in \textnormal{Ord}\cap M\), if \(\xi \geq \xi_0\) then \(j(\xi) \leq i(\xi)\).
\begin{proof}
We assume towards a contradiction that \(j(\xi) > i(\xi)\) for some ordinal \(\xi \geq\xi_0\). From here we return to the proof of \cref{adhoc}, now allowing that \(i\) might iterate out of \(M\). We claim in this case that the \(k\)-code \(\vec x = \langle i^n(x): n<\omega\rangle\) from that proof is \(M\)-bounded. In \cref{adhoc}, we proved something a bit stronger than \(i^n(\xi) < j_{0,n}(\xi)\), using the assumption that \(i(\xi) < j(\xi)\). Since \(\text{rank}(i^n(x)) \leq i^n(\xi)\), it follows that \(\vec x\) is based below \(\xi\), and hence \(\vec x\) is \(M\)-bounded by \cref{basedbounded}. This contradicts the fact that \(M_\omega\) is illfounded, by \cref{defwf}.
\end{proof}
\end{cor}

\begin{defn}\label{fingen}
A model \(M\) of ZFC is {\it finitely generated over} \(\Gamma\) by a sequence of ordinals \(a\in [\textnormal{Ord}]^{<\omega} \cap M\) if every element of \(M\) is definable over \(M\) from parameters in \(\Gamma\cup a\).
\end{defn}

The following is a slight generalization of Woodin's Close Embeddings Lemma.

\begin{thm} \label{cel} Suppose \(M\) is finitely generated over \(\Gamma\) by \(a \in [\textnormal{Ord}]^{<\omega} \cap M\). Suppose \(j : M \to N\) is a close embedding and \(i : M \to N\) is an arbitrary elementary embedding such that \(j \restriction \Gamma = i  \restriction \Gamma\). Then \(j(a) \leq i(a)\).
\begin{proof} Towards a contradiction assume \(i(a) < j(a)\).
The key is to reduce to the case where \(j\) is an internal ultrapower embedding of \(M\). Let \(U\) be the ultrafilter derived from \(j\) using \(i(a)\). Since \(j\) is close, \(U \in M\). As usual, there is a factor embedding \(k : \textnormal{Ult}(M,U) \to N\) given by \(k([f]_U) = j(f)(i(a))\). It follows that
\[i[M] = H^N (i[\Gamma] \cup i(a)) = H^N (j[\Gamma] \cup i(a)) \subseteq k[\textnormal{Ult}(M, U)]\]
and so we can define an elementary embedding \(i_* : M \to \textnormal{Ult}(M,U)\) by \(i_* = k^{-1} \circ i\).

We claim \(i_*(a) < j_U (a)\). For this we show \(k(i_*(a)) < k(j_U (a))\). In fact, \(k(i_*(a)) = k(k^{-1} \circ i(a)) = i(a)\) and \(k(j_U (a)) = j(a)\), and so since we have assumed \(i(a) < j(a)\), we have \(k(i_*(a)) < k(j_U (a))\). We now replace \(j\) with \(j_U\) and \(i\) with \(i_*\) and \(N\) with \(\textnormal{Ult}(M,U)\). Thus we have reduced to the case where \(j\) is an internal ultrapower embedding of \(M\).

Note that \(\langle i^n(U) : n < \omega\rangle\) is an iteration of \(M\) that has an illfounded direct limit by the argument of \cref{minult}. But it is based on a rank initial segment of \(M\), for example, on \(M\cap V_\xi\) where \(\xi = \max (a) + \omega\). This is because \(\textsc{sp}(U) <\xi\) since \(j(a)> i(a)\), and so in general \(\textsc{sp}(i^n(U)) \leq i^n(\xi) \leq j_{0n}(\xi)\). The last inequality follows (by induction) from the fact that \(i^{n+1}(a) < j_{n,n+1}(i^n(a))\).
\end{proof}
\end{thm}

\subsection{Application: the Rudin-Keisler order}
The minimality of definable embeddings yields an alternative definition of the seed order, assuming the Ultrapower Axiom, which on the face of it looks a bit weaker. (This will be the basis for lifting the basic theory of the seed order to the context in which the Ultrapower Axiom is not assumed in \cref{ZFCSection}.)

\begin{thm}[Ultrapower Axiom]\label{closecomp}
Suppose \(U_0\) and \(U_1\) are uniform ultrafilters. Let \(M_0 = \textnormal{Ult}(V,U_0)\) and \(M_1 = \textnormal{Ult}(V,U_1)\). Suppose \(W\) is an ultrafilter of \(M_1\) and \(k:M_0\to \textnormal{Ult}(M_1,W)\) is an elementary embedding such that \begin{align}k([\textnormal{id}]_{U_0})\leq j_{W}([\textnormal{id}]_{U_1})\label{Ewit}\end{align} Then \(U_0\wo U_1\).
\begin{proof}
Fix a comparison \(\langle i_0,i_1\rangle\) of \(\langle U_0,(U_1,W)\rangle\) by internal ultrapowers to a common model \(P\). Thus \(i_0:\textnormal{Ult}(V,U_0)\to P\), \(i_1: N\to P\), and \begin{align}i_0\circ j_{U_0} = i_1\circ (j_W\circ j_{U_1})\label{commy}\end{align} Note that \(i_0\) and \(i_1\circ j_W\) are a comparison of \(U_0\) and \(U_1\) by internal ultrafilters: the commutativity requirement \(i_0\circ j_{U_0} = (i_1\circ j_W) \circ j_{U_1}\) is an immediate consequence of \cref{commy}. But \[i_0\restriction \textnormal{Ord}\leq i_1 \circ k\restriction \textnormal{Ord}\] by \cref{minult}, so \[i_0([\textnormal{id}]_{U_0})\leq i_1\circ k([\textnormal{id}]_{U_0}) \leq i_1\circ j_W([\textnormal{id}]_{U_1})\]
The second inequality uses \cref{Ewit}. Thus \(\langle i_0,i_1\circ j_W\rangle\) witnesses \(U_0\wo U_1\).
\end{proof}
\end{thm}

This yields the following simple fact which is not obvious from the definition of the seed order:

\begin{thm}[Ultrapower Axiom]\label{factorlemma}
Suppose \(U_0\) is a minimal ultrafilter, \(U_1\) is a uniform ultrafilter. Suppose \(k:\textnormal{Ult}(V,U_0)\to \textnormal{Ult}(V,U_1)\) is an elementary embedding such that \(k\circ j_{U_0} = j_{U_1}\).  Then \(U_0\wo U_1\). In particular, the seed order extends the Rudin-Keisler order on minimal ultrafilters.

\begin{proof}
By \cref{closecomp}, it suffices to show that \(k([\text{id}]_{U_0}) \leq [\text{id}]_{U_1}\). Let \(a_0 = [\text{id}]_{U_0}\) and \(a_1 = [\text{id}]_{U_1}\). Fix \(f \in V\) such that \(k(a_0) = j_{U_1}(f)(a_1)\). Since \(U_0\) is a minimal ultrafilter, \(a_0\neq j_{U_0}(f)(v)\) for any \(v < a_0\). By the elementarity of \(k\), \(k(a_0)\neq k(j_{U_0}(f))(v)\) for any \(v < k(a_0)\). By our commutativity assumption that \(k\circ j_{U_0} = j_{U_1}\), we have \(k(j_{U_0}(f)) = j_{U_1}(f)\). Thus \(k(a_0)\neq  j_{U_1}(f)(v)\) for any \(v < k(a_0)\), and so since \(k(a_0) = j_{U_1}(f)(a_1)\) it follows that \(a_1\not< k(a_0)\). In other words, \(k(a_0) \leq a_1\), as desired.
\end{proof}
\end{thm}

\subsection{Application: bounds on the canonical comparison}

Another application of the minimality of definable embeddings is the following theorem which classifies the canonical comparison of any pair of ultrafilters externally in terms of the seed orders of their ultrapowers.

\begin{thm}[Ultrapower Axiom]\label{canonicalclassification}
Suppose \(U_0\) and \(U_1\) are uniform ultrafilters. Let \(M_0 = \textnormal{Ult}(V,U_0)\) and \(M_1 = \textnormal{Ult}(V,U_1)\). Let \(\langle W_0,W_1\rangle\) be the canonical comparison of \(\langle U_0,U_1\rangle\). 

Then \(W_0\) is the least uniform ultrafilter \(W\) in the seed order of \(M_0\) such that there is an elementary embedding \(k:M_1\to \textnormal{Ult}(M_0,W)\) with \(k([\textnormal{id}]_{U_1}) = [\textnormal{id}]_{W}\). Similarly, \(W_1\) is the least uniform ultrafilter \(W\) in the seed order of \(M_1\) such that there is an elementary embedding \(k:M_0\to \textnormal{Ult}(M_1,W)\) with \(k([\textnormal{id}]_{U_0}) = [\textnormal{id}]_{W}\). 
\begin{proof}
Suppose \(W\) is an ultrafilter of \(M_0\) such that there is an elementary embedding \(k: M_1 \to \text{Ult}(M_0,W)\) with \(k([\text{id}]_{U_1}) =  [\textnormal{id}]_{W}\). We must show \(W_0\wo W\) in \(M_0\). Working in \(M_0\), fix a comparison \(\langle Z_0,Z\rangle\) of \(\langle W_0,W\rangle\) to a common model \(P\). We have 
\begin{align}
j_{Z_0}([\text{id}]_{W_0}) &= j_{Z_0}(j_{W_1}([\textnormal{id}]_{U_1}))\label{k3}\\
&\leq j_Z(k([\textnormal{id}]_{U_1}))\label{k2} \\
&= j_Z([\text{id}]_{W}) \label{k1} 
\end{align}
\begin{figure}
\begin{center}\includegraphics[scale=.8]{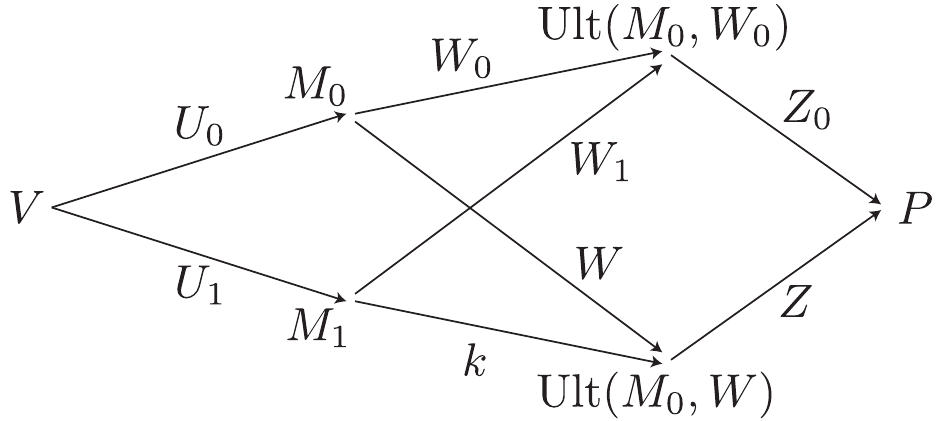}
\caption{The minimality of the canonical comparison}
\end{center}
\end{figure}
Here \cref{k3} follows from the fact that \(\langle W_0,W_1\rangle\) is canonical,  and \cref{k1} follows from the key property of \(k\). Finally, \cref{k2} is a consequence of the minimality of definable embeddings: since \(j_{Z_0}\circ j_{W_1}\) is definable in \(\text{Ult}(V,U_1)\), \[j_Z\circ k\restriction \text{Ord} \geq  j_{Z_0}\circ j_{W_1}\restriction \text{Ord}\] from which \cref{k2} follows. Thus \( j_{Z_0}([\text{id}]_{W_0})\leq j_Z([\text{id}]_{W})\), so the comparison \(\langle Z_0,Z\rangle\) witnesses \(W_0\wo W\) in \(M_0\), as desired.

The second half of the theorem involving \(W_1\) is exactly the same.
\end{proof}
\end{thm}

Regarding the previous theorem we mention an example as a kind of disclaimer. Suppose, in the situation of \cref{canonicalclassification}, that \(M_1\) is the ultrapower by a \(\mu\)-ultrafilter \(U_1\) and \(M_0\) is the ultrapower by its derived normal ultrafilter \(U_0\) on \(\kappa\), see \cref{mudef}. Thus \(M_0\) factors into \(M_1\). In other words there is an elementary embedding \(k:M_0\to \text{Ult}(M_1,F)\) where \(F\) is the principal ultrafilter whose seed is the empty sequence. Clearly \(F\) is the \(\wo\)-least uniform ultrafilter of \(M_1\), so one might misinterpret \cref{canonicalclassification} as asserting that \(W_1 \equiv_\text{RK} F\). This would imply, however, that \(M_1 = \text{Ult}(M_0,W_0)\), which is impossible in this case because \(M_1\nsubseteq M_0\): indeed, \(U_0\in M_1\setminus M_0\). In fact, \(W_1 = U_0\). The reason this does not contradict \cref{canonicalclassification} is that \(U_0 \swo F'\) where \(F'\) is the principal ultrafilter with seed \(\langle \kappa\rangle\). It is the ultrafilter \(F'\), and not \(F\), that according to \cref{canonicalclassification} lies above \(W_0\) in the seed order of \(M_0\), because of the demand there that \(k([\text{id}]_{F'}) = [\text{id}]_{U_0}\).

As a corollary of \cref{canonicalclassification}, we obtain the following theorem, which is quite useful for inductions along the seed order: it is used to show that one's inductive hypothesis actually holds of the comparison ultrafilters of the canonical comparison of the least counterexample. For an example of an application of \cref{bounding}, see \cref{decomp}.

\begin{cor}[Ultrapower Axiom]\label{bounding}
Suppose \(U_0\) and \(U_1\) are uniform ultrafilters. Let \(M_0 = \textnormal{Ult}(V,U_0)\) and \(M_1 = \textnormal{Ult}(V,U_1)\). Let \(\langle W_0,W_1\rangle\) be the canonical comparison of \(\langle U_0,U_1\rangle\). 

Then in \(M_0\), \(W_0 \wo j_{U_0}(W_1)\), and in  \(M_1\), \(W_1 \wo j_{U_1}(W_0)\).
\begin{proof}
We show that in \(M_0\), \(W_0 \wo j_{U_0}(W_1)\), since the second assertion has the same proof. For this, we apply \cref{canonicalclassification}: we claim there is an elementary embedding \(k : M_1\to \text{Ult}(M_0, j_{U_0}(W_1))\) such that \(k([\text{id}]_{U_1}) = [\text{id}]_{j_{U_0}(W_1)}\).

Indeed, let \(k = j_{U_0}\restriction M_1\). Then as usual, \begin{align*}k(M_1) &= j_{U_0}(M_1) \\&= j_{U_0}(\textnormal{Ult}(V,U_1)) \\&= \textnormal{Ult}(j_{U_0}(V),j_{U_0}(U_1)) \\&= \textnormal{Ult}(M_0,j_{U_0}(U_1))\end{align*} so \(k\) is an elementary embedding from \(M_1\) to \(\text{Ult}(M_0, j_{U_0}(W_1))\). Moreover \(k([\text{id}]_{U_1}) = j_{U_0}([\text{id}]_{U_1}) = [\text{id}]_{j_{U_0}(U_1)}\). Applying \cref{canonicalclassification}, we have that \(W_0 \wo j_{U_0}(W_1)\) in \(M_0\), as desired.
\end{proof}
\end{cor}

We note that \cref{bounding} is best possible in the sense that \(W_0 = j_{U_0}(W_1)\) and \(W_1 = j_{U_1}(W_0)\) is possible simultaneously: this happens for example whenever \(U_0\) and \(U_1\) are minimal ultrafilters such that \(\textsc{sp}(U_0) < \textsc{crt}(U_1)\).

A different perspective on \cref{canonicalclassification} gives the best possible lower bound on the canonical comparison as well. For this we need a definition.

\begin{defn}
Suppose \(U_0\) is a uniform ultrafilter. The {\it translation function} \(t_{U_0}\) is the function from the class of uniform ultrafilters to the class of uniform ultrafilters in \(\textnormal{Ult}(V,U_0)\) that sends a uniform ultrafilter \(U_1\) to the unique ultrafilter \(W_0\) of \(\textnormal{Ult}(V,U_0)\) such that for some \(W_1\), \(\langle W_0,W_1\rangle\) is the canonical comparison of \(\langle U_0,U_1\rangle\).
\end{defn}

We note that if \(U \equiv_{\textnormal{RK}} U'\) then \(t_{U} = t_{U'}\). Moreover, if \(W \equiv_{\textnormal{RK}} W'\) then \(t_U(W) \equiv_{\textnormal{RK}} t_U(W')\). On the other hand, it can happen that \(W \leq_{\textnormal{RK}} W'\) but \(t_U(W) \nleq_{\textnormal{RK}} t_U(W')\): indeed \(t_U(U)\) is always the principal ultrafilter of \(\text{Ult}(V,U)\) derived from \([\text{id}]_U\), but if \(U\) is a \(\mu\)-ultrafilter, there are ultrafilters \(Z <_\text{RK} U\) such that \(t_U(Z)\) is a nonprincipal ultrafilter of \(\text{Ult}(V,U)\). 

While the function \(t_U\) does not preserve the Rudin-Keisler order, it {\it does} preserve the seed order. 

\begin{prp}\label{fttf}
Suppose \(U\) is a uniform ultrafilter. If \(U_0 \swo U_1\) then \(t_U(U_0)\swo t_U(U_1)\) in \(\textnormal{Ult}(V,U)\).
\begin{proof}
Assume \(U_0\swo U_1\). Let \(U_0^* = t_U(U_0)\) and \(U_1^* =  t_U(U_1)\). Let \(\langle U^*_0,W_0\rangle\) be the canonical comparison of \(\langle U,U_0\rangle\) to a common model \(N_0\). Let \(\langle U^*_1,W_1\rangle\) be the canonical comparison of \(\langle U,U_1\rangle\) to a common model \(N_1\). Let \(\langle i_0,i_1\rangle\) be a comparison of \(\langle N_0,N_1\rangle\). We must show \(i_0([\text{id}]_{U_0^*}) < i_0([\text{id}]_{U_1^*})\). But note that \([\text{id}]_{U_0^*} = j_{W_0}([\text{id}]_{U_0})\) by the definition of a canonical comparison, and similarly \([\text{id}]_{U_1^*} = j_{W_1}([\text{id}]_{U_1})\). Since \(\langle i_0\circ j_{W_0}, i_1\circ j_{W_1}\rangle\) is a comparison of \(\langle U_0,U_1\rangle\), it must witness that \(U_0\swo U_1\), so \(i_0\circ j_{W_0}([\text{id}]_{U_0}) < i_1\circ j_{W_1}([\text{id}]_{U_1})\). Replacing like terms, \(i_0([\text{id}]_{U_0^*}) <  i_1([\text{id}]_{U_1^*})\), as desired.
\end{proof}
\end{prp}

The preceding proposition lets us think of \(t_U(W)\) as a kind of copy of \(W\) in \(\textnormal{Ult}(V,U)\), and this leads to the following theorem, where we use \(|\cdot |_S\) to denote the rank function for the seed order.

\begin{cor}
Suppose \(U_0\) and \(U_1\) are uniform ultrafilters. Let \(M_0 = \textnormal{Ult}(V,U_0)\) and \(M_1 = \textnormal{Ult}(V,U_1)\). Let \(\langle W_0,W_1\rangle\) be the canonical comparison of \(\langle U_0,U_1\rangle\). Then \(|W_0|_S^{M_0}\geq |U_1|_S\) and \(|W_1|_S^{M_1}\geq |U_0|_S\).
\begin{proof}
To see for example that \(|W_0|_S^{M_0}\geq |U_1|_S\), note that \(t_{U_0}\) order embeds the initial segment of the seed order below \(U_1\) into the initial segment of the seed order of \(M_0\) below \(W_0\), by \cref{fttf}.
\end{proof}
\end{cor}

\section{The \(E\)-order}\label{ZFCSection}
Motivated by \cref{closecomp}, we define a variant of the seed order which is equal to the seed order assuming the Ultrapower Axiom. The advantage to the \(E\)-order is that its transitivity is provable in ZFC alone. As a corollary, we show in ZFC that the seed order itself is an antisymmetric wellfounded relation, though it is not clear it must be transitive.

We begin by generalizing the notion of a comparison, relaxing the requirement that the comparison ultrafilters be internal.

\begin{defn}
Suppose \(U_0\) and \(U_1\) are countably complete ultrafilters. The pair \(\langle W_0,W_1\rangle\) is an {\it external comparison} of \(\langle U_0,U_1\rangle\) if 
\begin{enumerate}[(1)]
\item \((\textnormal{Ult}(V,U_0),W_0)\vDash W_0\) is a countably complete ultrafilter
\item \((\textnormal{Ult}(V,U_1),W_1)\vDash W_1\) is a countably complete ultrafilter
\item \(\textnormal{Ult}(\textnormal{Ult}(V,U_0),W_0) = \textnormal{Ult}(\textnormal{Ult}(V,U_1),W_1)\)
\item \(j_{W_0}\circ j_{U_0} = j_{W_1}\circ j_{U_1}\)
\end{enumerate}
The external comparison \(\langle W_0,W_1\rangle\) is {\it \(0\)-internal} if \(W_0\in \text{Ult}(V,U_0)\) and {\it \(1\)-internal} if \(W_0\in \text{Ult}(V,U_1)\). A comparison that is either \(0\)-internal or \(1\)-internal is called a {\it semi-comparison}.
\end{defn}

The ultrapowers above are formed using only functions from their domain models. We will not make any use of fully external comparisons, so our vagueness about whether the ultrapowers must be wellfounded should not be an issue: note that if a comparison is \(0\)-internal or \(1\)-internal then all the models involved must be wellfounded. It is the comparisons that are \(0\)-internal or \(1\)-internal in which we will be interested. 

The various abuses of notations involving comparisons generalize to external comparisons. In particular, we will sometimes denote an external comparison \(\langle W_0,W_1\rangle\) by \(\langle k_0,k_1\rangle\) where \(k_0 = j_{W_0}\) and \(k_1 = j_{W_1}\). We point out that if \(k_0:\text{Ult}(V,U_0)\to N\) and \(k_1:\text{Ult}(V,U_1)\to N\) are elementary embeddings with \(k_0\circ j_{U_0} = k_1\circ j_{U_1}\), then if \(k_0\) is a the ultrapower embedding associated with the possibly external \(\text{Ult}(V,U_0)\)-ultrafilter \(W_0\), \(k_1\) must also be a (possibly external) ultrapower embedding, associated for example with the ultrafilter derived from \(k_1\) using \(\langle k_0([\text{id}]_{U_0}),[\text{id}]_{W_0}\rangle\).

\begin{example}\label{eexample} For any two countably complete ultrafilters \(U_0\) and \(U_1\), there are \(0\)-internal and \(1\)-internal comparisons of \(\langle U_0,U_1\rangle\), namely \(\langle j_{U_0}(j_{U_1}), j_{U_0}\rangle\) and  \(\langle j_{U_1}, j_{U_1}(j_{U_0})\rangle\).
\end{example}

The main point of the basic theory of the \(E\)-order is that a slight generalization of this trivial example, which appears implicitly in \cref{Etransitive}, can be used in place of some of the simpler applications of the Ultrapower Axiom. We state this as a lemma:

\begin{lma}
Suppose \(M\) is a transitive model of \textnormal{ZFC}, that \(U_0\) is an \(M\)-ultrafilter, and \(U_1\) is a countably complete ultrafilter of \(M\). Then \(\langle U_0,U_1\rangle\) admits a \(0\)-internal comparison relative to \(M\).
\end{lma}

The following definition is the natural generalization of the notion of a canonical comparison to the context of external comparisons.

\begin{defn}
Suppose \(U_0\) and \(U_1\) are countably complete ultrafilters. Let \(a_0 = [\text{id}]_{U_0}\) and \(a_1 = [\text{id}]_{U_1}\). Suppose \(\langle W_0,W_1\rangle\) is an external comparison of \(\langle U_0,U_1\rangle\). Then \(\langle W_0,W_1\rangle\) is \(0\)-canonical if \(W_0\) is the \(\text{Ult}(V,U_0)\)-ultrafilter derived from \(j_{W_0}\) using \(j_{W_1}(a_1)\), and \(1\)-canonical if \(W_1\) is the \(\text{Ult}(V,U_1)\)-ultrafilter derived from \(j_{W_1}\) using \(j_{W_0}(a_0)\). Finally, \(\langle W_0,W_1\rangle\) is canonical if it is \(0\)-canonical and \(1\)-canonical.
\end{defn}

We have the following analog of \cref{closetoultra}, whose proof we omit:

\begin{thm}\label{Eclosetoultra}
Suppose \(U_0\) and \(U_1\) are countably complete ultrafilters. Let \(M_0 = \textnormal{Ult}(V,U_0)\) and \(M_1 = \textnormal{Ult}(V,U_1)\), and suppose that for some model \(N\),  \begin{align*}k_0&:M_0\to N\\ k_1&: M_1\to N\end{align*} are elementary embeddings such that \(k_0\circ j_{U_0} = k_1\circ j_{U_1}\). 

Then \(\langle U_0,U_1\rangle\) admits an external canonical comparison \(\langle W_0,W_1\rangle\) to a common model \(P\), which itself embeds in \(N\) by an elementary embedding \(h: P\to N\) such that \(h \circ j_{W_0} = k_0\) and \(h\circ j_{W_1} = k_1\). Finally, if \(k_0\) is close to \(M_0\), then \(\langle W_0,W_1\rangle\) is \(0\)-internal, and if \(k_1\) is close to \(M_1\), then \(W_1\) is \(1\)-internal.
\end{thm}

By \cref{eexample}, semi-comparisons do not really compare anything, in general. A semi-comparison is only meaningful in the following context.

\begin{defn}
The {\it \(E\)-order} is a binary relation \(\E\) defined for uniform ultrafilters \(U_0\) and \(U_1\) by \(U_0\E U_1\) if and only if there is a \(1\)-internal comparison \(\langle W_0,W_1\rangle\) of \(\langle U_0,U_1\rangle\) such that \(j_{W_0}([\textnormal{id}]_{U_0}) \leq j_{W_1}([\textnormal{id}]_{U_1})\). We say in this case that \(\langle W_0,W_1\rangle\) {\it witnesses} \(U_0\E U_1\) or that \(\langle W_0,W_1\rangle\) witnesses the \(E\)-order.

The {\it strict \(E\)-order} is the binary relation \(\sE\) defined for uniform ultrafilters \(U_0\) and \(U_1\) by \(U_0\sE U_1\) if and only if there is a \(1\)-internal comparison \(\langle W_0,W_1\rangle\) of \(\langle U_0,U_1\rangle\) such that \(j_{W_0}([\textnormal{id}]_{U_0}) < j_{W_1}([\textnormal{id}]_{U_1})\).
\end{defn}

Thus we do not define the strict \(E\)-order to be the strict part of the \(E\)-order. We will instead prove this. By \cref{closecomp}, we have the following:

\begin{prp}[Ultrapower Axiom]\label{Eextend}
The \(E\)-order is equal to the seed order.
\end{prp}

We now show that the \(E\)-order is a transitive, wellfounded, antisymmetric relation.

\begin{thm}\label{Etransitive}
The \(E\)-order is transitive.
\begin{proof}
Suppose \(U_0\E U_1\) and \(U_1\E U_2\). Let \(\langle h_0,h_1\rangle\) be a \(U_1\)-internal comparison of \(\langle U_0, U_1\rangle\) to a common model \(M\) witnessing \(U_0\E U_1\), and \(\langle i_1,i_2\rangle\) a \(U_2\)-internal comparison of \(\langle U_1, U_2\rangle\) to a common model \(N\) witnessing \(U_1\E U_2\). Let \(Q = i_1(M)\), which is well defined since \(M\) is a definable subclass of \(\text{Ult}(V,U_0)\). Then \(i_1\) restricts to an elementary embedding of \(M\) into \(Q\), but also \(i_1(h_1)\) is an ultrapower embedding from \(N\) into \(Q\). Moreover \(i_1(h_1)\circ i_1 = i_1 \circ h_1\). It follows that \(\langle i_1\circ h_0, i_1(h_1)\circ i_2\rangle\) is a \(U_2\)-internal comparison of \(\langle U_0,U_2\rangle\). Moreover this comparison witnesses \(U_0\E U_1\): \[i_1\circ h_0([\text{id}]_{U_0}) \leq i_1\circ h_1([\text{id}]_{U_1}) = i_1(h_1)\circ i_1([\text{id}]_{U_1}) \leq i_1(h_1)\circ i_2([\text{id}]_{U_2})\]
The inequalities above follow from the fact that \(\langle h_0,h_1\rangle\) and  \(\langle i_1,i_2\rangle\) witness the \(E\)-order.
\begin{figure}
\begin{center}
\includegraphics[scale=.8]{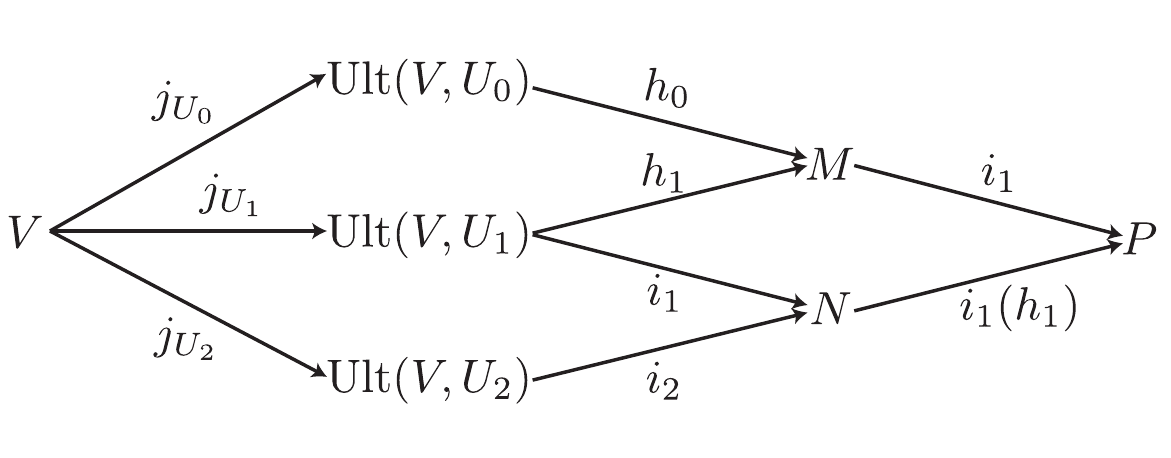}
\end{center}
\caption{The transitivity of the \(E\)-order}
\end{figure}
\end{proof}
\end{thm}

\begin{prp}\label{Estrict}
Suppose \(U\) is a uniform ultrafilter. There is no \(1\)-internal comparison \(\langle k_0,k_1\rangle\) of \(\langle U,U\rangle\) such that \(k_0([\textnormal{id}]_{U}) < k_1([\textnormal{id}]_{U})\). In other words, the strict \(E\)-order is strict.
\begin{proof}
This is an immediate consequence of \cref{minult}.
\end{proof}
\end{prp}

Thus we obtain the following nice property of semi-comparisons, the analog of \cref{welldefined}:

\begin{prp}\label{Ewelldefined}
Suppose \(U_0\E U_1\). There is no \(0\)-internal comparison \(\langle i_0,i_1\rangle\) of \(\langle U_0,U_1\rangle\) such that  \(i_0([\textnormal{id}]_{U_0}) > i_1([\textnormal{id}]_{U_1})\). That is, \(U_1 \not \sE U_0\).
\begin{proof}
Suppose otherwise. Let \(\langle h_0,h_1\rangle\) be a \(1\)-comparison of \(\langle U_0,U_1\rangle\) witnessing \(U_0 \E U_1\). Then by the proof of \cref{Etransitive} (with \(\langle i_1, i_2\rangle\) there replaced by \(\langle i_1,i_0\rangle\)), there is a \(1\)-internal comparison \(\langle k_0,k_1\rangle\) of \(U_0\) with itself such that \(k_0([\text{id}]_{U_0}) < k_1([\text{id}]_{U_1})\), contradicting \cref{Estrict}.
\end{proof}
\end{prp}

It is not true, however, that if \(U_0\E U_1\) then every \(1\)-internal comparison of \(\langle U_0,U_1\rangle\) witnesses \(U_0 \E U_1\). For example, if \(U_0 <_M U_1\) are normal measures on \(\kappa\), then \(\langle U_1, j_{U_1}(U_0)\rangle\) is a \(1\)-internal comparison of \(\langle U_0,U_1\rangle\), but \(j_{U_1}(\kappa) > \kappa = j_{U_1}(j_{U_0})(\kappa)\).

\begin{thm}\label{Eantisymmetric}
The \(E\)-order is antisymmetric.
\begin{proof}
Suppose \(U_0\E U_1\) and \(U_1\E U_0\). By \cref{Ewelldefined}, any \(0\)-internal comparison \(\langle k_0,k_1\rangle\) of \(\langle U_0,U_1\rangle\) witnessing \(U_1\E U_0\) must in fact be such that \(k_0([\text{id}]_{U_0}) = k_1([\text{id}]_{U_1})\). But then we can run the proof of \cref{antisymmetric}, using in particular the commutativity requirements of an external comparison.
\end{proof}
\end{thm}

\begin{cor}
The strict \(E\)-order is the strict part of the \(E\)-order.
\begin{proof}
If \(U_0\E U_1\) and \(U_0\neq U_1\) then \(U_0\sE U_1\), since otherwise there is an external comparison \(\langle k_0,k_1\rangle\) of \(\langle U_0,U_1\rangle\) with \(k_0([\text{id}]_{U_0}) = k_1([\text{id}]_{U_1})\), which implies \(U_0 = U_1\) by the argument of \cref{antisymmetric}. Combined with \cref{Estrict}, this establishes the corollary. 
\end{proof}
\end{cor}

By establishing that the \(E\)-order is a partial order, we have given a second, simpler proof of \cref{Eextend} (that the Ultrapower Axiom implies the seed order coincides with the extended seed order): the \(E\)-order is a partial order that by definition extends the seed order, and hence it must coincide with the seed order if the seed order is assumed to be linear.

\begin{thm}\label{Ewellfounded}
The \(E\)-order is wellfounded.
\begin{proof}
Suppose \[U_0\sgE U_1 \sgE U_2 \sgE\cdots\]
Let \(M^0_i = \textnormal{Ult}(V,U_i)\) and fix \(1\)-internal comparisons \(\langle \ell^0_{i+1},k^0_i\rangle\) of \(\langle M^0_{i+1},M^0_i\rangle\) to a common model \(M^1_i\) (so \(k^0_i\) is an internal ultrapower of \(M^0_i\)). Thus we have \[\ell^0_{i+1}([\text{id}]_{U_{i+1}}) < k^0_i([\text{id}]_{U_i})\] As in \cref{Etransitive}, we fix \(1\)-internal comparisons \(\langle \ell^1_{i+1},k^1_i\rangle\) of \(\langle M^1_{i+1},M^1_i\rangle\) to a common model \(M^2_i\). By construction \(\ell^1_1\circ k^0_1([\text{id}]_{U_1}) = k^1_0\circ \ell^0_1([\text{id}]_{U_1})\), and so \[\ell^1_1\circ \ell^0_2([\text{id}]_{U_2}) < \ell^1_1\circ k^0_1([\text{id}]_{U_1}) = k^1_0\circ \ell^0_1([\text{id}]_{U_1})\]
(Note that \(\langle \ell^1_1\circ \ell^0_2, k^1_0\circ \ell^0_1\rangle\) is {\it not} (necessarily) a \(1\)-internal comparison of \(\langle M^0_2,M^0_1\rangle\), since \(k^1_1\circ \ell^0_1\) is not necessarily an ultrapower embedding of \(M^0_1\), so we cannot cite \cref{Ewelldefined} here to prove \(\ell^1_1\circ \ell^0_2([\text{id}]_{U_2}) < k^1_0\circ \ell^0_1([\text{id}]_{U_1})\) as we cited in \cref{welldefined} at this point in \cref{wellfounded}.)

Continuing this way, we produce models \(M^n_i\) for all \(n,i<\omega\) and \(1\)-internal comparisons \(\langle \ell^n_{i+1},k^n_i\rangle\) of \(\langle M^n_{i+1},M^n_i\rangle\) to a common model \(M^{n+1}_i\). We also have  \[\ell^{i}_1\circ \ell^{i-1}_2 \cdots \ell^1_{i}\circ \ell^0_{i+1}([\text{id}]_{U_{i+1}}) < k^i_0\circ \ell^{i-1}_1\circ \ell^{i-2}_2\circ \cdots \ell^1_{i-1}\circ \ell^0_i([\text{id}]_{U_1})\]
Therefore the internal linear iterated ultrapower \[M^0_0\stackrel{k^0_0}{\longrightarrow} M^1_0\stackrel{k^1_0}{\longrightarrow}M^2_0\stackrel{k^2_0}{\longrightarrow}\cdots\] has an illfounded direct limit, and this is a contradiction.
\end{proof}
\end{thm}

The assertion that the \(E\)-order is linear is an apparently weak version of the Ultrapower Axiom, since the \(E\)-order is an extension of the seed order. A natural question is  whether the linearity of the \(E\)-order implies that of the seed order. We now show this is the case. It is perhaps a bit surprising that the existence of \(1\)-internal comparisons {\it witnessing the \(\E\)-order} is enough to prove the Ultrapower Axiom given that by \cref{eexample}, ZFC proves that every pair of ultrafilters admits \(0\)-internal and \(1\)-internal comparisons. 

For the proof that the linearity of the \(E\)-order implies the Ultrapower Axiom, we need a definition.

\begin{defn}\label{solid}
An \(1\)-internal comparison \(\langle W_0,W_1\rangle\) of the uniform ultrafilters \(\langle U_0,U_1\rangle\) to a common model \(P\) is {\it \(1\)-unstable} if there is a countably complete ultrafilter \(Z\) of \(P\) and a \(1\)-internal comparison \(\langle W_0',W_1'\rangle\) of \(\langle U_0,U_1\rangle\) to a common model \(P'\) such that \(W_1' = (W_1,Z)\) and \(j_{W_0'}([\text{id}]_{U_0}) < j_Z(j_{W_0}([\text{id}]_{U_0}))\). If \(\langle W_0,W_1\rangle\) is \(1\)-internal and not \(1\)-unstable, then \(\langle W_0,W_1\rangle\) is a {\it \(1\)-stable comparison}.

An external comparison \(\langle W_0,W_1\rangle\) is a {\it \(0\)-stable} comparison of the uniform ultrafilters \(\langle U_0,U_1\rangle\) if \(\langle W_1,W_0\rangle\) is a {\it \(1\)-stable} comparison of \(\langle U_1,U_0\rangle\).
\end{defn}

By \cref{minult}, every comparison (by internal ultrafilters) is stable. However, there can be unstable semi-comparisons that witness the \(E\)-order: for example, the factor embedding of the normal ultrapower derived from a \(\mu\)-measure is unstable.

\begin{lma}\label{solidmin}
Suppose \(\langle W_0,W_1\rangle\) is a \(1\)-canonical \(1\)-internal comparison of the uniform ultrafilters \(\langle U_0,U_1\rangle\). Then \(\langle W_0,W_1\rangle\) is \(1\)-stable if and only if \(W_1\) is minimal in the \(E\)-order of \(\textnormal{Ult}(V,U_1)\) among \(F_1\in \textnormal{Ult}(V,U_1)\) such that for some \(F_0\), \(\langle F_0,F_1\rangle\) is a \(1\)-canonical \(1\)-internal comparison of \(\langle U_0,U_1\rangle\).
\begin{proof}
Let \(M\) denote \(\text{Ult}(V,U_1)\). Suppose first \(\langle W_0,W_1\rangle\) is \(1\)-stable. Suppose towards a contradiction that \(\langle F_0,F_1\rangle\) is a \(1\)-canonical comparison of \(\langle U_0,U_1\rangle\) such that \(M\vDash F_1\sE W_1\), then let \(\langle D,E\rangle\) be a \(1\)-internal \(M\)-comparison of \(\langle W_1',W_1\rangle\) witnessing the \(E\)-order of \(M\). Then \(j_{D}(j_{F_0}([\text{id}]_{U_0})) = j_{D}([\text{id}]_{F_1}) < j_E([\text{id}]_{W_1}) = j_Z(j_{W_0}([\text{id}]_{U_0}))\). Thus the ultrafilter \(Z\) of \(M\) and the \(1\)-internal comparison \(\langle (F_0,D),(W_1,Z)\rangle\) contradict the \(1\)-stability of \(\langle W_0,W_1\rangle\).

Suppose conversely that \(W_1\) is minimal in the \(E\)-order of \(\textnormal{Ult}(V,U_1)\) among \(F_1\in \textnormal{Ult}(V,U_1)\) such that for some \(F_0\), \(\langle F_0,F_1\rangle\) is a \(1\)-canonical \(1\)-internal comparison of \(\langle U_0,U_1\rangle\). Suppose that \(\langle W_0,W_1\rangle\) is not \(1\)-stable. Then there is an ultrafilter \(Z\) of \(M\) and a \(1\)-internal comparison \(\langle W_0',W_1'\rangle\) of \(\langle U_0,U_1\rangle\) to the common model \(N\) such that \(W_1' = (W_1,Z)\) and \(j_{W_0'}([\text{id}]_{U_0}) < j_Z(j_{W_0}([\text{id}]_{U_0}))\). Let \(\langle F_0,F_1\rangle\) be a \(1\)-canonical \(1\)-internal comparison of \(\langle U_0,U_1\rangle\) to a common model \(P\) derived from \(\langle W_0',W_1'\rangle\) as in \cref{Eclosetoultra}, and fix \(h:P\to N\) such that \(h\circ j_{F_0} =  j_{W_0'}\) and \(h\circ j_{F_1} = j_{W_1'}\). Then \(\langle h,j_Z\rangle\) is a \(1\)-internal comparison witnessing \(F_1\sE W_1\) since \(h([\text{id}]_{F_1}) = h(j_{F_0}([\text{id}]_{U_0})) = j_{W_0'}([\text{id}]_{U_0}) < j_Z(j_{W_1}([\text{id}]_{U_0}))\). This contradicts the minimality of \(W_1\) among such \(F_1\).
\end{proof}
\end{lma}

\begin{cor}\label{solidlemma}
Every pair of uniform ultrafilters \(\langle U_0,U_1\rangle\) admits a canonical \(1\)-stable comparison. Moreover, if \(\langle U_0,U_1\rangle\) admits a \(1\)-internal comparison witnessing \(U_0\E U_1\), then \(\langle U_0,U_1\rangle\) admits a \(1\)-stable canonical comparison witnessing \(U_0\E U_1\).
\begin{proof}
The first part follows from \cref{eexample}, \cref{solidmin}, and the wellfoundedness of the \(E\)-order. For the second part, it is easy to check that if \(\langle W_0,W_1\rangle\) is a canonical \(1\)-internal comparison of \(\langle U_0,U_1\rangle\) witnessing the \(E\)-order such that \(W_1\) is minimal in the \(E\)-order of \(\textnormal{Ult}(V,U_1)\) among \(F_1\in \textnormal{Ult}(V,U_1)\) such that for some \(F_0\), \(\langle F_0,F_1\rangle\) is a \(1\)-canonical \(1\)-internal comparison of \(\langle U_0,U_1\rangle\) witnessing the \(E\)-order, then \(\langle W_0,W_1\rangle\) is \(1\)-stable.
\end{proof}
\end{cor}

An identical corollary holds for \(0\)-internal comparisons.

\begin{thm}
If the \(E\)-order is linear then the seed order is linear, or in other words, the Ultrapower Axiom holds.
\begin{proof}
Fix ultrafilters \(U_0\) and \(U_1\). We must show that \(\langle U_0, U_1\rangle\) admits a comparison by internal ultrafilters. We can reduce to the case that \(U_0\) and \(U_1\) are uniform ultrafilters concentrating on singletons, which is essentially the same as assuming they concentrate on ordinals, which what we actually assume. Let \(\xi_0 = [\text{id}]_{U_0}\) and \(\xi_1 = [\text{id}]_{U_1}\). We further assume that \(U_0 \neq U_1\).

Fix a canonical \(1\)-stable comparison \(\langle k_0,k_1\rangle\) of \(\langle U_0, U_1\rangle\) to a common model \(Q\) and a canonical \(0\)-stable comparison \(\langle i_0,i_1\rangle\) of \(\langle U_0,U_1\rangle\) to a common model \(P\). Let \(W\) be the uniform ultrafilter derived from \(k_1\circ j_{U_1}\) using \[\{k_0(\xi_0), k_1(\xi_1)\}\] and let \(Z\)  be the uniform ultrafilter derived from \(i_0\circ j_{U_0}\) using \[\{i_0(\xi_0), i_1(\xi_1)\}\] By the canonicity of \(\langle k_0,k_1\rangle\) and \(\langle i_0,i_1\rangle\), \(Q = \text{Ult}(V,W)\) and \(P = \text{Ult}(V,Z)\). By the linearity of the \(E\)-order, without loss of generality we may assume \(W\E Z\), and let \(\langle h_W,h_Z\rangle\) be a \(1\)-internal comparison of \(\langle W,Z\rangle\) witnessing \(W\E Z\). (The case \(Z\E W\) is handled symmetrically.) Thus \begin{equation}\label{WZ}h_W(\{k_0(\xi_0), k_1(\xi_1)\}) \leq h_Z(\{i_0(\xi_0), i_1(\xi_1)\})\end{equation}

We claim that \(h_W(k_0(\xi_0)) = h_Z(i_0(\xi_0))\). Note first that \begin{equation}\label{k0i0} h_W(k_0(\xi_0)) \geq h_Z(i_0(\xi_0))\end{equation} by the minimality of definable embeddings, \cref{minult}. On the other hand, by the stability of \(\langle i_0,i_1\rangle\), we must have \begin{equation}\label{i1k1} h_Z(i_1(\xi_1))\leq h_W(k_1(\xi_1))\end{equation} Otherwise, assuming \(h_W(k_1(\xi_1)) < h_Z(i_1(\xi_1))\), the \(0\)-internal comparison \[\langle h_Z\circ i_0, h_W\circ k_1\rangle\] along with the ultrafilter \(Z\) witnesses that \(\langle i_0,i_1\rangle\) is unstable. 

Combining \cref{WZ} and \cref{i1k1}, we see that \[h_W(k_0(\xi_0)) \leq h_Z(i_0(\xi_0))\] Otherwise, assuming \(h_W(k_0([\text{id}]_{U_0})) > h_Z(i_0([\text{id}]_{U_0}))\), \cref{i1k1} implies \[h_W(\{k_0(\xi_0),k_1(\xi_1)\}) > h_Z(\{i_0(\xi_0),i_1(\xi_1)\})\] which contradicts \cref{WZ}.

Thus \(h_W(k_0(\xi_0)) = h_Z(i_0(\xi_0))\), as claimed. Since \(\textnormal{Ult}(V,U_0)\) is generated by \(\xi_0\) over \(j_{U_0}[V]\) and since \(h_W\circ k_0\) and \(h_Z\circ i_0\) agree on \(j_{U_0}[V]\) and \(\xi_0\), we have that \(h_W\circ k_0 = h_Z\circ i_0\). Thus \(h_W\circ k_0\) is an internal ultrapower embedding of \(\text{Ult}(V,U_0)\), and it follows easily that \(k_0\) is close to \(\text{Ult}(V,U_0)\). Since \(k_0\) is an (external) ultrapower embedding by the definition of an external comparison, it follows from closeness that \(k_0\) is an {\it internal} ultrapower embedding. Hence \(\langle k_0,k_1\rangle\) is a comparison of \(\langle U_0, U_1\rangle\) by internal ultrafilters.
\end{proof}
\end{thm}

Our definition of the \(E\)-order was motivated by analogy with the seed order and therefore loosely by the mouse order. We close this section by trying to clarify the relationship between the \(E\)-order and other orderings on ultrafilters, in particular the Rudin-Keisler order, the Mitchell order, and perhaps somewhat surprisingly the Lipschitz order. We do this by putting down some obvious alternate definitions of the \(E\)-order.

\begin{defn}
Suppose \(U\) is a uniform ultrafilter and \(W\) is a uniform ultrafilter of \(\textnormal{Ult}(V,U)\). Then the {\it ultrafilter derived from \(U\) using \(W\)} is the uniform ultrafilter containing those \(X\) such that \(j_U(X)\cap [\textsc{sp}(W)]^{{<}\omega}\in W\).
\end{defn}

This is a variant on the notion of the \(U\)-limit of a sequence of ultrafilters, which is well-studied. We introduce this terminology only to point out the analogy between the Rudin-Keisler order and the \(E\)-order.

\begin{prp}\label{rkE}
Suppose \(U_0\) and \(U_1\) are uniform ultrafilters. Then \(U_0 \E U_1\) if and only if there is a uniform ultrafilter \(W\) of \(\textnormal{Ult}(V,U)\) such that \(\textsc{sp}(W) \leq [\textnormal{id}]_{U_1}\) and \(U_0\) is derived from \(U_1\) using \(W\).
\end{prp}

One must be a bit careful about the meaning of \(\textsc{sp}(W) \leq [\textnormal{id}]_{U_1}\) above. What we mean is that \([\text{id}]_W \leq j_W([\textnormal{id}]_{U_1})\) in the canonical wellorder. Thus the \(E\)-order appears as an extension of the Rudin-Keisler order: a version of the Rudin-Keisler order is given by restricting \(W\) to be principal in the proposition above, see the proof of \cref{factorlemma}.

We move on to the relationship between the \(E\)-order and the Lipschitz order. Here we restrict our attention to ultrafilters that concentrate on ordinals since this is somewhat easier to think about and we are not trying to develop the general theory.

\begin{defn}
Suppose \(\alpha\) is an ordinal. Suppose \(\tau : 2^{<\alpha}\to \{0,1\}\) and \(x\in 2^\alpha\). Then \(\tau * x\) is the element of \(2^\alpha\) given by \((\tau*x)(\xi) = \tau(x|\xi)\) for \(\xi < \alpha\).

Suppose \(A_0,A_1\subset 2^\kappa\). A function \(\tau : 2^{<\kappa}\to \{0,1\}\) is a {\it strong Lipschitz reduction} from \(A_0\) to \(A_1\) if for any \(x\in 2^\kappa\), \(x\in A_0\) if and only if \(\tau* x\in A_1\). If in addition, for each \(\alpha < \kappa\), \(F^\tau_\alpha = \{X\subseteq \alpha : \tau(X) = 1\}\) is a filter, then \(\tau\) is called a {\it filter reduction} from \(A_0\) to \(A_1\). If in addition \(F^\tau_\alpha\) is an ultrafilter, \(\tau\) is called an {\it ultrafilter reduction}. If in addition \(F^\tau_\alpha\) is countably complete, \(\tau\) is called an \(E\)-reduction.
\end{defn}

For clarity, we remark that in the definition of \(F^\tau_\alpha\), we identify a set \(X\subset \alpha\) with its characteristic function in \(2^\alpha\).

\begin{thm}
Suppose \(U_0\) and \(U_1\) are countably complete ultrafilters on a cardinal \(\kappa\). Then \(U_0 \sE U_1\) if and only if there is an \(E\)-reduction from \(U_0\) to \(U_1\).
\begin{proof}
Suppose \(\langle W_0,W_1\rangle\) is a canonical \(1\)-internal comparison of \(\langle U_0,U_1\rangle\). By the proof of \cref{smallsp}, \(\textsc{sp}(W_1) = [\text{id}]_{U_1}\). Let \(F:\kappa\to V\) be such that \([F]_{U_1} = W_1\). Then without loss of generality, for all \(\alpha < \kappa\), \(F(\alpha)\) is a countably complete ultrafilter on \(\alpha\). Moreover, by the usual argument, \(X\in U_0\) if and only if \(j_{U_1}(X)\cap [\text{id}]_{U_1}\in W_1\) if and only if \(\{\alpha < \kappa : X\cap \alpha\in F(\alpha)\}\in U_1\). Defining \(\tau : 2^{<\kappa}\to \{0,1\}\) by \(\tau(s) = F(\alpha)(s)\) for \(s\in 2^\alpha\), \(\tau\) is an \(E\)-reduction from \(U_0\) to \(U_1\). The converse is similar.
\end{proof}
\end{thm}

Hence the transitivity and strictness of the \(E\)-order are manifestations of the transitivity and strictness of the generalized Lipschitz order. Moreover, the relations induced by filter reductions and ultrafilter reductions are easily seen to be partial orders on filters and (perhaps countably incomplete) ultrafilters respectively. These proofs go through in ZF alone. For example:

\begin{thm}
Suppose \(U_0\), \(U_1\), and \(U_2\) are filters on \(\kappa\). If \(U_0\) filter-reduces to \(U_1\) and \(U_1\) filter-reduces to \(U_2\) then \(U_0\) filter-reduces to \(U_2\).

\begin{proof} Fix a filter reduction \(\tau_0\) from \(U_0\) to \(U_1\) and a filter reduction \(\tau_1\) from \(U_1\) to \(U_2\). We must construct a filter reduction \(\sigma\) from \(U_0\) to \(U_2\). This is just the composition: for \(\alpha < \kappa\) and \(s\in 2^\alpha\), define \(\sigma\) by \(\sigma(s) = \tau_1(\tau_0 * s)\). It is easy to check that \(\sigma\) is a strong Lipschitz reduction from \(U_0\) to \(U_2\). 

We must show that \(\sigma\) is a filter reduction. Fix \(\alpha < \kappa\), and we verify that \(F^\sigma_\alpha\) is a filter. For \(X\subseteq \alpha\), it is easy to check that \(X\in F^\sigma_\alpha\) if and only if \(\{\beta < \alpha : X\cap\beta \in F^{\tau_0}_\beta\} \in F^{\tau_1}_\alpha\). Then if \(X\in F^\sigma_\alpha\) and \(X\subseteq Y\), \[\{\beta < \alpha : X\cap\beta \in F^{\tau_0}_\beta\}\subseteq \{\beta < \alpha : Y\cap\beta \in F^{\tau_0}_\beta\}\] since for all \(\beta\), \(F^{\tau_0}_\beta\) is a filter. Thus since \(F^{\tau_1}_\alpha\) is a filter, \(\{\beta < \alpha : Y\cap\beta \in F^{\tau_0}_\beta\}\in F^{\tau_1}_\alpha\). It follows that \(Y\in F^\sigma_\alpha\). Verifying that \(F^\sigma_\alpha\) is closed under intersections is identical.
\end{proof}
\end{thm}

For normal ultrafilters, an \(E\)-reduction from \(U_0\) to \(U_1\) is essentially the same as a function \(F:\kappa \to V\) such that \([F]_{U_1} = U_0\). In other words, the \(E\)-order directly generalizes the Mitchell order on normal ultrafilters. Combinatorially, we believe this is a more natural generalization to all ultrafilters of the Mitchell order on normal ultrafilters than is the generalized Mitchell order. In some sense, considering the \(E\)-order rather than the generalized Mitchell order is a combinatorial manifestation of the shift in perspective regarding the linearity phenomenon described in the introduction. 

From the perspective of \(E\)-reductions, the Ultrapower Axiom arises as a natural generalization of Wadge's Lemma. This raises the following questions:

\begin{qst}
Assume \textnormal{AD + DC}. Must the Mitchell order on normal ultrafilters be linear? Must the \(E\)-order on uniform ultrafilters be linear?
\end{qst}

In the choiceless context, without Los's Theorem, the \(E\)-order is the order given by \(E\)-reductions (and similarly for the Mitchell order on normal ultrafilters). One can prove in ZF that this is a partial order. Its wellfoundedness is less clear, even assuming DC. The issue is that it is harder to make sense of the wellfoundedness of internal iterated ultrapowers in ZF + DC alone, which one needs in order to carry out the proof of \cref{Ewellfounded}. In the context of AD + DC, we know two proofs of this fact. One way is to generalize the Martin-Monk proof that the Wadge order is wellfounded. This shows the wellfoundedness of the \(E\)-order assuming that every set of reals has the property of Baire. Another proof, which looks more like \cref{Ewellfounded}, was suggested by Woodin: one can replace the use of Los's Theorem in \cref{Ewellfounded} by Los's Theorem for external ultrapowers of HOD, noting that since AD implies every uniform ultrafilter is OD, internal iterations of HOD of length \(\omega\) must be wellfounded by \cref{defwf}.

Finally, we remark that the \(E\)-order is a structural feature of ultrafilters that was not isolated before the investigation of the Ultrapower Axiom, despite extensive research on the Mitchell order and on orderings on ultrafilters in general.

\section{Irreducible Ultrafilters}\label{IrreducibleSection}
In this section we isolate the notion of an irreducible ultrafilter, which is a bit like a prime number but much larger. We then prove that assuming the Ultrapower Axiom, every countably complete ultrafilter factors as a finite iteration of irreducible ultrafilters. The factorizations are clearly not unique, and we do not know to what extent they are unique up to some kind of reordering.

Given these theorems, the analysis of ultrafilters in the context of the Ultrapower Axiom can be broken into two problems. First, what are the irreducible ultrafilters? Second, how can they be combined? What kind of iterations are possible? We take up the first question in this section, and the second in the next, to a certain extent. A key question, related to both questions, is whether the minimal irreducible ultrafilters are linearly ordered by the Mitchell order. This would be the best possible result for the Mitchell order, since if an ultrafilter is not irreducible, it is Mitchell incomparable with one of its factors.

We show under the Ultrapower Axiom alone that this optimal result holds for an initial segment of the irreducible ultrafilters. In fact this leads to a characterization of this initial segment of the seed order: below the least cardinal \(\kappa\) that carries a Radin sequence of length \(\kappa^+\), every countably ultrafilter is Rudin-Keisler equivalent to a finite iteration of ultrafilters \(U\) that are \(\alpha\)-normal for some \(\alpha < \textsc{crt}(U)^+\). The notion of an \(\alpha\)-Radin ultrafilter is the transfinite generalization of the notion of a \(\mu\)-measure, and the notion of an \(\alpha\)-normal ultrafilter is the transfinite generalization of the notion of a normal ultrafilter, see \cref{radindef} and \cref{alphanormal}. In particular, below the least \(\mu\)-measurable cardinal, every countably complete ultrafilter is Rudin-Keisler equivalent to a finite iteration of normal ultrafilters.

We end this section by noting that in the Mitchell-Steel models, the irreducible ultrafilters are linearly ordered by the Mitchell order. In fact, we will show that in this context, the irreducible ultrafilters are precisely the total ultrafilters that lie on the extender sequence. This is just a restatement of a theorem of Schlutzenberg \cite{Schlutzenberg} that characterizes total ultrafilters in the short extender models in terms of the extender sequence.

\subsection{A factorization lemma}
\begin{defn}\label{primitivity}
We say \(Z\) {\it factors as an iterated ultrapower \((U,W)\)} if \(U\) is a countably complete ultrafilter, \(W\) is a countably complete ultrafilter of \(\text{Ult}(V,U)\), and \(Z\equiv_\text{RK} (U,W)\), see \cref{iterating}. We say in this situation that \(U\) is a {\it factor} of \(Z\).

A nonprincipal countably complete ultrafilter \(Z\) is called {\it irreducible} if for any factor \(U\) of \(Z\), either \(U\) is principal or \(U\equiv_{\text{RK}} Z\).
\end{defn}

We point out that \(Z\) factors as \((U,W)\) if and only if \[\text{Ult}(V,Z) = \textnormal{Ult}(\text{Ult}(V,U),W)\] and \(j_Z = j_W\circ j_U\). We give two obvious examples of irreducible ultrafilters to motivate the definition.

\begin{prp}
If \(Z\) is a Dodd solid ultrafilter then \(Z\) is irreducible. 
\begin{proof}
Suppose \(Z\) factors as \((U,W)\). We must show that either \(U\) is principal or \(U\equiv_\text{RK} Z\). We may assume without loss of generality that \(U\) is a minimal ultrafilter. Under this assumption, we show that either \(U = Z\) or else \(U\) is the minimal principal filter.

 Assume therefore that \(U\neq Z\). By the proof of \cref{factorlemma}, \(U \wo Z\), but we can avoid using the Ultrapower Axiom here since the embedding \[j_W: \text{Ult}(V,U)\to \text{Ult}(V,Z)\] is an internal ultrapower embedding, and therefore the comparison \(\langle j_W,\text{id}\rangle\) witnesses \(U\wo Z\). Since we assumed \(U\neq Z\), we have \(U\swo Z\), and so by \cref{doddlin*}, \(U\mo Z\). But \(\text{Ult}(V,Z)\subseteq \text{Ult}(V,U)\), since \(\text{Ult}(V,Z)\) is an internal ultrapower of \(\text{Ult}(V,U)\). Thus \(U\mo U\), and it follows from \cref{mostrict} that \(U\) is principal. 
\end{proof}
\end{prp}

The following proposition is a consequence of the preceding one assuming GCH by \cref{soundult}, but using the proof of \cref{soundult}, one can avoid any cardinal arithmetic assumptions.

\begin{prp}\label{superirred}
If \(\mathcal U\) is a supercompactness measure then \(\mathcal U\) is irreducible. 
\begin{proof}
Assume without loss of generality that \(\mathcal U\) is a \(\kappa\)-supercompactness measure on \(\lambda\) where \(\text{cf}(\lambda)\geq \kappa\). By \cref{sollemma} and \cref{optimal}, \(\mathcal U\) is Rudin-Keisler equivalent to the ultrafilter \(Z\) derived from \(j_\mathcal U\) using the least generator \(\theta\) of \(j_\mathcal U\) above \(\lambda_* = \sup j_\mathcal U[\lambda]\). 

It suffices to show that \(Z\) is irreducible. Suppose that \(Z\) factors as \((U,W)\). Assuming without loss of generality that \(U\) is a minimal ultrafilter, we must show that either \(U = Z\) or else \(U\) is principal. 

Assume therefore that \(U \neq Z\). Let \(a = [\text{id}]_U\). Note then that \(j_W(a) \subseteq \theta\) by the proof of \cref{factorlemma}, since \(\theta\) is the minimum seed of \(Z\). It also follows from this proof that for any function \(f\in V\) and any \(b < j_W(a)\), \(j_W(a) \neq j_Z(f)(b)\). Since \(\theta\) is the least generator of \(Z\) above \(\lambda_*\), we must therefore have \(j_W(a) \subseteq \lambda_*\). Thus for some \(\delta < \lambda\), \(j_W(a)\subseteq j_Z(\delta)\), and so \(\textsc{sp}(U) < \lambda\) since \(U\) is the uniform ultrafilter derived from \(j_Z\) using \(j_W(a)\).

 On the other hand, \(\text{Ult}(V,Z)\subseteq \text{Ult}(V,U)\) since \(\text{Ult}(V,Z)\) is an internal ultrapower of \(\text{Ult}(V,U)\). Thus \(\text{Ord}^\lambda \subseteq \text{Ult}(V,U)\), since \(\text{Ord}^\lambda\subseteq\text{Ult}(V,Z)\), as \(Z\) is Rudin-Keisler equivalent to the supercompactness measure \(\mathcal U\) on \(\lambda\). Hence \(\text{Ult}(V,U)\) is closed under \(\lambda\)-sequences. Since \(\text{Ult}(V,U)\) is closed under \(\lambda\)-sequences and \(\textsc{sp}(U) < \lambda\), \(U\) is principal, as in \cref{doddgch}.
\end{proof}
\end{prp}

While irreducibility may seem like a somewhat weak condition since it refers only to factorizations into ultrafilters, it is actually equivalent to a property that on first glance may appear quite a bit stronger.

\begin{prp}
Suppose \(Z\) is an irreducible ultrafilter. Suppose \(j: V\to M\) is an elementary embedding, \(k:M \to \textnormal{Ult}(V,Z)\) is a close embedding, and \(k\circ j = j_Z\). Then either \(M = V\) and \(k = \textnormal{id}\) or else \(M = \textnormal{Ult}(V,Z)\) and \(j = j_Z\).
\begin{proof}
The point is that (even without assuming \(Z\) is irreducible) one can show that \(j\) is an ultrapower embedding and \(k\) is an internal ultrapower embedding of \(M\). 

To see that \(k\) is an ultrapower embedding of \(M\), note that the \(M\)-ultrafilter \(W\) derived from \(k\) using \([\text{id}]_Z\) is in \(M\) by closeness, and \(\text{Ult}(V,Z) = \text{Ult}(M,W)\) since the natural factor map is surjective. To see that \(j\) is an ultrapower embedding, suppose \(x\in M\). Then \(j_W(x) \in \textnormal{Ult}(V,Z)\) and so is of the form \(j_Z(f)([\text{id}]_Z)\) for some \(f\in V\). It follows that \[x = j_W^{-1}(j_Z(f)([\text{id}]_Z)) = j_W^{-1}(j_W(j(f))([\text{id}]_Z))\] In particular, every element \(x\) of \(M\) is definable in \(M\) from \([\text{id}]_Z\), \(W\), and a point in the range of \(j\). Thus \(j\) is an ultrapower embedding.
\end{proof}
\end{prp} 

Note if \(Z\) factors as an iterated ultrapower \((U,W)\), then \(\langle F,W\rangle\) is a comparison of \(\langle Z,U\rangle\), where \(F\) is any principal ultrafilter. We therefore call a factorization {\it canonical} if the corresponding comparison is canonical. Of course any factorization of \(Z\) is equivalent to a canonical one in a natural sense.

The following basic lemma is the key to much of our analysis. It seems to have no analogue in the context where the Ultrapower Axiom is not assumed.

\begin{lma}[Ultrapower Axiom]\label{decomplemma}
If \(Z\) is a uniform ultrafilter and \((U,W)\) is a canonical factorization of \(Z\) with \(U\) nonprincipal, then \(W\swo j_{U}(Z)\).
\begin{proof}
Since \((U,W)\) is canonical, by \cref{bounding}, \(W\wo j_{U}(Z)\). To finish we must show \(W\neq j_{U}(Z)\).

Assume to the contrary that \(W = j_{U}(Z)\). Note that \(j_{U}\) restricts to an elementary embedding \[j_U:M_Z\to \textnormal{Ult}(M_U,j_{U}(Z))\] We have assumed \(j_{U}(Z) = W\), and so \(\textnormal{Ult}(M_U,j_{U}(Z))\) is equal to \(M_Z\), since \(Z\) factors as \((U,W)\). Moreover, \[j_{U}\circ j_Z = j_{j_{U}(Z)}\circ j_{U} = j_{W}\circ j_{U} = j_Z\] Since \(U\) is nonprincipal, \(j_{U}\) is not the identity, and this contradicts the Rigid Ultrapowers Lemma, \cref{rul}.
\end{proof} 
\end{lma}

We note that the second paragraph of \cref{decomplemma} is just a proof that \(U\times Z \not\equiv_{\textnormal{RK}} Z\) if \(U\) is nonprincipal. It is the first sentence of the proof that we do not know how to generalize to the ZFC setting: we do not know of any wellfounded invariant that can replace the seed order in \cref{decomplemma}.

We define by induction the more general notion of factoring as a finite iterated ultrapower in order to introduce some notation for the following theorem. Having defined the notion of factoring as an iterated ultrapower of length \(n\), we then say \(Z\) factors as the iterated ultrapower \((U_0,U_1,\dots,U_n)\) of length \(n+1\) if \(Z\) factors as the iterated ultrapower \((U_0,W)\) and in \(\text{Ult}(V,U_0)\), \(W\) factors as the iterated ultrapower \((U_1,\dots,U_n)\) of length \(n\). In this case, we will occasionally denote \(\textnormal{Ult}(V,Z)\) by \(\textnormal{Ult}(V,U_0,\dots,U_n)\).

\begin{thm}[Ultrapower Axiom]\label{decomp}
Every countably complete ultrafilter \(W_0\) factors as a finite iteration of irreducible ultrafilters. 

Moreover this iteration \((U_0,U_1,\dots,U_{n_0-1})\) can be chosen with the property that for each \(i < n_0\), \begin{align}\label{funny}\textnormal{Ult}(V,U_0,\dots,U_i)\vDash W_{i+1}\swo j_{U_i}(W_i)\end{align} where for \(i \leq n_0\), \(W_i\) is the unique ultrafilter of \(\textnormal{Ult}(V,U_0,\dots,U_{i-1})\) such that \(((U_0,\dots,U_{i-1}),W_i)\) is a canonical factorization of \(Z\).
\begin{proof}
This is proved by iterating \cref{decomplemma}. The iteration must terminate in finitely many steps: otherwise one obtains an illfounded internal iteration by countably complete ultrafilters, since for all \(i< \omega\), \(W_{i+1}\swo j_{U_i}(W_i)\). We give the details, and a picture.
\begin{figure}
\begin{center}\includegraphics[scale=.9]{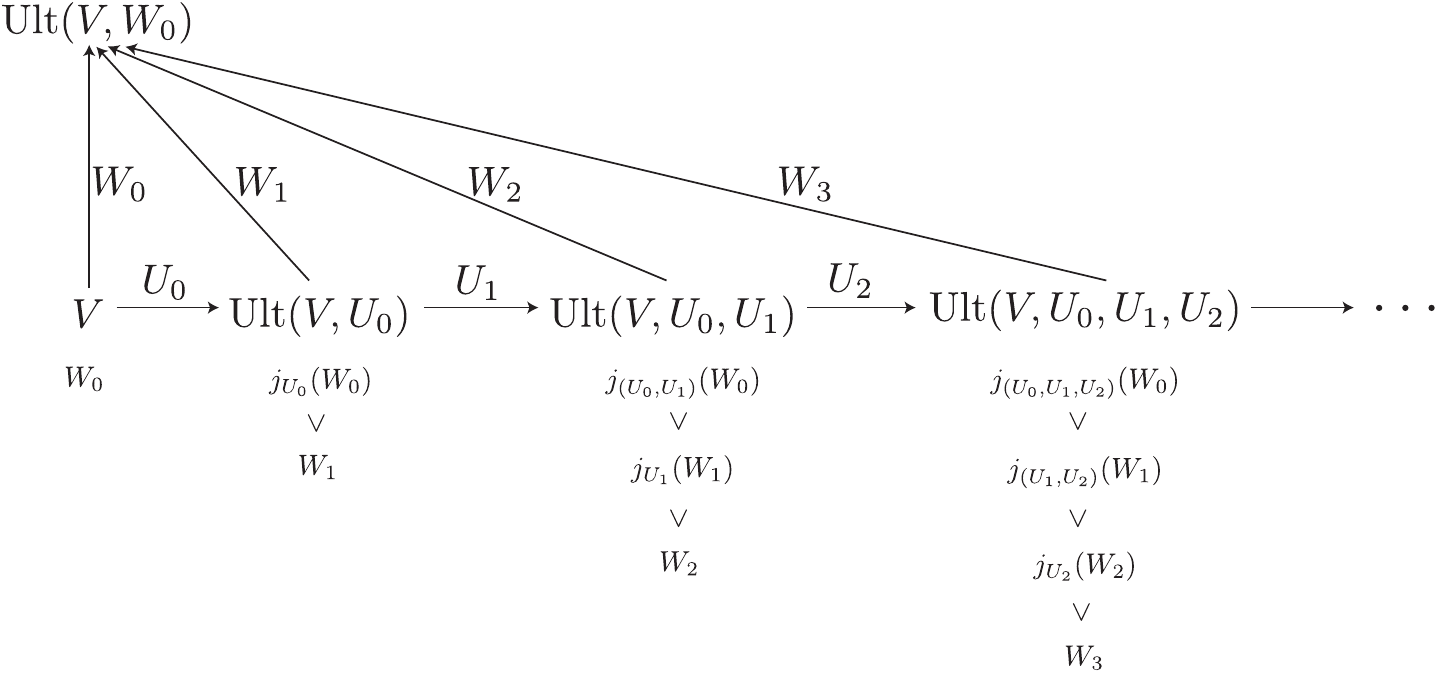}
\caption{The factorization algorithm.}
\end{center}
\end{figure}

Suppose \(W_0\) is nonprincipal. Let \(U_0\) be the \(\wo\)-least nonprincipal ultrafilter that is a factor of \(W_0\). By \cref{factorlemma}, \(U_0\) is irreducible: any factor of \(U_0\) precedes \(U_0\) in the seed order, yet is a factor of \(W_0\), hence is equal to \(U_0\). Let \(W_1\) be such that \((U_0,W_1)\) is a canonical factorization of \(W_0\). 

Now repeat this process, working in \(\textnormal{Ult}(V,U_0)\). Suppose \(W_1\) is nonprincipal. Let \(U_1\) be the \(\wo\)-least nonprincipal ultrafilter that is a factor of \(W_1\). By \cref{factorlemma}, \(U_1\) is irreducible. Let \(W_2\) be such that \((U_1,W_2)\) is a canonical factorization of \(W_1\).

Now repeat this process until the least \(n_0 \leq \omega\) such that \(W_{n_0}\) is nonprincipal or \(n_0 = \omega\). Thus we obtain for each finite \(n \leq n_0\) a sequence \((U_0,U_1,\dots,U_{n-1})\) and a sequence \((W_0,W_1,\dots, W_n)\) such that for each \(i < n\), \((U_i,W_{i+1})\) is a canonical factorization of \(W_i\) in \(\textnormal{Ult}(V,U_0,\dots,U_{i-1})\) and \(U_i\) is irreducible in \(\textnormal{Ult}(V,U_0,\dots,U_{i-1})\). By \cref{decomplemma}, \[\textnormal{Ult}(V,U_0,\dots,U_i)\vDash W_{i+1}\swo j_{U_i}(W_i)\]
A very easy induction shows that \(((U_0,\dots,U_{i-1}),W_i)\) is a canonical factorization of \(W_0\).

We claim \(n_0 < \omega\): otherwise the \(W_i\) form a descending sequence in the seed order of the direct limit of the internal linear iteration \(\langle U_0,U_1,U_2,\dots\rangle\). Thus \(W_{n_0}\) is principal, and so \((U_0,U_1,\dots,U_{n_0 - 1})\) is a factorization of \(Z\) into irreducible ultrafilters. By construction \((U_0,U_1,\dots,U_{n - 1})\) satisfies \cref{funny}.
\end{proof}
\end{thm}

\subsection{Radin Ultrafilters}\label{RealmSection}
We begin this section with a little lemma that shows that assuming the Ultrapower Axiom, any irreducible ultrafilter is either a normal ultrafilter or a \(\mu\)-ultrafilter. Hence below the least \(\mu\)-measurable cardinal, every ultrafilter factors as a finite iteration of normal ultrafilters. The rest of the section is a generalization of this result which involves a more complicated induction, which is why we have separated the much simpler base case.
\begin{defn}\label{mudef}
Suppose \(Z\) is a countably complete ultrafilter. Let \(U\) be its derived normal ultrafilter. Then \(Z\) is a \(\mu\)-ultrafilter if \(U \mo Z\). A cardinal \(\kappa\) is a \(\mu\)-measurable cardinal if there is a \(\mu\)-ultrafilter \(Z\) such that \(\textsc{crt}(Z) = \kappa\).
\end{defn}

A standard argument shows that \(\kappa\) is \(\mu\)-measurable if and only if there is a \(\kappa\)-complete \(\mu\)-ultrafilter on \(\kappa\).

Because we will cite it a couple times, we prove the following trivial fact.

\begin{lma}\label{trivialfact}
Suppose \(U_0\) and \(U_1\) are countably complete ultrafilters. Then \(\textnormal{Ult}(V,U_1)\) is an internal ultrapower of \(\textnormal{Ult}(V,U_0)\) if and only if there is a comparison \(\langle k_0,k_1\rangle\) of \(\langle U_0,U_1\rangle\) such that \(k_0([\textnormal{id}]_{U_0})\in \textnormal{ran}(k_1)\).
\begin{proof}
The forward direction is obvious, since one can take \(k_1\) to be the identity. Conversely, by \cref{closetoultra}, take a canonical comparison \(\langle W_0,W_1\rangle\) of \(\langle U_0,U_1\rangle\) that factors into \(\langle k_0,k_1\rangle\) via \(h\). Then \(h(j_{W_0}([\textnormal{id}]_{U_0})) =k_0([\textnormal{id}]_{U_0})\in \text{ran}(k_1) = \text{ran}(h\circ j_{W_1})\) and hence \(j_{W_0}([\textnormal{id}]_{U_0})\in \text{ran}(j_{W_1})\). Since \(\langle W_0,W_1\rangle\) is canonical, \(W_1\) is derived from \(j_{W_1}\) using \(j_{W_0}([\textnormal{id}]_{U_0})\) and hence \(W_1\) is principal. It follows that \(\text{Ult}(V,U_1)\) is the internal ultrapower of \(\text{Ult}(V,U_0)\) by \(W_0\), as desired.
\end{proof}
\end{lma}

The argument that leads to the structural analysis of the seed order turns out to be a simple variation on the argument of \cref{mitchellnormalmeasures}. 

\begin{thm}[Ultrapower Axiom]\label{muthm}
Suppose \(Z\) is an minimal irreducible ultrafilter. Either \(Z\) is a normal ultrafilter or \(Z\) is a \(\mu\)-ultrafilter.
\begin{proof}
Let \(U\) be the normal ultrafilter derived from \(Z\). Suppose \(U\) is an ultrafilter on \(\kappa\). Let \(\langle k_U,k_Z\rangle\) be a comparison of \(\langle U,Z\rangle\) to a common model \(M\). By \cref{factorlemma}, \(\langle k_U,k_Z\rangle\) witnesses \(U\wo Z\). 

Let \(i: \text{Ult}(V,U)\to \text{Ult}(V,Z)\) be the factor embedding. Then by \cref{minult}, \(k_U\restriction \text{Ord} \leq k_Z\circ i\restriction \text{Ord}\). Therefore since \(i(\kappa) = \kappa\), either \(k_U(\kappa) = k_Z(\kappa)\) or else \(k_U(\kappa) < k_Z(\kappa)\). If \(k_U(\kappa) = k_Z(\kappa)\) then by \cref{trivialfact}, \(U\) is a factor of \(Z\), so since \(Z\) is a minimal irreducible ultrafilter, \(U = Z\) and \(Z\) is normal. On the other hand, if \(k_U(\kappa) < k_Z(\kappa)\), then \(U \mo Z\) by the argument of \cref{mitchellnormalmeasures}: \(X\in U\) if and only if \(\kappa \in j_U(X)\) if and only if \(k_U(\kappa) \in k_U(j_U(X)) = k_Z(j_Z(X))\) if and only if \(k_U(\kappa)\in k_Z(j_Z(X))\cap k_Z(\kappa) = k_Z(j_Z(X)\cap \kappa) = k_Z(X)\), and since \(k_Z\) is definable over \(\textnormal{Ult}(V,Z)\), it follows that \(U\) is definable over \(\textnormal{Ult}(V,Z)\), so \(U\mo Z\). Thus in this case \(Z\) is a \(\mu\)-ultrafilter.
\end{proof}
\end{thm}

\begin{cor}[Ultrapower Axiom]
Below the least \(\mu\)-measurable cardinal, every ultrafilter factors as a finite iteration of normal ultrafilters.
\begin{proof}
This is immediate from \cref{muthm} and \cref{decomp}.
\end{proof}
\end{cor}

We now proceed with an inductive analysis of the seed order, which amounts to pushing the proof of \cref{muthm} as hard as we can. The following definition gives a sense of the analysis we are intending.

\begin{defn}\label{radindef}
Fix an embedding \(j:V\to M\) with critical point \(\kappa\). By recursion on \(\alpha < j(\kappa)\), we define simultaneously whether \(j\) is {\it \(\alpha\)-Radin} and, if \(j\) is \(\alpha\)-Radin, we define the {\it \(\alpha\)-normal ultrafilter \(F^j(\alpha)\) derived from \(j\)}. 

Any nontrivial elementary embedding \(j\) is {\it \(0\)-Radin}, and the {\it \(0\)-normal ultrafilter} \(F^j(0)\) derived from \(j\) is the normal ultrafilter derived from \(j\). Assume now that \(j:V\to M\) is \(\xi\)-Radin for all \(\xi < \alpha\) and \(\alpha > 0\). Then we say \(j\) is {\it \(\alpha\)-Radin} if the sequence \(\langle F^j(\xi): \xi < \alpha \rangle\) is an element of \(M\). If \(j\) is \(\alpha\)-Radin, then the {\it \(\alpha\)-normal ultrafilter} \(F^j(\alpha)\) derived from \(j\) is the ultrafilter  with space \(V_\kappa\) derived from \(j\) using \(\langle F^j(\xi) : \xi < \alpha\rangle\).
\end{defn}

We will say that \(j:V\to M\) is \({<}\alpha\)-Radin to mean that \(j\) is \(\xi\)-Radin for all \(\xi < \alpha\).

\begin{defn}\label{alphanormal}
We say an ultrafilter is {\it \(\alpha\)-normal} if it is the \(\alpha\)-normal ultrafilter derived from some elementary embedding.
\end{defn}

We mention that an \(\alpha\)-normal ultrafilter \(U\) is equal to the \(\alpha\)-normal ultrafilter derived from \(U\). We can now state the main theorem of this section:

\begin{thm}[Ultrapower Axiom]\label{radinanalysis}
Suppose \(U\) is an irreducible ultrafilter such that \(\textsc{crt}(U) = \kappa\). Either \(U\) is Rudin-Keisler equivalent to an \(\alpha\)-normal ultrafilter for some \(\alpha < \kappa^+\) or else \(U\) is \(\alpha\)-Radin for all \(\alpha < \kappa^+\).
\end{thm}

Our proof of this proceeds by an analysis of the extenders of ultrafilters that do not have too many generators, showing by induction that these ultrafilters are Dodd solid. Once this analysis has been carried out, we show that the ultrafilters below \({<}\kappa^+\)-Radin coincide with the ultrafilters we have analyzed, and this yields a proof of \cref{radinanalysis}. From the proof of this theorem one can extract various local variations in the spirit of \cref{muthm}. 

We start with a generalization of \cref{muthm} which seems to be part of a (failed attempt at an) inductive proof of the Dodd solidity of irreducible ultrafilters. We conjecture that the Ultrapower Axiom alone is not strong enough to carry out such an induction. Before this lemma, we need the following definition.

\begin{defn}
Suppose \(U\) is a countably complete ultrafilter. The {\it Dodd fragment ordinals} of \(U\) are the least \(a\in [\text{Ord}]^{<\omega}\) such that \(U|a\notin \text{Ult}(V,U)\). The set \(a\) is  {\it of measure type} if it is a successor element of \([\text{Ord}]^{<\omega}\), in which case we let \(a^-\) denote its predecessor, and call \(a^-\) the {\it Dodd fragment parameter} of \(U\). If \(a\) is not of measure type then we let \(a^-\) denote \(a \setminus \{\min a\}\), and we call \(a^-\) the {\it Dodd fragment parameter} of \(U\) and \(\min a\) the {\it Dodd fragment projectum} of \(U\). 
\end{defn}

\begin{lma}\label{doddfail1}
Suppose \(U\) is a countably complete ultrafilter with \(\textsc{crt}(U) =\kappa\). Suppose the Dodd fragment ordinals \(a\) of \(U\) are not of measure type. Then the Dodd fragment projectum \(\min a\) is a limit of \(a^-\)-generators of \(U\) and \(\textnormal{cf}(\min a) > \kappa\).
\end{lma}

\begin{thm}[Ultrapower Axiom]\label{doddfail2}
Suppose \(U\) is a countably complete ultrafilter. Let \(a\) be the Dodd fragment ordinals of \(U\), and suppose \(a\) is of measure type. Let \(W\) be the ultrafilter derived from \(U\) using \(a^-\). Then \(\textnormal{Ult}(V,U)\) is an internal ultrapower of \(\textnormal{Ult}(V,W)\).
\begin{proof}
Suppose first that \(a\) is of measure type. Let \(W\) and \(b\) be as in the statement of the theorem and let \(i: \text{Ult}(V,W)\to \text{Ult}(V,U)\) be the factor embedding. Let \(\bar b\) be such that \(i(\bar b) = b\). Let \(\langle k_W,k_U\rangle\) be a comparison of \(\langle W,U\rangle\). Thus for some inner model \(N\), \(k_W: M_W\to N\) and \(k_U:M_U\to N\). Note that \(k_W \restriction \text{Ord}\leq k_U\circ i \restriction \text{Ord}\) by the minimality of definable embeddings. In particular, \(k_W(\bar b) \leq k_U(b)\). There are two cases. Suppose first that \(k_W(\bar b) = k_U(b)\). In this case, \(\textnormal{Ult}(V,U)\) is an internal ultrapower of \(\textnormal{Ult}(V,W)\) by \cref{trivialfact}.

Suppose second that \(k_W(\bar b) < k_U(b)\). In this case we claim \(W\) is in \(\text{Ult}(V,U)\). We have 
\begin{align*}X\in W&\iff \bar b\in j_W(X)\\
&\iff k_W(\bar b)\in k_W(j_W(X))\\
&\iff k_W(\bar b)\in k_U(j_U(X))\\
&\iff k_W(\bar b)\in k_U(j_U(X))\cap \{u : u < k_U(b)\}\\
&\iff k_W(\bar b)\in k_U(j_U(X)\cap \{u : u < b\})\end{align*}
By definition the function \(X\mapsto j_U(X)\cap \{u : u < b\}\) is equal to \(U|b\), and by the minimality of the Dodd fragment ordinals \(a\), \(U|b\in \text{Ult}(V,U)\). It follows that \(W\in \text{Ult}(V,U)\). But from \(W\) and \(U|b\) one can compute \(U|a\), and this contradicts that \(U|a\notin \text{Ult}(V,U)\).
\end{proof}
\end{thm}

\begin{cor}[Ultrapower Axiom]
Suppose \(U\) is an irreducible ultrafilter whose Dodd fragment ordinals are of measure type. Then \(U\) is Dodd solid.
\end{cor}

It is the possibility of irreducible ultrafilters with Dodd fragment ordinals {\it not of measure type} that obstructs the attempt to prove that all irreducible ultrafilters are Dodd solid. Schlutzenberg's proof of \cref{shortland} shows that such pathological irreducible ultrafilters do not exist in the Mitchell-Steel models. Under the Ultrapower Axiom alone, it seems the following is the best we can do.

\begin{cor}[Ultrafilter Axiom]\label{doddgens}
Suppose \(U\) is a \(\kappa\)-complete minimal irreducible ultrafilter with at most \(\kappa\) generators. Then \(U\) is Dodd solid.
\begin{proof}
Let \(a\) be the Dodd fragment ordinals of \(U\). If \(a\) is of measure type, then letting \(W\) be derived from \(U\) using \(a^-\), by \cref{doddfail1}, \(W\) is a factor of \(U\), so \(U = W\) by irreducibility. It follows that \(U\) is Dodd solid, since \(U|a^-\in U\). If \(a\) is not of measure type, then \(\min a\) is a limit of \(a^-\)-generators of \(U\) and \(\textnormal{cf}(\min a) > \kappa\), but this contradicts that \(U\) has at most \(\kappa\) generators.
\end{proof}
\end{cor}

We can give a fairly complete analysis of the ultrafilters of \cref{doddgens} using the notion of \(\alpha\)-normal and \(\alpha\)-Radin ultrafilters.

\begin{defn}
Suppose \(E\) is an extender with critical point \(\kappa\), and let \(M = \text{Ult}(V,E)\). Let \(\delta = (2^\kappa)^M\). A generator \(\xi\) of \(E\) is called {\it local} if \(\xi < \delta^{+M}\). The extender \(E\) is called {\it local} if all its generators are local.
\end{defn}

We will use the following lemma repeatedly:

\begin{lma}\label{locallemma}
Suppose \(U\) is an ultrafilter and \(\xi\) is a local generator of \(U\). Then \(U | \xi + 1\) is equivalent to the ultrafilter \(W\) derived from \(U\) using \(\xi\).
\begin{proof}
Let \(N = \text{Ult}(V,W)\) and let \(M = \text{Ult}(V,U)\). Let \(k : N\to M\) be the factor embedding. It suffices to show that \(\xi\subseteq \text{ran}(k)\). Let \(\kappa = \textsc{crt}(U)\). We have \(k(\kappa) = \kappa\) and \(P(\kappa) \cap N = P(\kappa)\cap M = P(\kappa)\) so \(k(P(\kappa)) = P(\kappa)\). Note that \(M\) thinks there is a surjection from \(P(\kappa)\) to \(\xi\) since \(\xi\) is a local generator. By elementarity, \(\text{Ult}(V,W)\) thinks there is a surjection \(\bar f: P(\kappa) \to k^{-1}(\xi)\), so \(k(\bar f)\) is a surjection from \(k(P(\kappa)) = P(\kappa)\) to \(\xi\). For any \(\alpha < \xi\), there is some \(A\in P(\kappa)\) such that \(k(\bar f)(A) = \alpha\), and \(k(\bar f)(A) = k(\bar f)(k(A)) = k(\bar f(A))\), so \(\alpha\in \text{ran}(k)\), as desired.
\end{proof}
\end{lma}

\begin{prp}[Ultrapower Axiom]\label{normalgens}
Suppose \(U\) is a \(\kappa\)-complete ultrafilter that is \({<}\alpha\)-Radin for some \(\alpha \leq \kappa^+\). Let \(\langle \nu_i : i < \alpha\rangle\) be the first \(\alpha\) generators of \(U\). 
Let \(W_i\) denote the ultrafilter derived from \(U\) using \(\nu_i\). Then for all \(i < \alpha\), \(W_i\) is irreducible, \(\nu_i\) is a local generator, \(F^U(i)\) is Rudin-Keisler equivalent \(W_i\), and if \(i\) is a successor, say \(i = \bar i +1\), then \(\nu_i = o(W_{\bar i})\).
\begin{proof}
Fix \(i_0 < \alpha\). Suppose the statement is true for all \(i < i_0\). We show it is true for \(i_0\). 

Let \(k: M_{W_{i_0}} \to M_U\) be the factor embedding. Let \(\bar \nu = k^{-1}(\nu_{i_0})\), so \(\bar \nu\) is a seed for \(W_{i_0}\). To see that \(W_{i_0}\) is irreducible, suppose towards a contradiction \(W_{i_0}\) decomposes as a nontrivial iterated ultrapower \((F_0,F_1)\). Suppose \(F_0\) is minimal, and let \(a = [\text{id}]_{F_0}\). It is easy to see that \(j_{F_1}(a) < \{\bar \nu\}\), and in fact this follows from \cref{factorlemma}. Thus \(k\circ j_{F_1}\) embeds \(M_{F_0}\) into \(M_U\) sending \(a\) below \(\nu_{i_0}\). By the locality of the generators of \(U\) below \(\nu_{i_0}\), it follows that \(F_0 = W_i\) for some \(i < i_0\). But by our induction hypothesis, \(F^U(i)\) is Rudin-Keisler equivalent \(W_i\). We have that \(F^U(i)\in M_{W_{i_0}}\subseteq M_{F_0}\), and this implies \(F_0\in M_{F_0}\), a contradiction. So \(W_{i_0}\) is irreducible.

Carrying along the rest of the induction hypotheses now requires breaking into cases based on whether \(i_0\) is a limit or a successor ordinal.

Let \(\delta = (2^\kappa)^{M_U}\). Suppose \(i_0\) is a limit ordinal. Note that \(\delta^{+M_{U|\nu_{i_0}}} = \sup_{i < i_0} \delta^{+M_{U|\nu_i}}\), and so \(\text{cf}(\delta^{+M_{U|\nu_{i_0}}}) \leq \kappa\): by our locality induction hypothesis, \(\langle \delta^{+M_{U|\nu_i}} : i < i_0\rangle = \langle \nu_i : i < i_0\rangle\) is increasing, and hence its limit has cofinality \(\text{cf}(i_0) \leq \kappa\). Thus \(\delta^{+M_{U|\nu_{i_0}}} < \delta^{+M_U}\) since \(M_U\) is closed under \(\kappa\)-sequences. It follows that \(\nu_{i_0} = \delta^{+M_{U|\nu_{i_0}}}\), since \(\nu_{i_0}\) is the critical point of the factor map \(\text{Ult}(V,U|\nu_{i_0})\to \text{Ult}(V,U)\), which must be greater than or equal to \(\delta^{+M_{U|\nu_{i_0}}}\) and hence must equal \(\delta^{+M_{U|\nu_{i_0}}}\) since \(\delta^{+M_{U|\nu_{i_0}}} < \delta^{+M_U}\).

By induction, \(F^U(i)\in M_{U|\nu_{i+1}}\) for all \( i < i_0\), and so \(F^U(i) \in M_{U|\nu_{i_0}}\) for all \(i < i_0\). Since \(\nu_{i_0}\) is a local generator, \(U|\nu_{i_0}+1\) is equivalent to \(W_{i_0}\), the ultrafilter derived from \(U\) using \(\nu_{i_0}\). Hence \(M_{U|\nu_{i_0} + 1} = M_{W_{i_0}}\) is closed under \(\kappa\)-sequences, so \(\langle F^U(i) :  i < i_0\rangle\in M_{W_{i_0}}\). Therefore there is a factor embedding \(k : M_{F^U(i_0)}\to M_{W_{i_0}}\). We must show that \(k\) is the identity. But since \(\langle F^U(i) :  i < i_0\rangle\) is in the range of \(k\), \(\nu_{i_0}\) is in the range of \(k\), since \(\nu_{i_0} = \sup_{i < i_0} \nu_i\) and the sequence \(\langle \nu_i: i < i_0\rangle\) is definable from \(\langle F^U(i) :  i < i_0\rangle\). Since \(W_{i_0}\) is derived from \(\nu_{i_0}\), it follows that \(k\) is surjective, hence \(k\) is the identity. This finishes the case that \(i_0\) is a limit ordinal.

Suppose next that \(i_0\) is a successor ordinal, so \(i_0 = \bar i + 1\). Denote \(W_{\bar i}\) by \(W\). Let \(\gamma = o(W)\). Since \(o(W) = o^{M_U}(W)\), \(\gamma < \delta^{+M_U}\). By \cref{doddgens}, \(W\) is Dodd solid, and so by \cref{doddlin}, \(W\) is the unique Dodd solid ultrafilter of Mitchell rank \(\gamma\) in \(M_U\). Let \(k: M_W\to M_U\) be the factor embedding. Note that \(\textsc{crt}(k) = \nu_i\). We will show \(\textsc{crt}(k) = \gamma = \delta^{+M_W}\), which shows that \(\nu_i\) is local and \(\nu_i = o(W_{\bar i})\). First, note that \(k(\gamma) \neq \gamma\) so \(\textsc{crt}(k) \leq \gamma\): otherwise, in \(M_W\), there is a unique Dodd solid ultrafilter \(\bar W\) of Mitchell rank \(\gamma\), and \(k(\bar W) = W\), so that \(\bar W = W\), and hence \(W \in M_{W}\) contradicting the strictness of the Mitchell order. On the other hand, \(\gamma \leq \delta^{+M_W}\) since every ultrafilter \(W'\mo W\) is in \(M_W\) and its Mitchell rank is correctly computed there. It follows that \(\textsc{crt}(k) \leq \gamma \leq \delta^{+M_W}\leq \textsc{crt}(k)\), so all these ordinals are equal, as desired. 

Finally we must show that \(F^U(i_0)\) is Rudin-Keisler equivalent to \(W_{i_0}\). Clearly \(F^U(i_0)\) is Rudin-Keisler equivalent to the ultrafilter \(F\) derived from \(U\) using \(F^U(\bar i)\). Also \(F\) is Rudin-Keisler equivalent to the ultrafilter \(W'\) derived from \(U\) using \(W_{\bar i}\), since \(W_{\bar i}\) is Rudin-Keisler equivalent to \(F^U(\bar i)\) by induction. But then \(W'\) is Rudin-Keisler equivalent to \(W\) since \(W\) is derived \(U\) using from \(\gamma\), the Mitchell rank of \(W_{\bar i}\), which is interdefinable in \(M_U\) with \(W_{\bar i}\): clearly \(\gamma\) is definable from \(W_{\bar i}\) in \(M_U\), and conversely \(W_{\bar i}\) is the unique Dodd solid ultrafilter of Mitchell rank \(\gamma\) in \(M_U\).
\end{proof}
\end{prp}

One can actually also use the seed order where we used the Mitchell order above (in the Dodd solid case they are essentially the same thing). The analysis of \cref{normalgens} leads to the proof of the main theorem of this section, \cref{radinanalysis}.

\begin{proof}[Proof of \cref{radinanalysis}]
Let \(\alpha\) be least such that \(U\) is not \(\alpha\)-Radin. We assume \(\alpha < \kappa^+\) and show \(U\) is \(\alpha\)-normal. Let \(\langle \nu_i : i < \alpha\rangle\) be the first \(\alpha\) generators of \(U\). There are two cases. First suppose \(\alpha\) is a limit ordinal. Then \(F^{U}(i)\in M_U\) for all \(i < \alpha\) but since \(U\) is not \(\alpha\)-Radin, \(\langle F^U(i) : i < \alpha\rangle\notin M_U\). Since \(\alpha < \kappa^+\) and \(M_U\) is closed under \(\kappa\)-sequences, this is a contradiction.

Second suppose \(\alpha\) is a successor ordinal, say \(\alpha = i + 1\). Then the fact that \(U\) is not \(\alpha\)-Radin means that \(F^U(i)\notin M_U\). By \cref{normalgens}, \(F^U(i)\) is Rudin-Keisler equivalent to the ultrafilter derived from \(U\) using \(\nu_{i}\), or in other words to \(U|\nu_i+1\). It follows that the Dodd fragment ordinals of \(U\) are just \(\{\nu_i+1\}\): \(U|\nu_i\in M_U\) and \(U|\nu_i + 1\) is not. Thus \(U\) is of measure type with Dodd fragment parameter equal to \(\nu_i\). Hence by \cref{doddfail2},  the ultrafilter derived from \(U\) using \(\nu_i\) is a factor of \(U\), and so since \(U\) is irreducible, \(F^U(i) \equiv_{\text{RK}} U\), as desired.
\end{proof}

\begin{cor}[Ultrapower Axiom]\label{radinthm}
Suppose the seed rank of \(W\) is strictly below that of the \(\wo\)-least \(Z\) that is \({<}\textsc{crt}^+\)-Radin, if there is such a \(Z\). Then \(W\) factors as a finite iteration of ultrafilters \(U\) that are \(\alpha\)-normal for some \(\alpha < \textsc{crt}(U)^+\).
\end{cor}

What lies beyond \cref{radinthm} seems to rest on the answer to the following question.

\begin{qst}
Assume the Ultrapower Axiom. Is the \(\wo\)-least \({<}\kappa^+\)-Radin ultrafilter \(\kappa^+\)-Radin?
\end{qst}

In the short extender models, the answer is {\it yes}. In fact, much more is true.

\subsection{In the Mitchell-Steel Models}
The following theorem is due to Schlutzenberg in a much stronger form with a slightly different statement.
\begin{thm}[Schlutzenberg]\label{shortland}
Suppose \(\mathcal M\) is an iterable Mitchell-Steel premouse satisfying \textnormal{ZFC}. Then the minimal irreducible ultrafilters of \(\mathcal M\) are precisely the total ultrafilters coded by extenders on the sequence of \(\mathcal M\).
\end{thm}

\begin{cor}
Suppose \(\mathcal M\) is an iterable Mitchell-Steel premouse satisfying \textnormal{ZFC}. Then in \(\mathcal M\), the Mitchell order wellorders the class of minimal irreducible ultrafilters of \(\mathcal M\).
\end{cor}

The minimal irreducible ultrafilters are the largest class of ultrafilters that can be wellordered by the Mitchell order in the following sense. Obviously no such class can contain two Rudin-Keisler equivalent ultrafilters, since such ultrafilters are Mitchell incomparable. This is one reason for the minimality constraint: we must restrict to a canonical class of Rudin-Keisler representatives. Another reason is that for nonminimal ultrafilters, two ultrafilters \(U\) and \(W\) may be Mitchell incomparable even though for some \(U'\equiv_\textnormal{RK} U\), \(U'\mo W\). This is impossible if \(U\) is assumed to be minimal. 

Finally the reason for the irreducibility constraint is that if \(Z\) factors as an iterated ultrapower \((U,W)\), then \(Z\) and \(U\) bear no Mitchell relation with each other. On the one hand, \(\textnormal{Ult}(V,Z)\subseteq \textnormal{Ult}(V,U)\) and hence \(U\notin \textnormal{Ult}(V,Z)\) since \(U\notin \textnormal{Ult}(V,U)\) by \cref{mostrict}. On the other hand, assume towards a contradiction that \(Z\in \textnormal{Ult}(V,U)\). Assume without loss of generality that \(U\) is minimal. One can compute \(j_Z\restriction P(\textsc{sp}(Z))\) in \(\textnormal{Ult}(V,U)\) from \(Z\) alone as in \cref{annoying}, and since \(\textsc{sp}(Z)\geq \textsc{sp}(U)\) by \cref{spaces}, this suffices to compute \(j_Z\restriction P(\textsc{sp}(U))\). Then \(j_U\restriction P(\textsc{sp}(U))\) is in \(\textnormal{Ult}(V,U)\), since \(j_U\restriction P(\textsc{sp}(U))  = (j_W)^{-1}\circ j_Z \restriction P(\textsc{sp}(U))\) and \(j_W\) is definable in \(\textnormal{Ult}(V,U)\). But from \(j_U\restriction P(\textsc{sp}(U))\), one easily computes \(U\), contradicting \cref{mostrict}. Thus \(Z\notin \textnormal{Ult}(V,U)\).

\section{The Ultrapower Lattice}\label{UltrapowerLatticeSection}
Let \(\mathcal D\) be the category of ultrapowers of \(V\) with morphisms the internal ultrapower embeddings. In this section we study this category under the assumption of the Ultrapower Axiom. In a sense it turns out to be quite simple.

\begin{thm}[Ultrapower Axiom]\label{latticethm}
Let \(\mathcal D\) be the category of ultrapowers of \(V\) with morphisms the internal ultrapower embeddings. Then \(\mathcal D\) is a locally finite lattice with joins given by canonical comparisons. Its partial order is equal to the partial order of reverse inclusion.
\end{thm}

Let \(\mathcal I\) be the partial order consisting of minimal ultrafilters under the relation \(U <_\mathcal I Z\) if there is a countably complete ultrafilter \(W\) of \(M_U\) such that \(\textnormal{Ult}(M_U,W) = M_Z\) and \(j_W\circ j_U = j_Z\). (This is sometimes called the Rudin-Frolik order, but that term has also been used to refer to a certain suborder.) In this case we will say that \(U\) is a factor of \(Z\). Under the Ultrapower Axiom, the second requirement (commutativity) holds automatically, and so the partial order \(\mathcal I\) is isomorphic to the partial order \(\mathcal D\).

We prove in a series of propositions that \(\mathcal D\) forms a lattice.

\begin{prp}[Ultrapower Axiom]\label{joins}
Suppose \(U_0\) and \(U_1\) are uniform ultrafilters and \(\langle W_0,W_1\rangle\) is their canonical comparison to a common model \(N\). Suppose \(\langle W_0',W_1'\rangle\) is another comparison of \(\langle U_0,U_1\rangle\) to a common model \(N'\). Then there is an internal ultrapower embedding \(i:N \to N'\).
\begin{proof}
We are trying to show that the ultrafilter \(Z = (U_0,W_0) \equiv_{\textnormal{RK}} (U_1,W_1)\) is a factor of the ultrafilter \(Z' = (U_0,W_0') \equiv_{\textnormal{RK}} (U_1,W_1')\). Fix a comparison \(\langle F,F'\rangle\) of \(\langle Z,Z'\rangle\). It suffices to show that \(j_F([\text{id}]_Z) \in \text{ran}(j_{F'})\). Since \(\langle W_0,W_1\rangle\) is canonical, \([\text{id}]_Z = j_{W_0}([\text{id}]_{U_0}) \cup j_{W_1}([\text{id}]_{U_1})\). Thus 
\begin{align*}
j_F([\text{id}]_Z) &= j_F(j_{W_0}([\text{id}]_{U_0}) \cup j_{W_1}([\text{id}]_{U_1})) \\
&= j_{F}(j_{W_0}([\text{id}]_{U_0})) \cup j_F(j_{W_1}([\text{id}]_{U_1}))\\
&=j_{F'}(j_{W'_0}([\text{id}]_{U_0})) \cup j_{F'}(j_{W'_1}([\text{id}]_{U_1}))\in \text{ran}(j_{F'})
\end{align*}
The final equality follows from \cref{defemb}.
\end{proof}
\end{prp}

It follows that \(\mathcal D\) has a join operation: \(M\vee N\) is the common model of the canonical comparison. Similarly, \(\mathcal I\) has a join operation. We note that \cref{joins} can be easily extended to an interesting universal property for the canonical comparison, improving \cref{closetoultra} under the assumption of the Ultrapower Axiom.

\begin{prp}[Ultrapower Axiom]\label{universalproperty}
Suppose \(U_0\) and \(U_1\) are uniform ultrafilters. Let \(M_0 = \textnormal{Ult}(V,U_0)\) and \(M_1 = \textnormal{Ult}(V,U_1)\). Suppose \(\langle W_0,W_1\rangle\) is the canonical comparison of \(\langle U_0,U_1\rangle\) to a common model \(P\). Suppose that for some model \(N\), there are close embeddings \begin{align*}k_0&:M_0\to N\\ k_1&: M_1\to N\end{align*} such that \(k_0\circ j_{U_0} = k_1\circ j_{U_1}\). Then there is a close embedding \(h: P\to N\) such that \(h \circ j_{W_0} = k_0\) and \(h\circ j_{W_1} = k_1\).
\begin{proof}
By \cref{closetoultra} and the uniqueness of canonical comparisons, there is an elementary embedding \(h:P\to N\) such that \(h \circ j_{W_0} = k_0\) and \(h\circ j_{W_1} = k_1\). It suffices to show that \(h\) is close. Suppose \(a\in N\), and let \(F\) be the \(P\)-ultrafilter derived from \(h\) using \(a\). We must show that \(F\in P\). Then \(\langle (W_0,F),(W_1,F)\rangle\) is a comparison of \(\langle U_0,U_1\rangle\) by {\it internal} ultrafilters to the common model \(\textnormal{Ult}(P,F)\): this is because for example \((W_0,F)\) is derived from the close embedding \(k_0: M_0\to N\) and therefore is an internal ultrafilter of \(M_0\). It follows from \cref{joins} that \(\textnormal{Ult}(P,F)\) is an internal ultrapower of \(P\). Let \(i:P \to  \textnormal{Ult}(P,F)\) be the internal ultrapower embedding. We claim \(i = j_F\). This is because \(i\circ j_{W_0} = j_F\circ j_{W_0}\) and \(i\circ j_{W_1} = j_F\circ j_{W_1}\) by the uniqueness of ultrapower embeddings, \cref{everythingcommutes}, and \(P\) is generated by \(\text{ran}(j_{W_0}) \cup \text{ran}(j_{W_1})\) by the definition of the canonical comparison. Since \(i = j_F\), \(F\in P\). 
\end{proof}
\end{prp}

By the proof of Schlutzenberg's theorem \cref{shortland}, in the Mitchell-Steel models, the main branch embeddings of the comparison of two ultrapowers of \(V\) by least disagreement is equal to embeddings of the canonical comparison. It is conceivable that comparisons by least disagreement have abstract universality properties in general, but it is not at all clear this is the case. 

We now prove that \(\mathcal D\) satisfies the ascending chain condition below any point. 

\begin{prp}[Ultrapower Axiom]\label{acc}
Every infinite \(<_\mathcal D\)-increasing chain of ultrapowers is unbounded in \(\mathcal D\).
\begin{proof}
The proof is essentially the same as that of \cref{decomp}. Suppose towards a contradiction that \[M_0 <_\mathcal D M_1 <_\mathcal D M_2 <_\mathcal D\cdots\] is bounded by \(N\in \mathcal D\). Fix \(U_i\in M_i\) such that \(M_{i+1} = \textnormal{Ult}(V,U_i)\). Since each \(M_i <_\mathcal D N\), we may also fix \(W_i\in M_i\) such that \(N = \textnormal{Ult}(M_i,W_i)\). By \cref{decomplemma}, we have \(W_{i+1} \swo j_{U_i}(W_i)\), since \(W_i\) factors as the iterated ultrapower \((U_i,W_{i+1})\). Thus the direct limit of the \(M_i\) is illfounded, which is a contradiction since this is an internal iteration by the definition of \(\mathcal D\).
\end{proof}
\end{prp}

\begin{prp}[Ultrapower Axiom]\label{meets}
Suppose \(X\subseteq \mathcal D\) is a nonempty class. Then \(X\) has a greatest lower bound in \(\mathcal D\).
\begin{proof}
By \cref{decomplemma}, there are no bounded infinite \(<_\mathcal D\)-increasing chains. Thus the collection \(C = \{N : \forall M\in X\ N \leq_\mathcal D M\}\) has a maximal element. In fact, \(C\) has a {\it maximum} element: for this, it suffices to show that \(C\) is directed under \(\leq_\mathcal D\): then any two maximal elements of \(C\) are equal. Suppose that \(N_0,N_1\in C\). Let \(Q\) be the canonical comparison of \(\langle N_0,N_1\rangle\). By \cref{joins}, any comparison of \(\langle N_0,N_1\rangle\) is an internal ultrapower of \(Q\). But note that any \(M\in X\) itself constitutes a comparison of \(\langle N_0,N_1\rangle\), so \(Q\leq_\mathcal D M\). It follows that \(Q\in C\), and in particular \(\langle N_0,N_1\rangle\) has an upper bound in \(C\), so \(C\) is directed. This completes the proof. 
\end{proof}
\end{prp}

In particular, \(\mathcal D\) has a meet operation \(\wedge\) so that \((\mathcal D,\vee,\wedge)\) is a lattice compatible with the order \(\leq_\mathcal D\). This proves part of \cref{latticethm}. We now turn to local finiteness. It follows abstractly from \cref{meets} that if one adjoins a formal top element to \(\mathcal D\), one obtains a complete lattice. This may seem interesting, but in fact the reason for the conditional completeness of the ultrapower lattice is that there are no infinite joins to take: no infinite subclass of \(\mathcal D\) has an upper bound. In the terminology of lattice theory, the ultrapower lattice is locally finite. This is a strengthening of the ascending chain condition below a point. We note that it is perhaps the only significant application of treating seeds as sequences of ordinals rather than single ordinals. We also note that it settles Question 5.11 of \cite{KanamoriUltrafilters} negatively, in a very strong sense, under the Ultrapower Axiom.

\begin{lma}[Ultrapower Axiom]\label{locallyfinite}
A countably complete ultrafilter has only finitely many factors up to Rudin-Keisler equivalence.
\begin{proof}
Suppose not, and let \(Z\) be the \(\wo\)-least uniform ultrafilter with infinitely many factors. List \(\omega\) of them as \(U_0, U_1,U_2,\dots\).  For each \(n\), fix \(W_n\in M_{U_n}\) such that \(Z\) factors as the iteration \((U_n,W_n)\).  Let \(b\) be the minimum seed of \(Z\), and for \(n< \omega\), let \(a_n\) be the minimum seed of \(U_n\). We then have \(j_{W_n}(a_n) < b\) for all \(n\) by \cref{factorlemma}. For each \(n\), let \(\xi_n = \max b\setminus j_{W_n}(a_n)\). Since \(b\) is finite, we may fix an infinite set \(A\subseteq \omega\) such that \(\xi_n\) takes the constant value \(\xi\) for \(n\in A\). 

Fix \(m < \omega\). Let \(F_m = \bigvee_{n\in A\cap m} U_n\). Then \(F_m \leq_\mathcal I Z\). The key point is that in fact \(F_m <_\mathcal I Z\). For this let \(i : \textnormal{Ult}(V,F_m) \to \textnormal{Ult}(V,Z)\) be the internal ultrapower embedding. We will show that \(i\) is not surjective. For each \(n\in A\cap m\), let \(i_n: \textnormal{Ult}(V,U_n)\to \textnormal{Ult}(V,F_m)\) be the internal ultrapower embedding. Then \(i \circ i_n = j_{W_n}\) by the uniqueness of ultrapower embeddings, and so \(i\circ i_n(a_n) = j_{W_n}(a_n)\). Let \(u_m = i(\bigcup_{n\in A\cap m}i_n(a_n)) = \bigcup_{n\in A\cap m} j_{W_n}(a_n)\). By our choice of \(A\), \(u_m \setminus \xi = b \setminus \xi + 1\) since \( j_{W_n}(a_n)\setminus \xi = b \setminus \xi + 1\) for all \(n \in A\). Thus \(u_m < b\). Being the canonical comparison of the ultrafilters \(U_n\) for \(n \in A\cap m\), \(\textnormal{Ult}(V,F_m)\) is generated over \(j_{F_m}[V]\) by \(\bigcup_{n\in A\cap m}i_n(a_n)\). Hence the range of \(i\) is generated over \(j_Z[V]\) by \(u_m\). It follows that \(i\) is not surjective, since its range is generated over \(j_Z[V]\) by a set of ordinals lexicographically below the minimum seed \(b\) of \(Z\).

\begin{figure}
\begin{center}
\includegraphics[scale=.8]{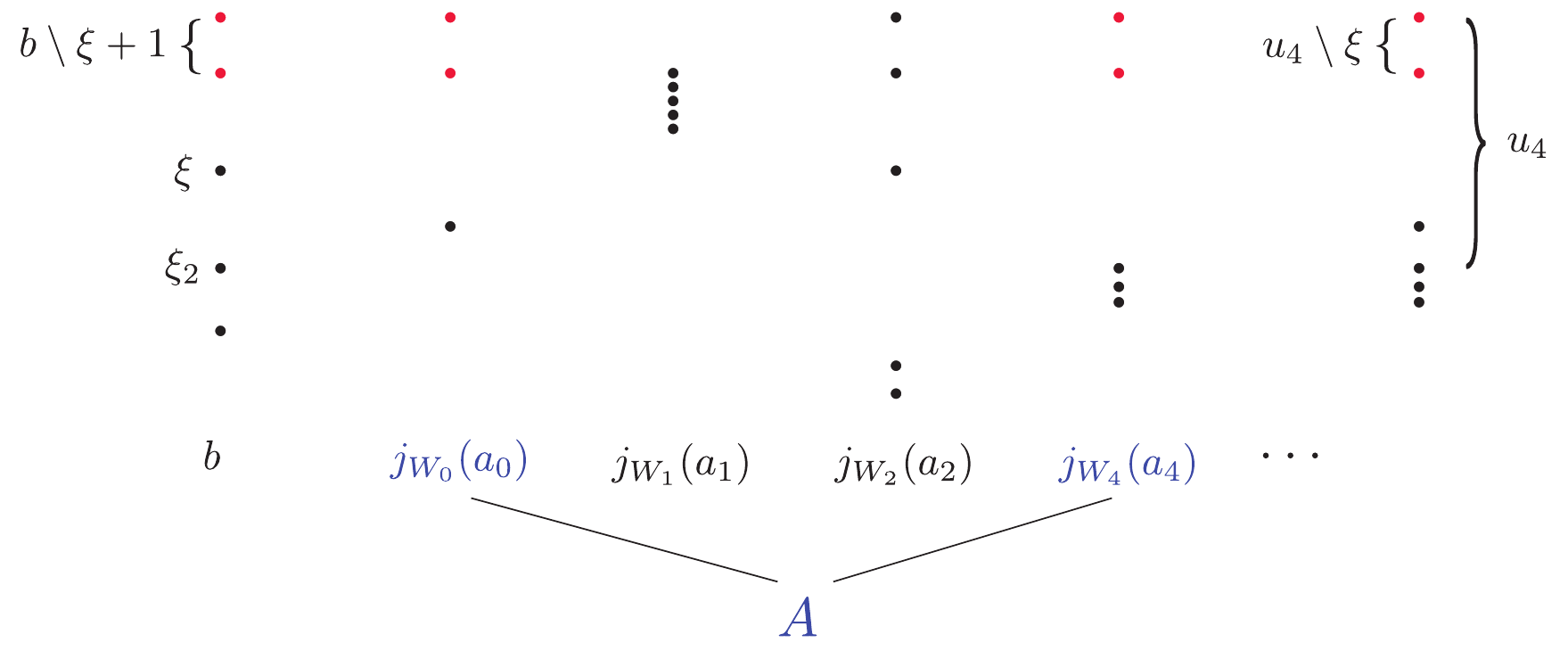}
\end{center}
\caption{Artist's rendering of \cref{locallyfinite}.}
\end{figure}

Now consider the sequence \(\langle F_m : m < \omega\rangle\). This sequence is weakly \(\leq_\mathcal D\)-increasing. The key point is that its supremum must be \(Z\). To see this, note that for any infinite set \(B\subseteq\omega\), \(Z = \bigvee_{n\in B} U_n\), since this join is a factor of \(Z\) and has infinitely many factors, while \(Z\) is the \(\wo\)-least ultrafilter with infinitely many factors. In particular \(\bigvee_{m < \omega} F_m =  \bigvee_{n\in A} U_n = Z\). Since \(F_m \neq Z\) for all \(m < \omega\), the sequence \(\langle F_m : m < \omega\rangle\) does not stabilize, and this contradicts the ascending chain condition below \(Z\), \cref{acc}.
\end{proof}
\end{lma}

We note that \(M \leq_\mathcal D N\) implies that \(N \subseteq M\). It is natural to ask whether \(\leq_\mathcal D\) is the reverse inclusion order. This turns out to be true. We use the following general fact.

\begin{lma}\label{defnfactor}
Suppose \(N\) is an inner model of \(V\). Suppose \(U_0\in V\) and \(U_1\in N\) are such that \(\textnormal{Ult}(V,U_0) = \textnormal{Ult}(N,U_1)\). Then \(j_{U_0}[\textnormal{Ord}]\subseteq j_{U_1}[\textnormal{Ord}]\).
\begin{proof}
Note that such an inner model \(N\) is definable from the parameters \(U_0\) and \(U_1\). This is because for any regular cardinal \(\delta > \textsc{sp}(U_1)\) and any transitive set \(S\) satisfying \(\text{ZFC}^-\) and containing \(U_1\), \(S = N \cap H_\delta\) if and only if \(\textnormal{Ult}(S,U_1) = (H_{\delta^*})^{\textnormal{Ult}(V,U_0)}\) where \(\delta^* =\sup j^S_{U_1}[\delta]\) and the ultrapower is computed using functions in \(S\). This is easily proved by induction on cardinals.

Assume towards a contradiction that \cref{defnfactor} fails. Let \(\alpha\) be the least ordinal such that for some ultrafilter \(U_0\) and set \(U_1\), the definition of \(N\) from \(U_0\) and \(U_1\) as indicated above succeeds in defining an inner model \(N\) of \(V\) relative to which \(U_1\) is an ultrafilter and such that \(\textnormal{Ult}(V,U_0) = \textnormal{Ult}(N,U_1)\) but \(j_{U_0}(\alpha) \notin j_{U_1}[N]\). Then \(\alpha\) is definable without parameters, and so \(j_{U_0}(\alpha)\) is definable without parameters in \(\textnormal{Ult}(V,U_0)\). Since \(j_{U_1}[N]\prec \textnormal{Ult}(V,U_0)\), \(\alpha\in j_{U_1}[N]\), which contradicts the definition of \(\alpha\).
\end{proof}
\end{lma}

It is not clear if it is necessary that \(j_{U_0}[V]\subseteq j_{U_1}[N]\) under the assumption of the Ultrapower Axiom. (Without the Ultrapower Axiom it is consistent with ZFC that \(j_{U_0}[V]\subseteq j_{U_1}[N]\) fail in the special case \(N = V\), but \cref{everythingcommutes} rules out this particular failure in the context of the Ultrapower Axiom.)

\begin{thm}[Ultrapower Axiom]
The factor ordering \(\leq_\mathcal D\) is the same as the order of reverse inclusion on \(\mathcal D\).
\begin{proof}
We must show that if \(N\subseteq M\) and \(N,M\in \mathcal D\), then \(M \leq_\mathcal D N\). Let \(P = N\vee M\) and let \(k_N:N\to P\) and \(k_M:M\to P\). By \cref{defnfactor} applied to \(M\), \(k_M[\text{Ord}]\subseteq k_N[\text{Ord}]\). In particular, if \(a_M\) is the seed of \(M\), then \(k_M(a_M) \in \text{ran}(k_N)\). Since \(k_N\) comes from the canonical comparison it follows that \(k_N\) is the identity and \(N = P\). Hence \(k_M: M \to N\) witnesses that \(M\leq_\mathcal D N\).
\end{proof}
\end{thm}

Many of the structural consequences of the Ultrapower Axiom for \(\mathcal D\) conceivably could be consequences of ZFC alone. For example, the question of whether an ultrafilter has only finitely many factors has been open for about 40 years. It is probably best to be clear about exactly what assumptions beyond ZFC we need to prove these consequences. Of course, the Ultrapower Axiom suffices. But let us state some weakened forms that work just as well. We will try to be brief.

\begin{defn}
The {\it Product Lemma} is the statement that \((U,W)\sE (U,W')\) implies \(\textnormal{Ult}(V,U) \vDash W \sE W'\).
\end{defn}

The converse of the Product Lemma is provable in ZFC:

\begin{lma}\label{productlemma}
Assume \(U\) is a uniform ultrafilter and \(W,W'\) are uniform ultrafilters of \(\textnormal{Ult}(V,U)\). An \(\textnormal{Ult}(V,U)\)-lopsided comparison \(\langle k,k'\rangle\) witnessing \(W\sE W'\) in \(\textnormal{Ult}(V,U)\) is a lopsided comparison of \(\langle(U,W), (U',W)\rangle\) witnessing \((U,W)\sE (U',W)\).
\end{lma}

\begin{cor}
The converse of the Product Lemma is true.
\end{cor}

We omit the proofs.

\begin{prp}
The Ultrapower Axiom implies the Product Lemma.
\begin{proof}
Suppose \((U,W)\sE (U,W')\). By the Ultrapower Axiom, \((U,W)\swo (U,W')\). In particular, \(W\neq W'\). Since the Ultrapower Axiom holds in \(\textnormal{Ult}(V,U)\), either \(\textnormal{Ult}(V,U) \vDash W \swo W'\) or \(\textnormal{Ult}(V,U) \vDash W' \swo W\). The latter cannot hold by \cref{productlemma} since it implies \((U,W')\swo (U,W)\), contradicting \cref{welldefined}.
\end{proof}
\end{prp}

\begin{defn}
The {\it Factor Lemma} is the statement that if \(U\) is a nonprincipal uniform ultrafilter and \(W\) is a uniform ultrafilter of \(\textnormal{Ult}(V,U)\) then letting \(Z = (U,W)\), \(\textnormal{Ult}(V,U)\vDash W \sE j_U(Z)\).
\end{defn}

The Factor Lemma is essentially the statement that \cref{decomplemma} holds, weakening its conclusion by replacing the seed order with the \(E\)-order.

\begin{prp}
The Product Lemma implies the Factor Lemma.
\begin{proof}
Since \((U,W) = Z\) and \(Z \sE U\times Z\), \((U,W) \sE U\times Z = (U,j_U(Z))\). By the Product Lemma, \(W \sE j_U(Z)\).
\end{proof}
\end{prp}

\begin{defn}
The {\it ACC Lemma} is the statement that \(\mathcal D\) satisfies the ascending chain condition below a point.
\end{defn}

\begin{prp}
The Factor Lemma implies the ACC Lemma.
\begin{proof}
The proof is exactly the same as \cref{acc}.
\end{proof}
\end{prp}

\begin{defn}
The {\it Local Finiteness Theorem} is the statement that a countably complete ultrafilter has only finitely many factors.
\end{defn}

\begin{defn}
The {\it Upper Bounds Lemma} is the statement that if two ultrapowers \(M_0\) and \(M_1\) have an upper bound \(N\) in \(\mathcal D\), then they have an upper bound \(P\leq N\) that is least among all upper bounds of \(M_0\) and \(M_1\) below \(N\).
\end{defn}

\begin{prp}
The ACC Lemma and the Upper Bounds Lemma imply the Local Finiteness Theorem.
\begin{proof}
This follows from the proof of \cref{locallyfinite}.
\end{proof}
\end{prp}

\begin{prp}\label{glbs}
The ACC Lemma and the Upper Bounds Lemma imply that any pair of ultrapowers has a greatest lower bound in \(\mathcal D\).
\begin{proof}
This is a purely order theoretic fact, proved as in \cref{meets}.
\end{proof}
\end{prp}

\begin{prp}
The ACC Lemma and the Upper Bounds Lemma imply that any pair of ultrapowers with an upper bound has a least upper bound.
\begin{proof}
This is a purely order theoretic fact. Suppose \(M_0\) and \(M_1\) have an upper bound \(N\). Let \(P\) be the least upper bound of \(M_0\) and \(M_1\) below \(N\), as given by the Upper Bounds Lemma. Let \(N'\) be any other upper bound of \(M_0\) and \(M_1\). Let \(Q\) be the greatest lower bound of \(P\) and \(N'\). Then \(M_0,M_1 \leq Q\) since \(M_0\) and \(M_1\) are lower bounds of \(P\) and \(N'\). Thus \(Q\) is an upper bound of \(M_0\) and \(M_1\), and so since \(Q\leq P \leq N\) and \(P\) is the least upper bound of \(M_0\) and \(M_1\) below \(N\), \(P \leq Q\). So \(P = Q\). Thus \(P \leq N'\). Since \(N'\) was arbitrary, \(P\) is the least upper bound of \(M_0\) and \(M_1\).
\end{proof}
\end{prp}

\begin{defn}
The {\it Irreducible Factorization Theorem} is the statement that every countably complete ultrafilter factors as a finite iteration of irreducible ultrafilters.
\end{defn}

\begin{prp}
The ACC Lemma implies the Irreducible Factorization Theorem.
\begin{proof}
Again this is just abstract order theory.
\end{proof}
\end{prp}

\subsection{The Limit of All Ultrapowers}\label{LimitSection}
Throughout this subsection we assume the Ultrapower Axiom. Consider \(\mathcal D\), the category of ultrapowers of \(V\) with morphisms the internal ultrapower embeddings. Assuming the Ultrapower Axiom, this forms a directed system. It is impossible to resist the temptation to take the direct limit.

Let \((N_\infty,E)\) be the direct limit of \(\mathcal D\). (Here \(E\) is the membership relation of \(N_\infty\).) For each \(M_0,M_1\in \mathcal D\), let \(k_{M_0,M_1}\) denote the unique internal ultrapower embedding from \(M_0\) to \(M_1\), if there is one. Also, for \(M\in \mathcal D\) denote by \(k_{M,N_\infty}\) the canonical direct limit embedding.  We denote \(k_{V,N_\infty}\) by \(k_\infty : V \to N_\infty\).

\begin{prp}
The relation \(N_\infty\) is wellfounded, but if there is a supercompact cardinal, then \(N_\infty\) is not setlike.
\begin{proof}
To show that \(N_\infty\) is wellfounded, it suffices to show that there exists no \(E\)-descending sequence \(k_{M_0,N_\infty}(x_0) \mathbin{E} k_{M_1,N_\infty}(x_1) \mathbin{E}  k_{M_2,N_\infty}(x_2)\mathbin{E} \cdots\) where \(E\) is the membership relation of \(N_\infty\). The point is that we may assume that \(M_0 <_\mathcal D M_1 <_\mathcal D M_2 <_\mathcal D \cdots\), by the usual argument, since \(\mathcal D\) is a directed system. But then the internal iteration \[V\stackrel{k_{V,M_0}}{\longrightarrow} M_0\stackrel{k_{M_0,M_1}}{\longrightarrow} M_{1}\stackrel{k_{M_1,M_2}}{\longrightarrow} M_{2}\stackrel{k_{M_2,M_3}}{\longrightarrow}\cdots\]
is illfounded, and this is a contradiction.

To see that if there is a supercompact cardinal then \(N_\infty\) is not setlike, note that for any ordinal \(\alpha\) and any minimal ultrafilter \(U\), \(j_U(\alpha)\) embeds in \(k_\infty(\alpha)\), via the map \(k_{\textnormal{Ult}(M,U),N_\infty}\). If \(\kappa\) is supercompact, then there exist ultrafilters \(U\) sending \(\kappa\) arbitrarily high. Thus arbitrarily large ordinals order embed in \(k_\infty(\kappa)\), so \(E|k_\infty(\kappa)\) does not form a set.
\end{proof}
\end{prp}

We identify the setlike part of \(N_\infty\) with its transitive collapse. We conjecture that in a fine structure model, the least ordinal \(\alpha\) such that \(k_\infty(\alpha)\) is not a set is precisely the least supercompact cardinal. We present some weak evidence for this conjecture. First we prove the following fact, which is quite obvious, but is what makes \(N\) interesting. 

\begin{prp}\label{definabilityoflimit}
For any \(M\in \mathcal D\), \((N_\infty)^{M} = N_\infty\) and \((k_\infty)^{M} = k_{M,N_\infty}\).
\end{prp}

We remark that since \(N\) is not setlike, one should really say \((N)^{M_U}\) is isomorphic to \(N\) via an isomorphism \(\pi: N^{M_U} \to N\) defined in \(V\) with the property that \(\pi\circ (k_\infty)^{N} = k_U\).

\begin{proof}[Proof of \cref{definabilityoflimit}]
The directed system \((\mathcal D)^{M}\) is cofinal in the directed system \(\mathcal D\), with the same maps, and this induces an isomorphism of the limit structures \((N_\infty)^{M}\) and \(N_\infty\).
\end{proof}

The following formula for \(k_\infty\) will be useful. Recall that \(\mathcal I\) is the class of minimal ultrafilters.
\begin{prp}
For any set \(x\), \(k_\infty(x) = \bigcup \{j_U(k_\infty[x]) : U\in \mathcal I\}\).
\end{prp}
We trust the reader to parse this proposition correctly. In particular, the union is meant in the sense of \(N\), but cannot really be computed within \(N\). It is literally true when \(k_\infty(x)\) is a subset of the setlike part of \(N_\infty\).
\begin{proof}
For any \(M\in \mathcal D\) and \(y\in M\), we have \(k_{M,N_\infty}(y) \in k_\infty(x)\) if and only if \(y\in k_{V,M}(x)\). Hence \[k_\infty(x) = \bigcup \{k_{M,N_\infty}[k_{V,M}(x)] : M\in \mathcal D\} = \bigcup \{k_{\textnormal{Ult}(M,U),N_\infty}[j_U(x)] : U\in \mathcal I\}\] But \(k_{\textnormal{Ult}(M,U),N_\infty}[j_U(x)] = j_U(k_\infty[x])\) by \cref{definabilityoflimit}. The proposition follows.
\end{proof}

\begin{prp}
Suppose that \(N_\infty\) is not setlike. Then there is an ordinal \(\lambda\) that is lifted arbitrarily high by ultrapower embeddings.
\begin{proof}
Let \(\kappa\) be the least ordinal such that \(k_\infty(\kappa)\) has a proper class of predecessors. Let \(\lambda = \sup k_\infty[\kappa]\) (recall that we identify the setlike part of \(N_\infty\) with its transitive collapse). Then \(\lambda\) is lifted arbitrarily high by ultrapower embeddings because \[k(\kappa) = \bigcup_{U\in \mathcal I} j_U(k_\infty[\kappa]) = \sup_{U\in\mathcal I} j_U(\sup k_\infty[\kappa]) = \sup_{U\in\mathcal I} j_U(\lambda)\]
Since \(k_\infty(\kappa)\) is not a set, \(\{j_U(\lambda): U\in \mathcal I\}\) is unbounded in the ordinals. 
\end{proof}
\end{prp}

We suspect that in a fine structure model, the least ordinal that can be lifted arbitrarily high by ultrapower embeddings is in fact supercompact, but in the context of the Ultrapower Axiom, we do not even know how to prove this ordinal is a cardinal. The least ordinal \(\kappa\) such that \(k_\infty(\kappa)\) has a proper class of predecessors is inaccessible, since it is easy to see that the class of predecessors of \(k_\infty(\kappa)\) is isomorphic to the class of all ordinals (so the inaccessibility of \(\kappa\) follows by replacement and powerset). Must \(\kappa\) be measurable? Must \(\kappa\) be lifted arbitrarily high by ultrapower embeddings? Equivalently, must \(\kappa = \sup k_\infty[\kappa]\)?

We now explore the relationship between \(N_\infty\) and the seed order, noting that the maps \(k_{M,N_\infty}: M\to N_\infty\) constitute a close comparison of all ultrapowers of \(V\) in the sense implicit in \cref{closetoultra}.

\begin{lma}
For any uniform ultrafilter \(U\), there is a finite sequence of \(N_\infty\)-ordinals \(a\) such that \(U\) is the ultrafilter derived from \(k_\infty\) using \(a\).
\begin{proof}
Let \(M = \textnormal{Ult}(V,U)\). The lemma follows immediately from the fact that \(k_\infty = k_{M,N_\infty}\circ j_U\), taking \(a = k_{M,N_\infty}([\text{id}]_U)\).
\end{proof}
\end{lma}

\begin{defn}
Suppose \(U\) is a uniform ultrafilter. We denote by \(a_U\) the least finite set of \(N_\infty\)-ordinals \(a\) such that \(U\) is the ultrafilter derived from \(k_\infty\) using \(a\).
\end{defn}

\begin{prp}
Suppose \(U\) is a uniform ultrafilter and let \(M = \textnormal{Ult}(V,U)\). Then \(a_U = k_{M,N_\infty}([\textnormal{id}]_U)\).
\begin{proof}
 This follows from the minimality of close embeddings. Suppose \(b\) is a sequence of \(N_\infty\)-ordinals such that \(U\) is derived from \(k_\infty\) using \(b\). Let \(i: M\to N_\infty\) be the factor map given by \(i(j_U(f)([\text{id}]_U)) = k_\infty(f)(b)\). Let \(\Gamma = j_U[V]\). Then \(k_{M,N_\infty}\restriction \Gamma = i\restriction \Gamma\) and \(M\) is finitely generated mod \(\Gamma\). We can therefore apply the Close Embeddings Lemma, \cref{cel}: \(k_{M,N_\infty}\) is close to \(M\), being definable over \(M\) by \cref{definabilityoflimit}, and so it follows that \(k_{M,N_\infty}([\text{id}]_U) \leq i([\text{id}]_U) = b\). 
 
Since  \(b\) was an arbitrary sequence of \(N_\infty\)-ordinals such that \(U\) is derived from \(k_\infty\) using \(b\), it follows that \(k_{M,N_\infty}([\text{id}]_U)\) is the least finite sequence of \(N_\infty\)-ordinals \(a\) such that \(U\) is the ultrafilter derived from \(k_\infty\) using \(a\). That is, \(k_{M,N_\infty}([\text{id}]_U) = a_U\).
\end{proof}
\end{prp}

Note, however, that there will be many other finite sets of \(N_\infty\)-ordinals with which \(U\) can be derived from \(k_\infty\). For example, suppose \(U\) is the unique normal measure on the least measurable cardinal \(\kappa\). Let \(M_0 = \textnormal{Ult}(V,U)\), let \(M_1 = \textnormal{Ult}(V,U\times U)\), and let \(\kappa_1 = j_U(\kappa)\). Then \[a_U = k_{M_0,N_\infty}(\kappa) = k_{M_1,N_\infty}(k_{M_0,M_1}(\kappa)) = k_{M_1,N_\infty}(\kappa) < k_{M_1,N_\infty}(\kappa_1)\]
But the ultrafilter derived from \(k_\infty\) using \(k_{M_1,N_\infty}(\kappa_1)\) is \(U\). In general, Rudin-Keisler reductions that are not internal ultrapower embeddings give rise to this kind of example.

\begin{cor}
Suppose \(U_0\) and \(U_1\) are uniform ultrafilters. Then \(U_0\wo U_1\) if and only if \(a_{U_0} \leq a_{U_1}\). 
\begin{proof}
Let \(M_0 = \textnormal{Ult}(V,U_0)\) and \(M_1 = \textnormal{Ult}(V,U_1)\). By \cref{closetoultra}, the close embeddings \(k_{M_0,N_\infty}:M_0\to N_\infty\) and \(k_{M_1,N_\infty}:M_{1}\to N_\infty\) suffice as a comparison of \(\langle U_0,U_1\rangle\), and they witness \(U_0\wo U_1\) since \(k_{M_0,N_\infty}([\text{id}]_{U_0}) = a_{U_0} \leq a_{U_1} = k_{M_1,N_\infty}([\text{id}]_{U_1})\).
\end{proof}
\end{cor}

We mention a related inner model that seems quite wild assuming very large cardinals. The model we have in mind is the intersection of all ultrapowers of \(V\).

\begin{prp}[Ultrapower Axiom]
The intersection \(P\) of all ultrapowers of \(V\) by countably complete ultrafilters is an inner model of \textnormal{ZF} containing the setlike part of \(N_\infty\).
\begin{proof}
Obviously \(P\) is a transitive class containing all the ordinals, and since the transitive collapse of the setlike part of \(N_\infty\) is contained in every ultrapower of \(V\) by \cref{definabilityoflimit}, it is contained in \(P\). The only real issue in verifying ZF is with the Axiom of Comprehension. The key point is that \(P\) is a definable subclass of every ultrapower of \(V\): this is because the Ultrapower Axiom implies that \((P)^{\textnormal{Ult}(V,U)} = P\), since it easily yields that the class of internal ultrapowers of \(\textnormal{Ult}(V,U)\) are cofinal in the partial order of ultrapowers of \(V\) ordered by reverse inclusion. Given this, it is easy to prove Comprehension. Suppose \(A\in P\), \(p\) is a parameter in \(P\), and \(\varphi(v_0,v_1)\) is a formula in the language of set theory. We must show \(\{x\in A: P\vDash \varphi(x,p)\}\in P\). But \(\{x\in A: P\vDash \varphi(x,p)\}\) is in \(\textnormal{Ult}(V,U)\) for any countably complete ultrafilter \(U\), since \(P\) is a definable subclass of \(\textnormal{Ult}(V,U)\) and comprehension holds in \(\textnormal{Ult}(V,U)\).
\end{proof}
\end{prp}

We note that \(P\) is closed under \(\kappa\)-sequences where \(\kappa\) is the least measurable cardinal. Does \(P\) satisfy the Axiom of Choice in general? In \(L[U]\), \(P\) is the Prikry extension of the \(\omega\)th iterate obtained by adjoining the critical sequence.

\section{Weak Comparison}\label{comparisonsection}
The point of this section is to provide evidence for our assertion that the Ultrapower Axiom holds in canonical inner models. In order to avoid committing ourselves to any particular choice of fine structure, we prove the Ultrapower Axiom from Woodin's axiom Weak Comparison (see \cref{axcompultra}). One can also prove the Ultrapower Axiom in more specific contexts with fewer auxiliary assumptions. (In particular one can in many circumstances avoid the assumption that \(V = \text{HOD}\): it is enough that every element of the model be ordinal definable with access to the predicate from which the model is constructed. Also, one can avoid using the \(\Sigma_2\) Closure Axiom using a bit of fine structure.) 

We use the following convention in our statement of Weak Comparison below, and throughout \cref{comparisonsection}: 

\begin{conv} A model of ZFC is just a set \(X\), not necessarily transitive, such that \((X,\in)\vDash \textnormal{ZFC}\).\end{conv}

\begin{defn}
The axiom of {\it Weak Comparison} is the following few sentences. First, \(V = \text{HOD}\). Second, suppose \(X_0,X_1\prec_{\Sigma_2} V\) are finitely generated models of ZFC (see \cref{fingen}). Let \(M_0\) be the transitive collapse of \(X_0\), and \(M_1\) the transitive collapse of \(X_1\). Assume \[\mathbb R\cap M_0 = \mathbb R\cap M_1\] Then there is a transitive set \(N\) admitting close embeddings \(k_0:M_0\to N\) and \(k_1: M_1\to N\).
\end{defn}

We briefly explain why Weak Comparison should hold in canonical inner models with a comparison process like the one that exists in the known models. The assumption \(V = \text{HOD}\) does not actually always hold in the fine structure models, but it simplifies things here, and it is a reasonable requirement since it is a formal consequence of Woodin's axiom \(V = \text{Ultimate }L\). We now try to justify the comparison principle itself. We keep this justification somewhat vague to show that it does not make many assumptions about the fine structure of the model in question, but we note that in particular the justification can be turned into a proof of a version of Weak Comparison that holds in the Mitchell-Steel models as well as the models at the finite levels of supercompactness using comparison by disagreement. We assume the reader has some familiarity with the comparison process.

The assumption that \(X_0,X_1\prec_{\Sigma_2} V\) should ensure that the transitive collapses \(M_0,M_1\) are iterable, using a copying construction to embed iterates of \(M_0,M_1\) into iterates of \(V\). That \(X_0,X_1\prec_{\Sigma_2} V\) should also ensure that \(M_0\) and \(M_1\) fall into the fine structural hierarchy used to construct \(V\). It should be possible to compare any two models in this hierarchy, and so we compare \(M_0\) and \(M_1\). Since \(M_0\) and \(M_1\) satisfy ZFC, if neither side of the comparison drops, their comparison yields a transitive set \(N\) admitting close embeddings \(k_0:M_0\to N\) and \(k_1: M_1\to N\), and thus one verifies Weak Comparison. The assumptions on \(X_0\) and \(X_1\), namely that \(X_0\) and \(X_1\) are finitely generated and \(\mathbb R\cap M_0 = \mathbb R\cap M_1\), are meant to guarantee that the comparison cannot drop. We explain this briefly. Suppose for example that the \(M_0\)-side of the comparison drops. Then the \(M_1\)-side does not drop, and the last model on the \(M_1\)-side is a proper initial segment of the last model \(\mathcal Q\) on the \(M_0\)-side. Since \(M_1\) is finitely generated, \(M_1\) is coded by a real in \(\mathcal Q\). By a backwards induction along the \(M_0\)-side of the comparison, one proves that there is a real coding \(M_1\) in \(M_0\), and this contradicts the assumption that \(\mathbb R\cap M_0 = \mathbb R\cap M_1\).

We spend the rest of this section proving that Weak Comparison implies the Ultrapower Axiom assuming that there is a proper class of strong cardinals. In fact we will make do with a much weaker large cardinal axiom than a proper class of strong cardinals.

\begin{defn}[\(\Sigma_2\) Closure Axiom]
For all sets \(x\), there is a \(\Sigma_2\) elementary substructure of \(V\) satisfying ZFC and containing \(x\) as an element.
\end{defn}

Our next proposition is an immediate consequence of the well-known fact that if \(\kappa\) is a strong cardinal then \(V_\kappa\prec_{\Sigma_2} V\).

\begin{prp}
Assume there is a proper class of strong cardinals. Then the \(\Sigma_2\) Closure Axiom holds.
\end{prp}

In fact, we do not need anything as strong as a strong cardinal. For example, Morse-Kelley set theory implies the \(\Sigma_2\) Closure Axiom.

\begin{prp}[\(\Sigma_2\) Closure Axiom]\label{axcompultra}
Weak Comparison implies the Ultrapower Axiom.
\end{prp}

The proof requires two easy lemmas.

\begin{lma}[\(\Sigma_2\) Closure Axiom + \(V = \text{HOD}\)]\label{finitelygenerated}
For all sets \(x\), there is a finitely generated \(\Sigma_2\)-elementary substructure \(Y\) of \(V\) that satisfies \textnormal{ZFC} and contains \(x\) as an element.
\begin{proof}
Fix \(x\). By the \(\Sigma_2\) Closure Axiom, there is a set \(X\prec_{\Sigma_2} V\) such that \(x\in X\) and \(X\vDash \textnormal{ZFC}\). Since the statement \(V = \textnormal{HOD}\) is \(\Pi_3\), it is downwards absolute to \(\Sigma_2\)-elementary substructures, and so \[X\vDash V = \textnormal{HOD}\] Therefore \(X\) has definable Skolem functions. Thus the set \[Y = \{y\in X: y\text{ is definable in \(X\) from the parameter }x\}\] is an elementary substructure of \(X\). It is clear that \(Y\) is finitely generated and contains \(x\). Since \(Y\prec X\prec_{\Sigma_2} V\), \(Y\prec_{\Sigma_2} V\). Since \(Y\prec X\), \(Y\) satisfies ZFC. Thus \(Y\) is as desired.
\end{proof}
\end{lma}

Our second lemma is a variation on Los's theorem.

\begin{lma}\label{ultragen} Suppose \(U\) is a countably complete ultrafilter on a set \(S\). Then for any finitely generated model \(X\prec_{\Sigma_2} V\) containing \(U\), there is a finitely generated model \(Y\prec_{\Sigma_2} V\) that is isomorphic to \(\textnormal{Ult}(X,U)\) where the ultrapower is computed in \(X\).
\begin{proof}
Fix such a model \(X\). Let \(a\) be an element of \(\bigcap (U \cap X)\), the intersection of all \(U\)-large sets that belong to \(X\). Note that \(\bigcap (U\cap X)\) is nonempty because \(X\) is countable and \(U\) is countably complete. Let \[Y = \{f(a) : f\in X\text{ and } \text{dom}(f) = S\}\]

We claim \(Y\) is a \(\Sigma_2\) elementary substructure of \(V\). For this, it suffices to show that \(C\cap Y\neq \emptyset\) for any nonempty class \(C = \{v : \varphi(v,y)\}\) defined in \(V\) from the parameter \(y\in Y\) by the \(\Sigma_2\) formula \(\varphi\). 

Let \(M\) be the transitive collapse of \(X\), and let \(\pi:M\to V\) be the \(\Sigma_2\) elementary embedding given by the inverse of the collapse. Fix \(f\in X\) such that \(y = f(a)\), and let \(\bar f\in M\) be such that \(\pi(\bar f) = f\). Let \(\bar S\in M\) be such that \(\pi(\bar S) = S\). Let \[B = \{w\in S: \exists v\ \varphi(v,f(w))\}\] and let \[\bar B = \{w\in \bar S: M\vDash \exists v\ \varphi(v,\bar f(w))\}\] We have \(\pi(\bar B) =B\) since \(\pi\) is \(\Sigma_2\) elementary and \(\bar B\) is defined by the same Boolean combination of \(\Sigma_2\) formulas that defines \(B\) in \(V\) using the corresponding parameters. (More precisely, since \(\forall w\in \bar B\ \exists v\ \varphi(v,\bar f(w))\) and \(\forall w\in S(\exists v\ \varphi(v,\bar f(w))\to w\in \bar B)\), these formulas hold of \(\pi(\bar B)\) with \(f\) replacing \(\bar f\), and so \(\pi(\bar B) = B\).) Note also that since \(C\) is nonempty, there is some \(v\) such that \(\varphi(v,f(a))\); in other words \(a\in B\).

Since \(M\) satisfies ZFC, there is a function \(\bar g\in M\) such that \(\text{dom}(\bar g) = \bar B\) and \[M\vDash \forall w\in \bar B\ \varphi(\bar g(w),\bar f (w))\] Again this is equivalent to a \(\Sigma_2\) formula, and so letting \(g = \pi(\bar g)\in X\), we have \[\forall w\in B\ \varphi(g(w),f(w))\] Since \(a\in B\), \(a\in \text{dom}(g)\) and \(\varphi(g(a),f(a))\). Thus \(g(a)\in C\cap Y\), which shows \(C\cap Y\neq \emptyset\), as desired.
\end{proof}
\end{lma}

\begin{proof}[Proof of \cref{axcompultra}]
Assume Weak Comparison and suppose \(F_0\) and \(F_1\) are countably complete ultrafilters. By \cref{finitelygenerated}, we may fix a finitely generated \(\Sigma_2\) substructure  \(X\) of \(V\) satisfying ZFC and containing \(F_0\) and \(F_1\) as elements. Let \(M\) be the transitive collapse of \(X\) and \(U_0\) and \(U_1\) the images of \(F_0\) and \(F_1\) under the collapse. By \cref{ultragen}, the ultrapowers \(\textnormal{Ult}(M,U_0)\) and \(\textnormal{Ult}(M, U_1)\) can each be \(\Sigma_2\)-elementarily embedded into \(V\), and so by Weak Comparison, there is a transitive set \(N\) that admits close embeddings \(k_0:\textnormal{Ult}(M, U_0)\to N\) and \(k_1:\textnormal{Ult}(M,U_1)\to N\). 

By the Close Embeddings Lemma (\cref{cel}) applied to the finitely generated model \(M\), \(k_0\circ j_{U_0} \restriction \text{Ord} = k_1\circ j_{U_1}\restriction \text{Ord}\). Since \(M\vDash V = \text{HOD}\), this implies \(k_0\circ j_{U_0} = k_1\circ j_{U_1}\). Now by \cref{closetoultra}, in \(M\) there is a comparison of \(\langle U_0, U_1\rangle\). 

Since the inverse of the transitive collapse of \(X\) is a \(\Sigma_2\)-elementary embedding of \(M\) into \(V\) sending \(\langle U_0,U_1\rangle\) to \(\langle F_0,F_1\rangle\), and since \(M\) satisfies the \(\Sigma_2\) statement that there is a comparison of \(\langle U_0,U_1\rangle\) (see \cref{sigma2comp}), there is a comparison of \(\langle F_0,F_1\rangle\). Thus the Ultrapower Axiom holds.
\end{proof}

\section{An independence result}\label{IndependenceSection}
Intuitively, it is clear that the Ultrapower Axiom is more powerful than the mere linearity of the Mitchell order on normal ultrafilters, a statement we will refer to as \(\textnormal{UA}_\text{normal}\). (Note that the linearity of the Mitchell order on normal ultrafilters is equivalent to the restriction of the Ultrapower Axiom to normal ultrafilters.) Finding a model that separates \(\textnormal{UA}_\text{normal}\) from the Ultrapower Axiom is a little subtle: as we mentioned in the introduction, the only known method for obtaining instances of \(\textnormal{UA}_\text{normal}\) involves the theory of inner models, and in the inner models, the Ultrapower Axiom holds. 

In this section we show that if it is consistent that there is a measurable cardinal, then it is consistent that \(\textnormal{UA}_\text{normal}\) holds but the Ultrapower Axiom fails. In fact we show something stronger, and of independent interest:

\begin{thm}\label{lu}
If \(U\) is a normal measure on a measurable cardinal \(\kappa\), there is a forcing extension \(N\) of \(L[U]\) with a unique measurable cardinal \(\kappa\), a unique normal ultrafilter \(W\) on \(\kappa\), and a \(\kappa\)-complete ultrafilter \(Z\) on \(\kappa\) that is not Rudin-Keisler equivalent to a finite iterated product of \(W\). 
\end{thm}

\begin{cor}
If the theory \(\textnormal{ZFC} + \textnormal{there is a measurable cardinal}\) is consistent, then \(\textnormal{ZFC} + \textnormal{UA}_\textnormal{normal}\) does not prove the Ultrapower Axiom.
\begin{proof}
Clearly \(\textnormal{UA}_\text{normal}\) holds in any model of ZFC with a single measurable cardinal that carries a unique normal ultrafilter. On the other hand, the Ultrapower Axiom implies that every \(\kappa\)-complete ultrafilter on the least measurable cardinal \(\kappa\) is Rudin-Keisler equivalent to an iterated product of the the unique normal ultrafilter on \(\kappa\). Thus assuming the consistency of a measurable cardinal, \cref{lu} yields a model in which \(\textnormal{UA}_\text{normal}\) holds while the Ultrapower Axiom fails. 
\end{proof}
\end{cor}

The technique we use to prove \cref{lu} is a very minor twist on the Friedman-Magidor iterated forcing with \(*\)-perfect trees \cite{FriedmanMagidor}.

Fix \(U\), a normal ultrafilter on \(\kappa\) in \(L[U]\). Let \(j:L[U]\to L[U']\) be the ultrapower embedding, with \(U'\) the unique normal ultrafilter on \(\kappa'\) in \(L[U']\). In \(L[U]\), we define \(\mathbb P\) to be the nonstationary support iteration of length \(\kappa+1\) with \(\mathbb P(\alpha) = \text{Sacks}^*(\alpha) * \text{Code}(\alpha)\). (The forcing \(\text{Sacks}^*(\alpha)\) is defined on page 11 of \cite{FriedmanMagidor}, and \(\text{Code}(\alpha)\) is defined on page 4 in the very similar context of \(\text{Sacks}(\alpha)\)-forcing.) We remark, in preparation for the statement of \cref{FMlemma}, that \(j(\mathbb P)({\leq}\kappa) = \mathbb P\).

\begin{lma}\label{kuneniter}
If \(M\) is an inner model such that every \(\kappa\)-sequence in \(M\) is dominated by a \(\kappa\)-sequence of \(L[U]\) and \(W\) is a normal ultrafilter on \(\kappa\) in \(M\), then \(j_W:M \to \textnormal{Ult}(M,W)\) lifts \(j_U: L[U]\to \textnormal{Ult}(L[U],U)\).
\end{lma}

\begin{lma}\label{FMlemma}
Suppose \(G\subseteq \mathbb P({<}\kappa)\) is \(L[U]\)-generic and \(t\subseteq \textnormal{Sacks}^*(\kappa)_{G}\) is \(L[U][G]\)-generic. Given any \(L[U'][G][t]\)-generic \(g\subseteq \textnormal{Code}(\kappa)_{G*t}\), there is a unique \(L[U']\)-generic \(G'\subseteq j(\mathbb P({<}\kappa))\) projecting to \(G*t*g\) with \(j[G]\subseteq G'\). Let \(j^*: L[U][G]\to L[U'][G']\) be the lift of \(j\). Then there is a unique \(L[U'][G']\)-generic filter \(t'\subseteq\textnormal{Sacks}^*(\kappa')_{G'}\) such that \(j^*[t]\subseteq t'\).

Finally, assume \(\kappa^{+L[U]}\) is a regular cardinal and \(L[U][G][t][g]\) is correct about stationary subsets of \(\kappa^{+L[U]}\). Then \(t*g\) is the unique \(L[U'][G]\)-generic filter on \((\textnormal{Sacks}^*(\kappa) * \textnormal{Code}(\kappa))_{G}\).
\end{lma}

In particular, if \(\kappa^{+L[U]}\) is a regular cardinal and \(L[U][G][t][g]\) is correct about stationary subsets of \(\kappa^{+L[U]}\), there is a unique lift of \(j\) to an elementary embedding of \(L[U][G][t][g]\): its existence follows from the existence of \(t'\) and the \({\leq}\kappa\)-distributivity of \(\text{Code}(\kappa)_{G*t}\), and its uniqueness follows from the fact that any lift is uniquely determined by the choice of an \(L[U'][G]\)-generic on \((\textnormal{Sacks}^*(\kappa) * \textnormal{Code}(\kappa))_{G}\).

Given \cref{kuneniter} and \cref{FMlemma}, we can complete the proof of \cref{lu}. (Although one can easily prove the uniqueness of the normal ultrafilter in the Friedman-Magidor extension using these lemmas, we will skip this step since it is spelled out in \cite{FriedmanMagidor}.)

\begin{proof}[Proof of \cref{lu}]
Fix \(U_0\), a normal ultrafilter on \(\kappa_0\) in \(L[U_0]\). Let \(j_{01}: L[U_0]\to L[U_1]\) and \(j_{12}: L[U_1]\to L[U_2]\) be the ultrapower embeddings, with \(U_1\) and \(U_2\) the unique normal ultrafilters on \(\kappa_1\) and \(\kappa_2\) in \(L[U_1]\) and \(L[U_2]\) respectively. 

Let \(L[U_0][G_0][t_0][g_0]\) be a forcing extension of \(L[U_0]\) by \(\mathbb P\). By this we mean that \(G_0\subseteq \mathbb P({<}\kappa_0)\) is \(L[U_0]\)-generic, \(t_0\subseteq \text{Sacks}^*(\kappa_0)_{G_0}\) is \(L[U_0][G_0]\)-generic, and \(g_0\subseteq \text{Code}(\kappa_0)_{G_0*t_0}\) is \(L[U_0][G_0][t_0]\)-generic. Let \(W\) denote the unique normal ultrafilter on \(\kappa_0\) in \(L[U_0][G_0][t_0][g_0]\).

The model \(N\) is obtained by forcing over \(L[U_0][G_0][t_0][g_0]\) with \[\mathbb Q = j_W(\textnormal{Code}(\kappa_0)_{G_0*t_0}) = \textnormal{Code}(\kappa_1)_{G_1*t_1}\] Note that \(\mathbb Q\) is \(\kappa_1\)-closed in the model \(L[U_1][G_1][t_1]\), which is closed under \(\kappa_0\)-sequences in \(L[U_0][G_0][t_0][g_0]\), and hence \(\mathbb Q\) is a \({\leq}\kappa_0\)-closed forcing of size \(\kappa_0^+\) in \(L[U_0][G_0][t_0][g_0]\). (We remark that this implies that in \(L[U_0][G_0][t_0]\), \(\mathbb Q \simeq \text{Add}(\kappa_0^+,1)\).) Hence \(\mathbb Q\) preserves cofinalities, adds no new \(\kappa_0\)-sequences, and preserves stationary subsets of \(\kappa_0^{+L[U_0]}\). Thus let \(h\subseteq \mathbb Q\) be \(L[U_0][G_0][t_0][g]\) generic, and let \(N = L[U_0][G_0][t_0][g_0][h]\). 

We first show that in \(N\), there is a unique normal ultrafilter on \(\kappa_0\). Existence is easy: since \(\mathbb Q\) adds no new \(\kappa_0\)-sequences, \(W\) is a normal ultrafilter on \(\kappa_0\) in \(N\). For uniqueness, let \(W'\) be any normal ultrafilter on \(\kappa_0\) in \(N\), and we will show \(W' = W\). Every \(\kappa_0\)-sequence in \(N\) is dominated by a \(\kappa_0\)-sequence in \(L[U_0]\), since this is true of \(L[U_0][G_0][t_0][g_0]\) by the proof of Lemma 6 in \cite{FriedmanMagidor}, and hence the ultrapower embedding \(j_{W'}: N\to \text{Ult}(N,W')\) lifts the ultrapower embedding \(j_{01}: L[U_0]\to L[U_1]\). Since \(\kappa_0^{+L[U_0]}\) is a regular cardinal of \(N\) and \(L[U_0][G_0][t_0][g_0]\) is correct about stationary subsets of \(\kappa_0^{+L[U_0]}\) in \(N\), there is at most one lift of \(j_{01}\) to an embedding of \(L[U_0][G_0][t_0][g_0]\) that is definable in \(N\) by the remarks following \cref{FMlemma}. The lift of \(j_{01}\) must then be the ultrapower embedding by \(W\). It follows that \(j_{W'}\) lifts \(j_{W}|L[U_0][G_0][t_0][g_0]\), and so since \(N\) has no new \(\kappa\)-sequences, \(W' = W\).

Finally, we show that in \(N\), there is a \(\kappa_0\)-complete ultrafilter on \(\kappa_0\) that is not an iterated product of \(W\). Let \(j_{12}: L[U_1]\to L[U_2]\) be the ultrapower embedding. We apply \cref{FMlemma} to \(L[U_1]\), now with \(U = U_1\) and \(U' = U_2\), but still working in \(N\). But now we are free to take \(g' = h\) in that lemma, since \(h\subseteq \textnormal{Code}(\kappa_1)_{G_1*t_1} = \mathbb Q\) is \(L[U_2][G_1][t_1]\)-generic, being \(L[U_0][G_0][t_0][g_0]\)-generic. Thus we obtain an elementary embedding \(j_{12}^**: L[U_1][G_1][t_1]\to L[U_2][G_2][t_2]\) such that \(G_2\) projects to \(G_1 * t_1 * h\). Moreover, by distributivity considerations, \(j_{12}^**\) lifts to an embedding \(k: \text{Ult}(N,W)\to N'\): note that \(\text{Ult}(N,W)\) is generic over \(L[U_1][G_1][t_1]\) by the \({\leq}\kappa_1\)-distributive forcing \(\textnormal{Code}(\kappa_1)_{G_1*t_1}\) followed by the \({<}\kappa_2\)-distributive forcing \(j_W(\mathbb Q)\). The model \(N'\), however, is not contained in \(\text{Ult}(N,W)\): \(\text{Ult}(N,W)\) has the same \({<}\kappa_2\)-sequences as \(L[U_1][G_1][t_1][g_1]\), and hence does not contain \(h\) since \(t_1*g_1\) is the unique \((\textnormal{Sacks}^*(\kappa_1) * \textnormal{Code}(\kappa_1))_{G_1}\)-generic over \(L[U_1][G_1]\) in \(L[U_1][G_1][t_1][g_1]\). The elementary embedding \(k\circ j_W\) is the ultrapower embedding of \(N\) by a \(\kappa_0\)-complete ultrafilter \(Z\) on \(\kappa_0\) since it lifts \(j_{02}\). But \(Z\) is not Rudin-Keisler equivalent to an iterated product \(F\) of \(W\), since the ultrapower by such an \(F\) is contained in \(\text{Ult}(N,W)\).
\end{proof}

We note the following curious corollary of the proof of \cref{lu}.

\begin{cor}
If it is consistent that there is a measurable cardinal, then it is consistent that there is a measurable cardinal \(\kappa\) carrying a unique normal ultrafilter and that \(\textnormal{Add}(\kappa^+,1)\) adds a \(\kappa\)-complete ultrafilter on \(\kappa\) without adding a normal ultrafilter.
\end{cor}

It is not clear that it is {\it possible} for \(\textnormal{Add}(\kappa^+,1)\) to add a normal ultrafilter on a measurable cardinal \(\kappa\) that carries a unique normal ultrafilter in the ground model. Over \(L[U]\), for example, \({\leq}\kappa\)-distributive forcing adds no new \(\kappa\)-complete ultrafilters on \(\kappa\). On the other hand, if \(\kappa\) is indestructibly supercompact, and \(G\subseteq \textnormal{Add}(\kappa^+,1)\) is \(V\)-generic, there is a new normal ultrafilter on \(\kappa\) in \(V[G]\): since \(\kappa\) is \(2^\kappa\)-supercompact in \(V[G]\), there is in \(V[G]\) a normal ultrafilter \(W\) on \(\kappa\) such that \(G\in \text{Ult}(V[G],W)\), and it is easy to see that no such \(W\) can lie in \(V\). Another curious corollary:

\begin{cor}
Assume it is consistent that there is a measurable cardinal. Then it is consistent that there is a measurable cardinal \(\kappa\) carrying a unique normal ultrafilter and \(2^{2^\kappa}\) distinct \(\kappa\)-complete ultrafilters.
\begin{proof}
In the proof of \cref{lu}, instead of forcing over \(L[U_0][G_0][t_0][g_0]\) with \(\mathbb Q\), force with the \({\leq}\kappa\)-support product of \(\mathbb Q\) with itself. Equivalently, instead of \(\text{Add}(\kappa^+,1)\), use \(\text{Add}(\kappa^+,\kappa^{++})\).
\end{proof}
\end{cor}

The model of \cref{lu} also separates the seed order from the \(\E\)-order:

\begin{cor}
Let \(N\) be the model of \cref{lu}, and let \(Z\) be as in the proof of \cref{lu}. In \(N\), \(W\E Z\) but \(W\) and \(Z\) are \(\wo\)-incomparable.
\begin{proof}
By the proof of \cref{muthm}, if \(\langle W,Z\rangle\) admits a comparison by internal ultrafilters, then either \(W\) is an internal factor of \(Z\) or else \(W\mo Z\). The former is impossible since \(\text{Ult}(V,Z)\) is not contained in \(\text{Ult}(V,W)\), and the second is impossible since there is only one measurable cardinal in \(N\). Thus \(W\) and \(Z\) are \(\wo\)-incomparable. On the other hand, \(W\leq_{\text{RK}} Z\) by construction, and so \(W\E Z\) by \cref{rkE}.
\end{proof}
\end{cor}

The ultrafilter \(Z\) is the only example we know of a minimal irreducible ultrafilter that is not Dodd solid.

\section{Questions}\label{QuestionSection}
We conclude with some questions. We begin with the most obvious one:

\begin{qst}\label{supercompactua}
Is the Ultrapower Axiom consistent with a supercompact cardinal?
\end{qst}

It seems unlikely that this will be answered positively without a solution to the Inner Model Problem at the level of supercompact cardinals. \cref{comparisonsection} argues that if a solution to the Inner Model Problem is within reach of current technology, then this solution yields a positive answer to \cref{supercompactua}. But this is somewhat speculative.

Our next questions regards the analysis in \cref{RealmSection}, which admittedly does not reach very far in large cardinal terms. The question of how detailed an inductive analysis of countably complete ultrafilters the Ultrapower Axiom provides is made precise by the following question:
\begin{qst}
Assume the Ultrapower Axiom. Does the Mitchell order wellorder the class of minimal irreducible ultrafilters?
\end{qst}

It seems unlikely that the Ultrapower Axiom proves that minimal irreducible ultrafilters are Dodd solid, though this would give a positive answer to this question by \cref{doddlin}. In fact, if the Ultrapower Axiom proves minimal irreducible ultrafilters are Dodd solid in general, one has a positive answer to \cref{gch} as well, by \cref{superirred} and \cref{doddgch}. It seems more likely that it is consistent that GCH fails at \(\kappa\) for \(\kappa\) a supercompact while the Ultrapower Axiom holds, which implies that there are minimal irreducible ultrafilters on \(\kappa^+\) that are not Dodd solid. Perhaps there is a weakening of Dodd solidity that makes sense for ultrafilters on \(\delta\) when \(2^{<\delta}> \delta\), and which suffices for the proof of \cref{doddlin}. The natural inductive proof that minimal irreducible ultrafilters are Dodd solid (see \cref{doddfail1}) seems to break down when one needs to analyze objects that are not finitely generated; that is, when one reaches ultrafilters whose extender initial segments are not themselves ultrafilters.

There are many other questions one can ask about the algebra of irreducible ultrafilters. Are the ultrafilters of the canonical comparison of a pair of irreducible ultrafilters irreducible themselves in the models to which they are applied? Suppose \(U_0\) is an irreducible ultrafilter and \(U_0\) is a factor of \(U_1\vee U_2\). Must \(U_0\) be a factor of either \(U_1\) or of \(U_2\)? That is, are irreducible ultrafilters prime? If so, this generalizes Euclid's lemma, a first step towards a {\it unique} factorization theorem assuming the Ultrapower Axiom.

Our last question is whether it is possible to generalize the Ultrapower Axiom in a way that applies to a wider class of elementary embeddings. For example, we make the following definition.

\begin{defn}
The {\it Extender Axiom} is the statement that if \(E_0\) and \(E_1\) are extenders then there exist extenders \(F_0\in \textnormal{Ult}(V,E_0)\) and \(F_1\in \textnormal{Ult}(V,E_1)\) such that \(\textnormal{Ult}(\textnormal{Ult}(V,E_0),F_0) = \textnormal{Ult}(\textnormal{Ult}(V,E_1),F_1)\).
\end{defn}

We have no reason (in the style of \cref{axcompultra}) to believe that the Extender Axiom is consistent with very large cardinals.

\begin{conj}
Assume there is a cardinal \(\kappa\) that is \(2^\kappa\)-supercompact. Then the Extender Axiom is false.
\end{conj}
\newpage
\bibliography{seedorder}{}
\bibliographystyle{unsrt}
\end{document}